\documentclass{article}
\usepackage[utf8]{inputenc}
\usepackage{amssymb}
\usepackage{amsmath}
\usepackage{hyperref}
\usepackage{amsthm}
\usepackage{mathrsfs}
\usepackage{mathtools}
\usepackage{tikz}
\usepackage{bbm}
\usepackage{hyperref}
\usepackage{tikz-cd}
\usepackage{comment}
\usepackage{hieroglf}
\usepackage{fullpage}
\usepackage{todonotes}
\usepackage[title]{appendix}

\newcommand{\Q}{\mathbb{Q}}
\newcommand{\Z}{\mathbb{Z}}
\newcommand{\F}{\mathbb{F}}
\newcommand{\Qp}{\Q_{p}}
\newcommand{\Zp}{\Z_{p}}

\newcommand{\R}{\mathbb{R}}

\newcommand{\C}{\mathbb{C}}
\newcommand{\A}{\mathbb{A}}

\newcommand{\norm}[1]{\left\lVert#1\right\rVert}
\DeclareFontFamily{U}{rcjhbltx}{}
\DeclareFontShape{U}{rcjhbltx}{m}{n}{<->rcjhbltx}{}
\DeclareSymbolFont{hebrewletters}{U}{rcjhbltx}{m}{n}

\DeclareMathSymbol{\kaf}{\mathord}{hebrewletters}{107}

\newtheorem{theorem}{Theorem}[section]
\newtheorem{lemma}[theorem]{Lemma}

\newtheorem{proposition}[theorem]{Proposition}
\newtheorem{corollary}[theorem]{Corollary}
\theoremstyle{definition}

\newtheorem{definition}[theorem]{Definition}
\newtheorem{remark}[theorem]{Remark}

\title{Spherical varieties and non-ordinary families of cohomology classes}
\author{by Rob Rockwood}
\date{}

\begin{document}

\maketitle

\begin{abstract}
We show that $p$-adic families of cohomology classes associated to symmetric spaces vary $p$-adically over small discs in weight space, without any ordinarity assumption. This generalises previous work of Loeffler, Zerbes and the author. Furthermore, we show that these families exhibit full variation in the cyclotomic direction, generalising previous constructions of Euler systems and $p$-adic $L$-functions. As an application we show that the Lemma--Flach Euler system of Loeffler--Skinner--Zerbes interpolates in Coleman families.
\end{abstract}

\tableofcontents

\section{Introduction}

This article is a sequel to the articles \cite{loefflerspherical} and \cite{loeffler2021spherical}. In the first of these articles Loeffler constructs a very general machine for constructing norm-compatible families of cohomology classes associated to pairs of reductive groups satisfying an `open orbit' condition. Such pairs are said to be \textit{spherical}. Loeffler's machine subsumes the constructions of many well-known norm-compatible classes, including the Euler systems of Heegner Points and Beilinson--Flach classes. In the second listed article, Loeffler, Zerbes and the author show that the ordinary projections of the classes constructed in \cite{loefflerspherical} interpolate $p$-adically  as the weight varies. 

In this paper we show that the classes constructed in \cite{loefflerspherical} continue to exhibit $p$-adic variation if we relax the ordinarity hypothesis and assume only that our families have finite slope and satisfy a suitably defined `non-criticality' condition. Our methods differ substantially from the ordinary case studied in \cite{loeffler2021spherical}; since ordinary classes are fundamentally integral the methods used to study them are algebraic in nature. Since we allow our classes to exhibit unbounded $p$-adic growth our methods take on a more $p$-adic analytic flavour. 

A somewhat novel aspect of our work is that we work with cohomology with coefficients in (sheaves associated to) locally analytic function modules as opposed to the usual modules of distributions, an idea first appearing in \cite{greenberg2020triple}. Using modules of function reduces the problem of constructing branching maps to the construction of analytic functions invariant under a particular group, which turns out to be not so difficult. Our main innovation in using these modules pertains to the observation that when working with \'etale cohomology the usual modules of locally analytic functions are not profinite and therefore their \'etale cohomology groups have poor finiteness properties. We instead construct spaces of `Iwasawa analytic' functions: whereas locally analytic functions are represented by power series on `closed' discs of some fixed radius, our `Iwasawa' functions are represented on `open' discs. We show that these modules are profinite and that the union of these spaces over all radii coincides with the union of the locally analytic spaces.  

We give a brief outline of the contents of the paper:

\begin{itemize}
    \item In \S 2 we collect together the objects and notations we will use for the rest of the paper, in particular a spherical pair of reductive groups $(G, H)$, their associated locally symmetric spaces $Y_{G}, Y_{H}$ and their cohomology. We introduce the modules of analytic and Iwasawa functions and their associated sheaves which will serves as our coefficients. 
    \item In \S 3 given a spherical pair $(G, H)$ we construct pushforward maps from the cohomology of $Y_{H}$ to that of $Y_{G}$, interpolating the classical pushforwards defined in \cite[Definition 4.3.4]{loeffler2021spherical}.
    \item In \S 4 we collect together facts about complexes of Banach modules and slope decompositions that we will need in the ensuing chapters.
    \item In \S 5 we construct `Abel--Jacobi maps' from \'etale to Galois cohomology which interpolate the `classical weight' Abel-- Jacobi maps obtained as edge maps in the Hochschild--Serre spectral sequence and used in, for example, \cite{KLZ}, \cite{LZGsp}, to construct norm-compatible classes in Galois cohomology. By composing these maps with the pushforwards constructed in \S 3 we obtain pushforward maps into Galois cohomology. 
    \item In \S 6 we show how one can extend the range of interpolation of the branching maps constructed in \S 3 by modifying the spherical datum. 
    \item In \S 7 we show that the branching maps constructed in \S 3 satisfy certain congruences. Using these congruences, we construct a map into distribution valued Galois cohomology, interpolating the Galois pushforwards we constructed in \S5 over all cyclotomic twists. 
    \item In \S8 we use the pushforwards constructed in \S 7 to interpolate the classes of the Lemma--Flach Euler system, constructed in \cite{LZGsp}, in Coleman families. 
    \begin{theorem}
    Set $p > 4$ and let $\mathcal{F}$ be a Coleman family of genus $2$ Siegel modular forms over a sufficiently small affinoid $U$ of weight space admitting a specialisation with generic $L$-parameter in the sense of \cite[Definition 1.1]{yang2025generic}. Then there is an $\mathcal{O}(U)$-linear Galois representation $W_{\mathcal{F}}$ and a cohomology class
    $$
        z_{\mathcal{F}} \in H^{1}(\Q, D^{\ell a}(\Gamma, W_{\mathcal{F}}))
    $$
    whose specialisations at all specialisations of $\mathcal{F}$ and all characters of $\Gamma = \mathrm{Gal}(\Q(\zeta_{p^{\infty}})/\Q)$ coincide with the Lemma--Eisenstein classes of \cite{LZGsp} of these specialisations and characters, wherever these are defined. 
    \end{theorem}
    See \S \ref{sec:lemfam} for a more precise statement. 
    \begin{remark}
        The classes constructed in \cite{LZGsp} are only constructed for ordinary $p$-stabilisations of general type automorphic representations of $\mathrm{GSp}_{4}$ however the construction of the classes generalises immediately to the finite-slope case using the Abel--Jacobi maps constructed in \S \ref{sec:abeljac}, assuming only a weaker non-criticality condition in place of ordinarity. It is these classes which we interpolate.
    \end{remark}
 \end{itemize}

 \section*{Acknowledgements}
I would like to thank David Loeffler for suggesting this problem to me and for many helpful conversations and correspondences. I would also like to thank Chris Williams, Andrew Graham and Marco Seveso for helpful conversations. I am funded by the Heilbronn Institute for Mathematical Research.

\section{Setup}
\label{sec:setup}
We fix the notation used throughout the paper. 
Fix a prime $p$. We recall the following setting from \cite{loeffler2021spherical}:
\begin{itemize}
    \item Let $\mathcal{G}$ be a connected reductive group over $\mathbb{Q}$ satisfying Milne's axiom (SV5), that is, the torus $\mathcal{Z} = Z(\mathcal{G})^{\circ}$ (the identity component of the centre of $\mathcal{G}$) contains no $\mathbb{R}$-split torus which is not $\mathbb{Q}$-split \footnote{As noted in \cite{loeffler2021spherical} this condition is for convenience and can be relaxed with some fastidious bookkeeping.} . 
    \item $G$ is a reductive group-scheme over $\mathbb{Z}_{p}$ whose base-extension to $\mathbb{Q}_{p}$ coincides with that of $\mathcal{G}$. In particular $G$ is quasi-split and thus admits a Borel subgroup $B_{G} \subset G$. 
    \item $Q_{G}$ is a choice of standard parabolic subgroup of $G$ and $\overline{Q}_{G}$ is the opposite parabolic, so that $L_{G} = Q_{G} \cap \overline{Q}_{G}$ is a Levi subgroup of $Q_{G}$ and the big Bruhat cell $U_{\mathrm{Bru}}^{G} = \overline{N}_{G} \times L_{G} \times N_{G}$ is an open subscheme of $G$ over $\mathbb{Z}_{p}$, where $N_{G}$ is the unipotent radical of $Q_{G}$ and $\overline{N}_{G}$ is its opposite.
    
    \item Let $S_{G}$ denote the torus $L_{G}/L_{G}^{\mathrm{der}}$ with character lattice $X^{\bullet}(S_{G})$ and let $X^{\bullet}_{+}(S_{G})$ denote the $Q_{G}$-dominant weights. Let $C_{G} = G/G^{\mathrm{der}}$ denote the maximal torus quotient of $G$. 
    \item We choose a subtorus $S_{G}^{0} \subset S_{G}$ and let $L_{G}^{0}$ and $Q_{G}^{0}$ be its preimages under the quotient maps $L_{G} \to S_{G}$ and $Q_{G} \to S_{G}$ respectively. 
    \item Let $J_{G} \subset G(\mathbb{Z}_{p})$ be the parahoric subgroup associated to $Q_{G}$. This group admits an Iwahori decomposition
    $$
        J_{G} = (J_{G} \cap \overline{N}_{G}(\mathbb{Z}_{p})) \times L_{G}(\mathbb{Z}_{p}) \times N_{G}(\mathbb{Z}_{p})
    $$
    \item Let $\Phi_{G}$ denote the roots of $G$, $\Delta_{G}$ the simple roots and $\Phi^{+}_{G}$ the positive roots relative to the Borel subgroup $B_{G}$ and let $\Phi_{L_{G}}, \Delta_{L_{G}}, \Phi^{+}_{L_{G}}$ be the corresponding root sets for $L_{G}$ relative to the Borel $B_{L_{G}} = B_{G} \cap L_{G}$. 
\end{itemize}
Let $\mathcal{H}$ be another reductive $\Q$-group with reductive $\Zp$-model $H$. By systematically replacing $G$ with $H$ in the notation above we define analogous groups $B_{H}, Q_{H}, \overline{Q}_{H}, L_{H}, \ldots$ for $H$. From Section \ref{sec:branchlaw} onwards we consider a closed immersion $\mathcal{H} \to \mathcal{G}$ and choose $S_{H}^{0}$ so that $Q_{H}^{0}$ has an open orbit on the flag variety $G/\overline{Q}_{G}$, satisfying further compatibilities. 
\begin{itemize}
    \item Let $A$ be the maximal $\mathbb{Q}_{p}$-split torus in the centre of $L_{G}$ and define submonoids:
    \begin{align*}
         A^{-} &= \{a \in A(\Qp): v_{p}(\alpha(a)) \geq 0 \ \forall \ \alpha \in \Phi^{+}_{G} \backslash \Phi^{+}_{L_{G}}\} \\
         A^{--} &=  \{a \in A(\Qp): v_{p}(\alpha(a)) > 0 \ \forall \ \alpha \in \Phi^{+}_{G} \backslash \Phi^{+}_{L_{G}}\}.
    \end{align*}
    Note that the positive roots occurring in $\Phi^{+}_{G} \backslash \Phi^{+}_{L_{G}}$ are precisely those corresponding to the root spaces occurring in the unipotent radical $N_{G}$. It follows that for any $a \in A^{-}$ we have $aN_{G}(\Zp)a^{-1} \subset N_{G}(\Zp)$ and for any $a \in A^{--}$ we have $aN_{G}(\Zp)a^{-1} \equiv \ 1 \ \mathrm{mod} \ p$. 
    \item Fix $\tau \in A^{--}$ and for $n \geq 0$ define $N_{n} = \tau^{n} N_{G}(\Zp)\tau^{-n}, \overline{N}_{n} = \tau^{-n} \overline{N}_{G}(\Zp) \tau^{n}$ and $L_{n} = \{\ell \in L_{G}(\Zp): \ell \in L^{0}_{G}(\Z/p^{n}\Z) \ \mathrm{mod} \ p^{n}\}$. For $n\geq 1$ define open-compact subgroups of $G(\Zp)$:
    $$
        U_{n} = \overline{N}_{0} \times L_{n} \times N_{n}, \ \ \
        V_{n} = \tau^{-n}U_{n}\tau^{n}
    $$
\end{itemize}

Set $d_{G} = \mathrm{dim}N_{G}$.  
\subsection{Algebraic representations}
Let $K/\mathbb{Q}_{p}$ be a finite unramified extension over which $G$ splits and let $\mathcal{O}$ be the ring of integers of $K$. For $\lambda \in X_{+}^{\bullet}(S_{G})$ we define 
$$
    V^{G}_{\lambda} = \{f \in K[G]: f(\overline{n}\ell g) = \lambda(\ell)f(g) \ \forall \overline{n} \in \overline{N}_{G}, \ell \in L_{G}, g \in G\}
$$
with $g$ acting by right-translation. By the Borel--Weil--Bott theorem this is the irreducible $G$-representation of highest weight $\lambda$ with respect to a Borel subgroup $B_{G} \subset Q_{G}$.  We let $f_{\lambda}$ denote the unique choice of highest weight vector satisfying $f_{\lambda}(\overline{n}\ell n) = \lambda(\ell)$ for $\overline{n}\ell n \in U_{\mathrm{Bru}}^{G}$. We also write 
$$
    \mathcal{P}_{\lambda}^{G} = \{f \in K[G]: f(n\ell g) = \lambda^{-1}(\ell)f(g)  \ \forall n \in N_{G}, \ell \in L_{G}, g \in G\},
$$
so that $(\mathcal{P}^{G}_{\lambda})^{\vee} \cong V^{G}_{\lambda}$ as $G$-representations. A canonical choice of highest weight vector of $(\mathcal{P}^{G}_{\lambda})^{\vee}$ with respect to $B_{G}$ is given by the functional $\delta_{1}: f \mapsto f(1)$. 
\subsection{Integral lattices}
\begin{definition}
An \textit{admissible lattice} in $V^{G}_{\lambda}$ is an $\mathcal{O}$-lattice $\mathcal{L} \subset V^{G}_{\lambda}$ invariant under $G_{/\mathcal{O}}$ and whose intersection with the highest weight subspace is $\mathcal{O} \cdot f_{\lambda}$. 
\end{definition}
Admissible lattices have the following properties:
\begin{itemize}
    \item There are only finitely many admissible lattices in a given $V^{G}_{\lambda}$ with a unique maximal and minimal lattice. 
    \item Every admissible lattice in $V^{G}_{\lambda}$ is a direct sum of its intersections with the weight spaces of $V^{G}_{\lambda}$.
    \item Taking the dual of an admissible lattice in $V^{G}_{\lambda}$ gives an admissible lattice in $\left(V^{G}_{\lambda}\right)^{\vee}$. The dual of the maximal lattice in $V^{G}_{\lambda}$ is the minimal lattice in $V^{G, \vee}_{\lambda}$ and vice-versa. 
\end{itemize}

Let $V_{\lambda, \mathcal{O}}$ denote the maximal admissible lattice in $V^{G}_{\lambda}$, given in the Borel--Weil--Bott presentation as 
$$
    V_{\lambda, \mathcal{O}}= \{f \in \mathcal{O}[G]: f(\overline{n}\ell g) = \lambda(\ell)f(g) \ \forall \ \overline{n} \in \overline{N}_{G}, \ell \in L_{G}, g \in G\}.
$$
Similarly defining $\mathcal{P}^{G}_{\lambda, \mathcal{O}}$, then $(\mathcal{P}^{G}_{\lambda, \mathcal{O}})^{\vee}$ is isomorphic to the minimal admissible lattice in $V^{G}_{\lambda}$. 

\subsection{Cohomology of locally symmetric spaces} \label{sec:coho}

As in \cite{loeffler2021spherical} we fix a neat prime-to-$p$ level group $K^{p}$ and for an open compact subgroup $U \subset G(\mathbb{Z}_{p})$ let $Y_{G}(U)$ denote the locally symmetric space of level $K^{p}U$. We consider the following cohomology theories:
\begin{itemize}
    \item Betti cohomology of the locally symmetric space $Y_{G}(U)$ viewed as a real manifold, with coefficients in a locally constant sheaf $\mathscr{F}$, denoted for $i \geq 0$ by
    $$
    H^{i}(Y_{G}(U), \mathscr{F}).
    $$
\end{itemize}
Suppose now that $\mathcal{G}$ admits a Shimura datum with reflex field $E$. 
\begin{itemize}
    \item \'Etale cohomology of $\overline{Y}_{G}(U) := Y_{G}(U)_{\overline{\mathbb{Q}}}$ with coefficients in a lisse \'etale sheaf $\mathscr{F}$, denoted for $i \geq 0$ by
    $$
    H^{i}(\overline{Y}_{G}(U), \mathscr{F}).
    $$
    \item Assume further that $\mathcal{G}$ admits an abelian type Shimura datum and that $\mathcal{Z}$ splits over a CM field. Then \cite[Theorem 2.2.1]{lovering2017integral} says that if $\Sigma$ is a  finite set of primes containing those dividing $p$ and at which $K^{p}$ is hyperspecial, the Shimura variety $Y_{G}$ admits a canonical integral model $Y_{G}(U)_{\Sigma}$ over $\mathcal{O}_{E}[\Sigma^{-1}]$. We consider the \'etale cohomology of $Y_{G}(U)_{\Sigma}$ with coefficients in $\mathscr{F}$, denoted for $i \geq 0$ by
    $$
    H^{i}(Y_{G}(U)_{\Sigma}, \mathscr{F}).
    $$
\end{itemize}

For any open compact subgroups $V \subset U \subset G(\Zp)$ let 
$$
\mathrm{pr}^{U}_{V}: Y_{G}(V) \to Y_{G}(U)
$$
denote the natural degeneracy map. 
If $A$ is a representation of an open compact subgroup $U \subset G(\Zp)$ then we can consider $A$ as a locally constant sheaf $\mathscr{A}_{V}$ on $Y_{G}(V)$ for any $V \subset U$ open-compact, and if $A$ is profinite with $U$ acting continuously and $G$ admits a Shimura datum we can consider $A$ as a lisse \'etale sheaf (which we also denote by $\mathscr{A}_{V}$) on the variety $Y_{G}(V)$. These sheaves satisfy the compatibility
$$
    \left(\mathrm{pr}^{V}_{V'}\right)^{*}\mathscr{A}_{V} = \mathscr{A}_{V'}
$$
for any open-compact $V' \subset V$ and thus we have natural pushforward maps 
$$
    \left(\mathrm{pr}^{V}_{V'}\right)_{*}: H^{i}(Y_{G}(V'), \mathscr{A}_{V'}) \to H^{i}(Y_{G}(V), \mathscr{A}_{V}).
$$
\begin{definition}
Let $A$ be a representation of an open-compact subgroup $U \subset G(\Zp)$ with associated family of locally constant sheaves $\{\mathscr{A}_{V}: \text{$V \subset U$ open compact}\}$. Let $X \subset G(\Zp)$ be \textit{any} subgroup (not necessarily open-compact). We then define \textit{Iwasawa cohomology} of $X$ to be
$$
   H^{i}_{\mathrm{Iw}}(X, \mathscr{A}) :=  \varprojlim_{V \supset X}H^{i}(Y_{G}(V), \mathscr{A}_{V}),
$$
where the limit is taken over open-compact subgroups $V\subset U$ containing $X$.
\end{definition}

From now on we will drop the subscript $V$ for sheaves $\mathscr{A}_{V}$ on $Y_{G}(V)$ associated to representations $A$ of open-compact subgroups $V \subset G(\Zp)$.
\subsection{Branching laws for algebraic representations} \label{sec:branchlaw}
We now work in the situation of \cite{loefflerspherical} and \cite{loeffler2021spherical} and consider an embedding $\iota: \mathcal{H} \to \mathcal{G}$ of reductive $\mathbb{Q}$-groups, extending to an embedding 
$$
\iota: H \to G
$$
of reductive group schemes over $\mathbb{Z}_{p}$. We choose data as in Section \ref{sec:setup} for both $H$ and $G$. We in particular note that we require no compatibility between the choices of parabolic subgroups $Q_{H}, Q_{G}$ other than those stated below.

Denote by $\mathcal{F} := \overline{Q}_{G} \backslash G$ the flag variety associated to the parabolic $\overline{Q}_{G}$. We assume that there is $u \in \mathcal{F}(\mathbb{Z}_{p})$ satisfying:
\begin{itemize}
    \item[(A)] The $Q_{H}^{0}$-orbit of $u$ is Zariski open in $\mathcal{F}$,
    \item[(B)] The image of $\overline{Q}_{G} \cap uQ_{H}^{0}u^{-1}$ under the projection $\overline{Q}_{G} \to  S_{G}$ is contained in $S^{0}_{G}$. 
\end{itemize}

Under the above assumptions, the space $U_{\mathrm{Sph}} = \overline{Q}_{G}uQ_{H}^{0}$ is a Zariski open subset of $G$. We refer to it (and its image in the flag variety) as the \textit{spherical cell}. 
\begin{remark}
Note that since the flag variety is connected, $U_{\mathrm{Bru}}^{G} \cap U_{\mathrm{sph}} \neq \emptyset$ so we can always take $u \in U_{\mathrm{Bru}}^{G}(\mathbb{Z}_{p})$ and in particular we can take $u \in N_{G}(\mathbb{Z}_{p})$, which we assume from now on.
\end{remark}

By \cite[Proposition 3.2.1]{loeffler2021spherical}, for $\lambda \in X^{\bullet}_{+}(S_{G})$ the space $\left(V_{\lambda}^{G}\right)^{Q_{H}^{0}}$ has dimension $\leq 1$. 
\begin{definition} \label{def:admiss}
We call weights satisfying $\mathrm{dim}\left(V_{\lambda}^{G}\right)^{Q_{H}^{0}} = 1$ $Q_{H}^{0}$-\textit{admissible weights} and denote the cone of such weights by $X_{+}^{\bullet}(S_{G})^{Q_{H}^{0}}$.
\end{definition}
For $\lambda \in X_{+}^{\bullet}(S_{G})^{Q_{H}^{0}}$ the space $\left(V_{\lambda}^{G}\right)^{Q_{H}^{0}}$ is spanned by the polynomial function $f^{\mathrm{sph}}_{\lambda} \in K[G]$ which we normalise by setting $f^{\mathrm{sph}}_{\lambda}(u) = 1$. The $H$-orbit of $f^{\mathrm{sph}}_{\lambda}$ generates an irreducible representation of $H$ of some highest weight $\mu \in X_{+}^{\bullet}(S_{H})$ with respect to a Borel $B_{H} \subset Q_{H}$.

\begin{proposition}
Let $\lambda \in X^{\bullet}_{+}(S_{G})^{Q_{H}^{0}}$  and let $\mu \in X_{+}^{\bullet}(S_{H})$ be the associated weight of $H$. Then there is a unique $H$-equivariant map 
$$
    \mathrm{br}^{\lambda}_{\mu}: (\mathcal{P}^{H}_{\mu, \mathcal{O}})^{\vee} \to V_{\lambda, \mathcal{O}}
$$
sending $\delta_{1}$ to $f^{\mathrm{sph}}_{\lambda}$. 
\end{proposition}
\begin{proof}
    By the preceding paragraph we obtain an $H$-equivariant map 
    $$
        \mathrm{br}_{\mu}^{\lambda}: \mathcal{P}^{H, \vee}_{\mu} \to V^{G}_{\lambda}\vert_{H}
    $$
    uniquely defined by sending $\delta_{1} \mapsto f_{\lambda}^{\mathrm{sph}}$. Since $V_{\lambda, \mathcal{O}}$ is the maximal admissible lattice in $V^{G}_{\lambda}$ then the restriction of $\mathrm{br}_{\mu}^{\lambda}$ to $(\mathcal{P}^{H}_{\mu, \mathcal{O}})^{\vee}$ must land in $V_{\lambda, \mathcal{O}}$ and since $\delta_{1} \in (\mathcal{P}^{H}_{\mu, \mathcal{O}})^{\vee}$ and $f_{\lambda}^{\mathrm{sph}} \in V_{\lambda, \mathcal{O}}$ we conclude by \cite[Proposition 3.2.6]{loeffler2021spherical}. 
\end{proof}     
Since the lattices $(\mathcal{P}^{H}_{\mu, \mathcal{O}})^{\vee} $ and $V_{\lambda, \mathcal{O}}$ are finite-free $\mathcal{O}$-modules they are profinite and so for any open compact subgroups $K_{H} \subset \mathcal{H}(\A_{f}), K_{G} \subset \mathcal{G}(\A_{f})$ acting on  $(\mathcal{P}^{H}_{\mu, \mathcal{O}})^{\vee} $ and $V_{\lambda, \mathcal{O}}$ via projection to $H(\Qp)$ and $G(\Qp)$ respectively we can associate respective lisse \'etale sheaves $\mathscr{P}^{H, \vee}_{\mu, \mathcal{O}}$ and $\mathscr{V}_{\lambda, \mathcal{O}}$ over $Y_{H}(K_{H}), Y_{G}(K_{G})$. If we further assume that $K_{H} \subset \iota^{-1}\left(K_{G}\right)$ then we have a finite map 
$$
    \iota: Y_{H}(K_{H}) \to Y_{G}(K_{G}).
$$
Letting $c := \mathrm{dim}_{\R}Y_{G} - \mathrm{dim}_{\R}Y_{H} - \mathrm{rk}_{\R}\left(\frac{Z_{H}}{Z_{G} \cap H}\right)$, where $\mathrm{rk}_{\R}$ denotes the split rank over $\R$, we have pushforward maps (c.f. \cite[2.1.2]{KLZ})
$$
    \iota_{*}: H^{i}_{\mathrm{Iw}}(Y_{H}(Q_{H}^{0} \cap u^{-1}U_{n}u), \mathcal{F}) \to H^{i + c}_{\mathrm{Iw}}(Y_{G}(Q_{H}^{0} \cap u^{-1}U_{n}u), \mathcal{F}(c)). 
$$
for any lisse \'etale sheaf $\mathcal{F}$. For any $i \geq 0$ we define a map
\begin{align*}
    \iota_{\mu}^{\lambda}(K_{H}, K_{G}): H^{i}(Y_{H}(K_{H}), \mathscr{P}^{H, \vee}_{\mu, \mathcal{O}}) &\xrightarrow{\mathrm{br}_{\mu, *}^{\lambda}} H^{i + c}(Y_{H}(K_{H}), \iota^{*}\mathscr{V}_{\lambda, \mathcal{O}}) \\
    &\xrightarrow{\iota_{*}} H^{i + c}(Y_{G}(K_{G}), \mathscr{V}_{\lambda, \mathcal{O}}(c)),
\end{align*}
which we refer to as an `algebraic branching map'. We will mostly be interested in the case where $K_{H} = K_{H, p}K_{H}^{p}, K_{G} = K_{G, p}K_{G}^{p}$ where $K_{H, p} = Q_{H}^{0} \cap u^{-1}U_{n}u \subset H(\Zp), K_{G, p} = u^{-1}U_{n}u \subset G(\Zp)$ and $K_{H}^{p} \subset \mathcal{H}(\A_{f}^{(p)}), K_{G}^{p} \subset \mathcal{G}(\A_{f}^{(p)})$ are open compact subgroups satisfying $K_{H}^{p} \subset \iota^{-1}\left(K_{G}^{p}\right)$. For this family of subgroups we will abuse notation and just refer to the branching maps as $\iota^{\lambda}_{\mu}$, dropping the level groups from the notation. 
\subsection{Weight spaces}
Write $\mathfrak{S}_{G} = L_{G}(\mathbb{Z}_{p})/L_{G}^{0}(\mathbb{Z}_{p}) =  S_{G}(\mathbb{Z}_{p})/S_{G}^{0}(\mathbb{Z}_{p}) \subset (S_{G}/S_{G}^{0})(\mathbb{Z}_{p})$. The torus $\mathfrak{S}_{G}$ splits into a direct product $\mathfrak{S}_{G} = \mathfrak{S}_{G}^{\mathrm{tor}} \times \mathfrak{S}_{G,1}$ with the logarithm map identifying $\mathfrak{S}_{G, 1} \cong \mathbb{Z}_{p}^{n_{G}}$ for some integer $n_{G}$ and $\mathfrak{S}_{G}^{\mathrm{tor}}$ of finite prime-to-$p$ order.
\begin{definition}
Write $\mathcal{W}_{G}$ for the rigid analytic space over $\mathbb{Q}_{p}$ parameterising continuous characters of $\mathfrak{S}_{G}$. 
\end{definition}

\begin{definition}
For $i = 1, \ldots, n_{G}$ let $s_{i} \in \mathfrak{S}_{G,1}$ be a $\mathbb{Z}_{p}$-basis. For an integer $m \geq 0$ we denote by $\mathcal{W}_{m} \subset \mathcal{W}_{G}$ the wide-open subspace consisting of weights $\lambda$ satisfying 
$$
    v_{p}(\lambda(s_{i}) - 1) > \frac{1}{p^{m}(p - 1)}
$$
for all $i$. 
\end{definition}

Fix a finite extension $L/\mathbb{Q}_{p}$ and for some $m \geq 0$ let $\mathcal{U} \subset \mathcal{W}_{m}$ be a wide-open disc defined over $L$. Let $\Lambda_{\mathcal{U}} \cong \mathcal{O}_{L}[[t_{1}, \ldots, t_{n_{G}}]]$ denote the $\mathcal{O}_{L}$-algebra of bounded-by-1 rigid functions on $\mathcal{U}$ with $\mathfrak{m}_{\mathcal{U}}$ its maximal ideal. Define the universal character for the torus $\mathfrak{S}_{G}$
\begin{align*}
    k_{\mathrm{univ}}^{G}: \mathfrak{S}_{G} &\to \Lambda(\mathfrak{S}_{G})^{\times} \\
    s &\mapsto [s]
\end{align*}
where $\Lambda(\mathfrak{S}_{G})$ is the Iwasawa algebra of $\mathfrak{S}_{G}$ and is canonically isomorphic to the bounded-by-1 global sections of $\mathcal{W}_{G}$.

Let $\mathcal{U} \subset \mathcal{W}_{G}$ be a wide open disc. Define 
$$
    k_{\mathcal{U}}^{G}: \mathfrak{S}_{G} \to \Lambda_{\mathcal{U}}^{\times}
$$
to be the character given by composing $k_{\mathrm{univ}}^{G}$ with restriction to $\mathcal{U}$. 
\begin{lemma} \label{lem:analchar}
If $\mathcal{U} \subset \mathcal{W}_{m}$ then the character $k_{\mathcal{U}}^{G}$ is $m$-analytic on $\mathfrak{S}_{G}$, viewed as a disjoint union of copies of $\mathbb{Z}_{p}^{n_{G}}$ indexed by $\mathfrak{S}_{G}^{\mathrm{tor}}$. 
\end{lemma}
\begin{proof}
The proof follows the method of \cite[Lemma 4.1.5]{recoleman}. As in \textit{op.cit.} it suffices take $\mathcal{U} = \mathcal{W}_{m}$. Let $\{s_{i}\}_{i}$ be a basis for the free $\Zp$-module $\mathfrak{S}_{G, 1}$. We identify 
$$
    \Lambda(\mathfrak{S}_{G}) \cong \mathcal{O}[\mathfrak{S}_{G}^{\mathrm{tor}}][[T_{1}, \ldots, T_{n_{G}}]]
$$
in such a way that $k_{\mathrm{univ}}^{G}(s_{i}) = 1 + T_{i}$ and similarly we identify
$$
    \Lambda_{G}(\mathcal{W}_{m}) \cong \mathcal{O}[\mathfrak{S}_{G}^{\mathrm{tor}}][[u_{1}, \ldots, u_{n_{G}}]]
$$
such that $T_{i} = \varepsilon_{i}u_{i}$ where $\varepsilon_{i} \in \mathcal{O}$ satisfies $v_{p}(\varepsilon_{i}) > \frac{1}{p^{m}(p - 1)}$ for each $i = 1, \ldots, n_{G}$. Thus $k_{\mathcal{U}}^{G}(s_{i}) = 1 + \varepsilon_{i}u_{i}$ and so for $s = \sum a_{i}s_{i} \in \mathfrak{S}_{G, 1}$, $a_{i} \in \Zp$, we have 
$$
    k_{\mathcal{U}}^{G}(s) = \prod_{i = 1}^{n_{G}}\left(\sum_{j_{i} = 1}^{\infty} {a_{i}\choose j_{i}}(\varepsilon u_{i})^{j_{i}}\right)
$$
and we are done since ${a_{i}\choose j_{i}}\varepsilon^{j_{i}}$ is locally $m$-analytic on $\Zp$ for all $j_{i}$. 
\end{proof}

\subsection{Hecke algebras}
Let $S$ be a finite set of primes and let $K^{Sp} = \prod_{\ell \notin S \cup \{p\}}G(\Z_{\ell}) \subset \mathcal{G}(\mathbb{A}_{f}^{S \cup \{p\}})$. Define 
$$
    \mathbb{T}_{S} := \mathcal{C}^{\infty}_{c}(K^{Sp} \backslash \mathcal{G}(\mathbb{A}^{S \cup \{p\}}_{f})/K^{Sp}, \mathbb{Z}_{p})
$$
the space of $\mathbb{Z}_{p}$-valued compactly supported locally constant $K^{Sp}$-biinvariant functions on $\mathcal{G}(\mathbb{A}_{f}^{S \cup \{p\}})$. It's well known that this is a commutative $\mathbb{Z}_{p}$-algebra under convolution.

Define 
$$
    \mathfrak{U}_{p}^{-} := \mathcal{C}_{c}^{\infty}(J_{G} \backslash J_{G}A^{-}J_{G}/J_{G},\Zp),
$$
the algebra locally constant, compactly supported, $\Zp$-valued, $J_{G}$-biinvariant functions on $J_{G}A^{-}J_{G}$, and similarly $\mathfrak{U}_{p}^{--}$. These are both commutative algebras under the convolution operation. Indeed, there is a $\mathbb{Z}_{p}$-algebra isomorphism 
$$
    \mathbb{Z}_{p}[A^{-}/A(\mathbb{Z}_{p})] \cong \mathfrak{U}^{-}_{p}. 
$$
We let $U'_{p}$ denote the image of the element $\tau \in A^{--}$ chosen in \S\ref{sec:setup} under this isomorphism. 
\begin{definition}
Define the unramified $Q_{G}$-parahoric Hecke algebra:
$$
    \mathbb{T}_{S, p}^{-}:= \mathbb{T}_{S} \otimes_{\Z}\mathfrak{U}^{-}_{p}.
$$
\end{definition}

\subsection{Locally analytic function spaces}

Let $\mathcal{U} \subset \mathcal{W}_{G}$ be a wide-open disc and let $\mathfrak{m}_{\mathcal{U}} 
$ denote the maximal ideal of $\Lambda_{\mathcal{U}}$. We allow the case that $\mathcal{U} = \{\lambda\}$
for some $\lambda \in X^{\bullet}(S_{G})$ in which case the character $k_{\mathcal{U}}^{G}$ is just $\lambda$.
\begin{definition}
 For $m \geq 0$ define 
$$
    \mathrm{LA}_{m}(\mathbb{Z}_{p}^{d}, \Lambda_{\mathcal{U}}) = \{f: \mathbb{Z}_{p}^{d} \to \Lambda_{\mathcal{U}}:  \forall \ \underline{a} \in \mathbb{Z}_{p}^{d}, \exists f_{\underline{a}} \in \Lambda_{\mathcal{U}}\langle T_{1} \ldots, T_{d}\rangle  \ \text{s.t.}\  f(\underline{a} + p^{m}\underline{x}) = f_{\underline{a}}(\underline{x}) \ \forall \underline{x} \in \mathbb{Z}_{p}^{d}\}.
$$
This space is isomorphic to $\prod_{\underline{a}}\Lambda_{\mathcal{U}}\langle p^{-m}T_{1}, \ldots, p^{-m}T_{d} \rangle$ as a $\Lambda_{\mathcal{U}}$-module, where $\underline{a}$ runs over $(\mathbb{Z}/p^{m}\mathbb{Z})^{d}$. 
\end{definition}

 The logarithm map gives an isomorphism $N_{G}(\mathbb{Z}_{p}) \cong (\mathrm{Lie}N_{G})(\mathbb{Z}_{p}) \cong \mathbb{Z}_{p}^{d_{G}}$ which we use to identify $\mathrm{LA}_{m}(\Z^{d_{G}}_{p}, \Lambda_{\mathcal{U}})$ with a space of functions on $N_{G}(\Zp)$.
\begin{definition}
For $m\geq 0$, define 
$$
    A^{\mathrm{an}}_{\mathcal{U}, m} = \{f: U^{G}_{\mathrm{Bru}}(\mathbb{Z}_{p}) \to \Lambda_{\mathcal{U}}: f(\overline{n}t n) = k_{\mathcal{U}}^{G}(t)f(n), \ \text{and} \  f \vert_{N_{G}(\mathbb{Z}_{p})} \in  \mathrm{LA}_{m}(\mathbb{Z}_{p}^{d_{G}}, \Lambda_{\mathcal{U}})\}
$$
equipped with the $\mathfrak{m}_{\mathcal{U}}$-adic topology. In the case $\mathcal{U} = \{\lambda\}$, we write $A^{\mathrm{an}}_{\lambda, m}$ in place of $A^{\mathrm{an}}_{\mathcal{U}, m}$.
\end{definition}
Restriction to $N_{G}(\mathbb{Z}_{p})$ gives a $\Lambda_{\mathcal{U}}$-module isomorphism $A^{\mathrm{an}}_{\mathcal{U}, m} \cong  \mathrm{LA}_{m}(\mathbb{Z}_{p}^{d_{G}}, \Lambda_{\mathcal{U}})$ with inverse $f \mapsto \left( \overline{n}tn \mapsto k_{\mathcal{U}}^{G}(t)f(n) \right)$. We give these spaces an action of $a \in A^{-}$ via
$$
    (a \cdot f)(\overline{n}\ell n) = f(\overline{n}\ell a n a^{-1}). 
$$
The following proposition is well-known and the proof is very similar to that of Proposition \ref{prop:invar} (which is not well-known) so we omit it. 
\begin{proposition}
Suppose $\mathcal{U} \subset \mathcal{W}_{m}$ for some $m \geq 0$. The modules $A^{\mathrm{an}}_{\mathcal{U}, m}$ are preserved by both the right translation action of $J_{G}$ and the action of $A^{-}$ described above.
\end{proposition}
The action of $A^{-}$ on $A^{\mathrm{an}}_{\mathcal{U}, m}$ thus extends to an action of $J_{G}A^{-}J_{G}$. Adapting the proof of \cite[3.2.8]{urbaneigen} to our situation we see that for $a \in A^{--}$
\begin{equation} \label{eq:anal}
    aA^{\mathrm{an}}_{\mathcal{U}, m + 1} \subset A^{\mathrm{an}}_{\mathcal{U}, m}
\end{equation}
and thus the action of $A^{--}$ is by compact operators since the inclusions $ A^{\mathrm{an}}_{\mathcal{U}, m} \subset A^{\mathrm{an}}_{\mathcal{U}, m + 1}$ are compact.

If $\lambda \in X^{\bullet}_{+}(S_{G})$ then there is a natural $J_{G}$-equivariant inclusion 
$$
    V^{G}_{\lambda} \hookrightarrow A^{\mathrm{an}}_{\lambda, m},
$$
preserving the $A^{-}$ action.
\begin{definition}
If  $\lambda \in \mathcal{U} \subset \mathcal{W}_{m}$ is the restriction to $\mathfrak{S}_{G}$ of an algebraic character of $(S_{G}/S_{G}^{0})(\mathbb{Q}_{p})$, then there is a natural specialisation map 
$$
    \rho_{\lambda}: A^{\mathrm{an}}_{\mathcal{U}, m} \to A_{\lambda, m}^{\mathrm{an}}
$$
given by post-composition with the map
$$
    \Lambda_{\mathcal{U}} \to \Lambda_{\mathcal{U}} \otimes_{\lambda}\mathcal{O}_{L} \cong \mathcal{O}_{L}.
$$
\end{definition}
The space $A^{\mathrm{an}}_{\mathcal{U}, m}$ is not the unit ball in a Banach algebra, but we can define a basis $\{e_{i}\}_{i \in I}$ for a countable indexing set $I$ such that for any $f \in A^{\mathrm{an}}_{\mathcal{U}, m}$ there are $a_{i} \in \Lambda_{\mathcal{U}}$ such that
\begin{itemize}
    \item $a_{i} \to 0$ in the cofinite filtration on $I$.
    \item We have $f = \sum_{i \in I}a_{i}e_{i}$.
\end{itemize}
This can be seen from the identification of $A^{\mathrm{an}}_{\mathcal{U}, m}$ with a product of Tate algebras. This basis is sufficient to define Fredholm determinants of compact operators. 
\begin{definition}
Let $\mathcal{V} \subset \mathcal{U} \subset \mathcal{W}_{G}$ be an affinoid contained in a wide open disc $\mathcal{U}$. Define
$$
    A^{\mathrm{an}}_{\mathcal{V}, m} :=  A^{\mathrm{an}}_{\mathcal{U}, m} \hat{\otimes}_{\Lambda_{\mathcal{U}}} \mathcal{O}(\mathcal{V})^{\circ}
$$
where $\mathcal{O}(\mathcal{V})^{\circ}$ are the bounded-by-1 global sections of $\mathcal{V}$. 
\end{definition}
The space $A^{\mathrm{an}}_{\mathcal{V}, m}[1/p] =  A^{\mathrm{an}}_{\mathcal{U}, m} \hat{\otimes}_{\Lambda_{\mathcal{U}}} \mathcal{O}(\mathcal{V})$ is an orthonormalisable Banach $\mathcal{O}(\mathcal{V})$-module with unit ball $A^{\mathrm{an}}_{\mathcal{V}, m}$.

\subsection{Locally Iwasawa functions} \label{sec:iwasawa}
We continue with the notation of the previous section. 

\begin{definition}
For $m \geq 0$ define 
$$
    \mathrm{LI}_{m + 1}(\mathbb{Z}_{p}^{d_{G}}, \Lambda_{\mathcal{U}} ) = \{f: \mathbb{Z}_{p}^{d_{G}} \to \Lambda_{\mathcal{U}}:  \forall \ \underline{a} \in \mathbb{Z}_{p}^{d_{G}}, \exists f_{\underline{a}} \in \Lambda_{\mathcal{U}}[[ T_{1} \ldots, T_{d_{G}}]]  \ \text{s.t.}\  f(\underline{a} + p^{m + 1}\underline{x}) = f_{\underline{a}}(p\underline{x}) \ \forall \underline{x} \in \mathbb{Z}_{p}^{d_{G}}\}.
$$
This space admits a natural structure as a $\Lambda_{\mathcal{U}}[[T_{1}, \ldots, T_{d_{G}}]]$-module and is isomorphic to \\ $\prod_{\underline{a}}\Lambda_{\mathcal{U}}[[p^{-m}T_{1} \ldots, p^{-m}T_{d_{G}}]]$.
\end{definition}

\begin{definition}Let $\mathfrak{n}$ be the maximal ideal of $\Lambda_{\mathcal{U}}[[T_{1}, \ldots, T_{d_{G}}]]$. Define a filtration $\mathrm{Fil}_{m + 1}^{n}$ on $\mathrm{LI}_{m + 1}(\mathbb{Z}_{p}^{d_{G}}, \Lambda_{\mathcal{U}})$ by 
$$
\mathrm{Fil}_{m + 1}^{n} := \mathfrak{n}^{n} \mathrm{LI}_{m + 1}(\mathbb{Z}_{p}^{d_{G}}, \Lambda_{\mathcal{U}}).
$$
\end{definition}
We see that 
$$
     \mathrm{LI}_{m + 1}(\mathbb{Z}_{p}^{d_{G}}, \Lambda_{\mathcal{U}}) \cong \varprojlim_{n} \mathrm{LI}_{m + 1}(\mathbb{Z}_{p}^{d_{G}}, \Lambda_{\mathcal{U}})/\mathrm{Fil}_{m + 1}^{n}.
$$
The modules $\mathrm{LI}_{m + 1}(\mathbb{Z}_{p}^{d_{G}}, \Lambda_{\mathcal{U}})/\mathrm{Fil}_{m + 1}^{n}$ are finite and thus the modules $ \mathrm{LI}_{m + 1}(\mathbb{Z}_{p}^{d_{G}}, \Lambda_{\mathcal{U}})$ are profinite and in particular they are $\mathfrak{n}$-adically complete and separated.

For $m \geq 0$ there is a chain of inclusions 
\begin{equation} \label{eq:inc}
    \mathrm{LA}_{m}(\mathbb{Z}_{p}^{d_{G}}, \Lambda_{\mathcal{U}}) \subset \mathrm{LI}_{m + 1}(\mathbb{Z}_{p}^{d_{G}}, \Lambda_{\mathcal{U}}) \subset \mathrm{LA}_{m + 1}(\mathbb{Z}_{p}^{d_{G}}, \Lambda_{\mathcal{U}})
\end{equation}
given by restriction corresponding to the obvious inclusions
\begin{align*}
    \prod_{\underline{a} \mod p^{m}}\Lambda_{\mathcal{U}}\langle p^{-m}T_{1}, \ldots, p^{-m}T_{d_{G}}\rangle  &\hookrightarrow \prod_{\underline{a} \ \mathrm{mod} \ p^{m}}\Lambda_{\mathcal{U}}[[p^{-m}T_{1}, \ldots, p^{-m}T_{d_{G}}]] \\ &\hookrightarrow \prod_{\underline{a}  \ \mathrm{mod} \ p^{m + 1}}\Lambda_{\mathcal{U}}\langle p^{-(1 + m)}T_{1}, \ldots, p^{-(1 + m)}T_{d_{G}}\rangle.
\end{align*}

Write $\varinjlim_{m} \mathrm{LA}_{m}(\mathbb{Z}_{p}^{d_{G}}, \Lambda_{\mathcal{U}}) = \mathrm{LA}(\mathbb{Z}_{p}^{d_{G}}, \Lambda_{\mathcal{U}})$ for the space of all $\Lambda_{\mathcal{U}}$-valued locally analytic functions on $\mathbb{Z}_{p}^{d_{G}}$. The inductive systems $\{\mathrm{LA}_{m}(\mathbb{Z}_{p}^{d_{G}}, \Lambda_{\mathcal{U}})\}_{m \geq 0}$ and $\{\mathrm{LI}_{m + 1}(\mathbb{Z}_{p}^{d_{G}}, \Lambda_{\mathcal{U}})\}_{m \geq 0}$ are final in the inductive system 
$$
    \ldots \to  \mathrm{LA}_{m}(\mathbb{Z}_{p}^{d_{G}}, \Lambda_{\mathcal{U}}) \to \mathrm{LI}_{m + 1}(\mathbb{Z}_{p}^{d_{G}}, \Lambda_{\mathcal{U}}) \to \mathrm{LA}_{m + 1}(\mathbb{Z}_{p}^{d_{G}}, \Lambda_{\mathcal{U}}) \to \ldots,
$$
with the arrows given by the inclusions \eqref{eq:inc}. Thus 
$$
    \varinjlim_{m} \mathrm{LI}_{m}(\mathbb{Z}_{p}^{d_{G}}, \Lambda_{\mathcal{U}}) = \mathrm{LA}(\mathbb{Z}_{p}^{d_{G}}, \Lambda_{\mathcal{U}}).
$$

Suppose $\mathcal{U} \subset \mathcal{W}_{m}$.

\begin{definition}
For $m \geq 0$ define
$$
    A^{\mathrm{Iw}}_{\mathcal{U}, m + 1} := \{f: U_{\mathrm{Bru}}^{G}(\mathbb{Z}_{p}) \to \Lambda_{\mathcal{U}}: f(\overline{n}t n) = k_{\mathcal{U}}^{G}(t)f(n), \ \text{and} \  f \vert_{N_{G}(\mathbb{Z}_{p})} \in  \mathrm{LI}_{m + 1}(\mathbb{Z}_{p}^{d_{G}}, \Lambda_{\mathcal{U}})\},
$$
given the $\mathfrak{m}_{\mathcal{U}}$-adic topology. As in the locally analytic case we write $A^{\mathrm{Iw}}_{\lambda, m}$ in place of $A^{\mathrm{Iw}}_{\mathcal{U}, m}$ when $\mathcal{U} = \{\lambda\}$.
\end{definition}
These spaces are isomorphic to $\mathrm{LI}_{m}(\mathbb{Z}_{p}^{d_{G}}, \Lambda_{\mathcal{U}})$ as $\Lambda_{\mathcal{U}}$-modules and the filtration on $\mathrm{LI}_{m}(\mathbb{Z}_{p}^{d_{G}}, \Lambda_{\mathcal{U}})$ defines a filtration $\mathrm{Fil}_{\mathcal{U}, m}^{n} A^{\mathrm{Iw}}_{\mathcal{U}, m} \subset A^{\mathrm{Iw}}_{\mathcal{U}, m}$. When no danger of ambiguity exists we will write $\mathrm{Fil}_{\mathcal{U}, m}^{n}$ for $\mathrm{Fil}_{\mathcal{U}, m}^{n} A^{\mathrm{Iw}}_{\mathcal{U}, m}$.
\begin{proposition} \label{prop:invar}
for $n \geq 1$ the modules $A^{\mathrm{Iw}}_{\mathcal{U} m}$ and $ \mathrm{Fil}_{\mathcal{U}, m}^{n}$ are preserved by the actions of $J_{G}$ and $A^{-}$ inherited from those on $A^{\mathrm{an}}_{\mathcal{U}, m + 1}$.
\end{proposition}
\begin{proof}
Write $G^{\mathrm{an}}$ for the rigid analytification of $G_{\Qp}$. For $m \in \mathbb{Q}_{\geq 0}$, let $\mathcal{G}_{m}$ be the $\Qp$-rigid subgroup of $G^{\mathrm{an}}$ given by the rigid generic fibre (in the sense of Berthelot) of the formal group $\mathfrak{G}_{m}$ defined by
$$
    \mathfrak{G}_{m}(R) := \mathrm{ker}\left(G(R) \to G(R/p^{m})\right).
$$
for an admissible $\Zp$-algebra $R$ i.e. a quotient of $\Zp\langle T_{1}, \ldots, T_{n}\rangle$ for some $n$. 
Set $\mathcal{G}_{m}^{\circ} = \cup_{m' > m}\mathcal{G}_{m'} \subset \mathcal{G}_{0}$.
Define a $\Qp$-rigid analytic group $\mathcal{J}_{G,m}^{\circ} =  \mathcal{G}_{m}^{\circ} \cdot J_{G}$ (the group generated by $J_{G}$ and $\mathcal{G}_{m}^{\circ}$ in $\mathcal{G}_{0}$). 
\begin{lemma}
This group admits an Iwahori decomposition 
$$
    \mathcal{J}_{G, m}^{\circ} = (\overline{\mathcal{N}}_{m}^{\circ} \cdot \overline{N}_{G}(\mathbb{Z}_{p})) \times (\mathcal{L}_{m}^{\circ} \cdot L_{G}(\mathbb{Z}_{p})) \times (\mathcal{N}_{m}^{\circ} \cdot N_{G}(\mathbb{Z}_{p})).
$$
\end{lemma}
\begin{proof}
The proof of \cite[Proposition 3.3.2]{boxer2021higher} generalises easily to this case. 
\end{proof}
If we choose a set $\Gamma$ of representatives for $J_{G} \ \mathrm{mod} \ p^{m}$ then 
$$
    \mathcal{J}_{G,m}^{\circ} = \sqcup_{\gamma \in \Gamma}\mathcal{G}^{\circ}_{m}\gamma.
$$
Consider $\mathcal{O}(\mathcal{J}_{G, m}^{\circ})_{\Lambda_{\mathcal{U}}} := \mathcal{O}(\mathcal{J}_{G, m}^{\circ})^{\leq 1} \hat{\otimes}_{\mathbb{Z}_{p}} \Lambda_{\mathcal{U}}$. Since by assumption $k_{\mathcal{U}}^{G}: \mathfrak{S}_{G}\to \Lambda_{\mathcal{U}}^{\times}$ is an $m$-analytic character then by definition it extends to a character $k_{\mathcal{U}}^{G}: \mathcal{L}^{\circ}_{m}\cdot L_{G}(\Zp) \to \Lambda_{\mathcal{U}}^{\times}$ and if $f \in \mathrm{LI}_{m}(N_{G}(\Zp), \Lambda_{U})$ then $f$ extends uniquely to an element of $\mathcal{O}\left(\mathcal{N}_{m}^{\circ}N_{G}(\Zp)\right)_{\Lambda_{\mathcal{U}}}$. The product $k_{\mathcal{U}}^{G}\times f$ defines a function on $\mathcal{O}(\mathcal{J}_{G}^{\circ})_{\Lambda_{\mathcal{U}}}$ by extending $k_{\mathcal{U}}^{G}$ trivially to $(\overline{\mathcal{N}}_{m}^{\circ} \cdot \overline{N}_{G}(\mathbb{Z}_{p})) \times (\mathcal{L}_{m}^{\circ} \cdot L_{G}(\mathbb{Z}_{p}))$ which induces an isomorphism
$$
    \mathcal{O}(\mathcal{J}^{\circ}_{G})^{(k_{\mathcal{U}}^{G})}_{\Lambda_{\mathcal{U}}} \cong A^{\mathrm{Iw}}_{\mathcal{U}, m},
$$
where the isotypic $k_{\mathcal{U}}^{G}$-component is taken with respect to the left translation action of $\overline{Q}_{G}(\Zp)$.

Clearly the right translation action of the subgroup $J_{G} \subset \mathcal{J}_{G}^{\circ}$ preserves $\mathcal{O}(\mathcal{J}_{G}^{\circ})$ and commutes wtih isotypic components, giving us an action on $A^{\mathrm{Iw}}_{\mathcal{U}, m}$. 

Let $\mathbbm{1}$ be the trivial character valued in $\Lambda_{U}^{\times}$. The $\Lambda_{U}$-algebra $A^{\mathrm{Iw}}_{\mathbbm{1}, m} $ embeds as a subring of $\mathcal{O}(\mathcal{J}_{G, m}^{\circ})_{\Lambda_{\mathcal{U}}}$ and acts on $A^{\mathrm{Iw}}_{\mathcal{U}, m}$ via multiplication of functions in $\mathcal{O}(\mathcal{J}_{G, m}^{\circ})_{\Lambda_{\mathcal{U}}}$. Then $\mathrm{Fil}^{n}_{
\mathbbm{1}, m}$ is an ideal of $A^{\mathrm{Iw}}_{\mathbbm{1}, m}$ and $\mathrm{Fil}^{n}_{
\mathbbm{1}, m} = (\mathrm{Fil}^{1}_{
\mathbbm{1}, m})^{n}$. Furthermore, 
$$
    \mathrm{Fil}^{n}_{\mathcal{U}, m} = (\mathrm{Fil}^{1}_{\mathbbm{1}, m})^{n}A^{\mathrm{Iw}}_{\mathcal{U}, m}. 
$$
Since the action of $J_{G}$ respects the algebra structure of $\mathcal{O}(\mathcal{J}_{G,m}^{\circ})_{\Lambda_{\mathcal{U}}}$ in order to show that the filtration is preserved by $J_{G}$ it suffices to show that $\mathrm{Fil}^{1}_{\mathcal{U}, m}$ is preserved by $J_{G}$. To see this, we observe that 
$$
    \mathrm{Fil}^{1}_{\mathcal{U}, m} = \{F \in A^{\mathrm{Iw}}_{
    \mathcal{U}, m}: F(\gamma) \equiv 0 \ \mathrm{mod} \ \mathfrak{m}_{\mathcal{U}} \ \forall \ \gamma \in \Gamma\},
$$
so taking $F \in \mathrm{Fil}^{1}_{\mathcal{U}, m}, j \in J_{G}$ then for $\gamma \in \Gamma$ there is $\gamma_{j} \in \Gamma$ and $\varepsilon \in \mathcal{G}^{\circ}_{m}$ such that $\varepsilon \gamma_{j} = \gamma j$ and thus, since $p \in \mathfrak{m}_{\mathcal{U}}$ and $\varepsilon \equiv \ 1 \ \mathrm{mod} \ p$, we have
$$
    (j \cdot F)(\gamma) = F(\varepsilon \gamma_{j}) \equiv 0 \ \mathrm{mod} \ \mathfrak{m}_{\mathcal{U}}.
$$

The action of $A^{-}$ on $\mathcal{J}_{G,m}^{\circ}$ is by
$$
    a * \mathcal{J}_{G,m}^{\circ} = (\overline{\mathcal{N}}_{m}^{\circ} \cdot\overline{N}_{G}(\mathbb{Z}_{p})) \times (\mathcal{L}_{m}^{\circ} \cdot L_{G}(\mathbb{Z}_{p})) \times a(\mathcal{N}_{m}^{\circ} \cdot N_{G}(\mathbb{Z}_{p}))a^{-1}
$$
We can see that this is well-defined noting that for a set of representatives $\Gamma'$ of $N_{G}(\mathbb{Z}_{p})$ mod $p^{m}$ we have 
$$
    \mathcal{N}_{m}^{\circ} \cdot N_{G}(\mathbb{Z}_{p}) = \sqcup_{\gamma \in \Gamma'} \mathcal{N}_{m}^{\circ} \cdot \gamma
$$
and each $\mathcal{N}_{m}^{\circ} \cdot \gamma$ is isomorphic to a direct product of balls $\mathcal{U}_{\alpha,m}^{\circ} = \cup_{m' > m}\mathcal{U}_{\alpha, m}$, where $\mathcal{U}_{\alpha, m}$ is the ball of radius $p^{-m}$ and centre $0$ contained in the root space $\mathcal{U}_{\alpha}$: 
$$
\mathcal{N}_{m}^{\circ} \cdot \gamma= \prod_{\alpha \in \Phi^{L}} \mathcal{U}^{\circ}_{\alpha, m},
$$
where $\Phi^{L} = \Phi_{G} \backslash \Phi_{L}$. Thus we have an isomorphism
$$
    \mathcal{N}_{m}^{\circ} \cdot N_{G}(\mathbb{Z}_{p}) = \prod_{\alpha \in \Phi^{L}} \mathcal{U}_{\alpha, m} \times N_{G}(\mathbb{Z}/p^{m}\mathbb{Z})
$$
with the action of $a \in A^{-}$ given by 
\begin{align*}
     a\mathcal{N}_{m}^{\circ} \cdot N_{G}(\mathbb{Z}_{p})a^{-1} &= \prod_{\alpha \in \Phi^{L}} p^{v_{p}(\alpha(a))}\mathcal{U}_{\alpha, m} \times aN_{G}(\mathbb{Z}/p^{m}\mathbb{Z})a^{-1} \\ 
     &= \prod_{\alpha \in \Phi^{L}} \mathcal{U}_{\alpha, m + v_{p}(\alpha(a))} \times aN_{G}(\mathbb{Z}/p^{m}\mathbb{Z})a^{-1} 
\end{align*}
and $\mathcal{U}_{\alpha, m + v_{p}(\alpha(a))} \subset \mathcal{U}_{\alpha, m}$ since $v_{p}(\alpha(a)) \geq 0$ by definition of $A^{-}$, so $ a\mathcal{N}_{m}^{\circ} \cdot N_{G}(\mathbb{Z}_{p})a^{-1} \subset  \mathcal{N}_{m}^{\circ} \cdot N_{G}(\mathbb{Z}_{p})$. We can prove the filtration $\mathrm{Fil}^{n}_{\mathcal{U}, m}$ is invariant under the action of $A^{-}$ in a similar way to that of $J_{G}$. Since the action of $A^{-}$ commutes with taking isotypic components for the left translation action of $Q_{G}(\Zp)$ we are done. 

\end{proof}

\begin{corollary}
Let $K^{p} \subset \mathcal{G}(\mathbb{A}^{p})$ be a neat open-compact subgroup. The modules $A^{\mathrm{Iw}}_{\mathcal{U}, m}$ induce lisse \'etale sheaves $\mathscr{A}^{\mathrm{Iw}}_{\mathcal{U}, m}$ over $Y_{G}(K^{p}J_{G})$. 
\end{corollary}
\begin{proof}
As each $A^{\mathrm{Iw}}_{\mathcal{U}, m}/\mathrm{Fil}^{n}_{\mathcal{U}, m}$ is finite it defines an inverse system of constructible \'etale sheaves. 
\end{proof}
\section{Locally analytic branching laws}
\label{sec:LABL}
Let $\mathcal{U} \subset \mathcal{W}_{m} \subset \mathcal{W}_{G}$ be a wide-open disc for some $m \geq 0$ and let $c := \mathrm{dim}_{\R}Y_{G} - \mathrm{dim}_{\R}Y_{H} - \mathrm{rk}_{\R}\left(\frac{Z_{H}}{Z_{G} \cap H}\right)$, where $\mathrm{rk}_{\R}$ denotes the split rank over $\R$. We construct branching maps 
$$
    H^{i}(Y_{H}(Q_{H}^{0} \cap u^{-1}U_{n}u), \mathcal{O}) \to H^{i + c}(Y_{G}(V_{n}), A_{\mathcal{U}, m}^{\mathrm{an}})
$$
interpolating the algebraic branching maps of Section \ref{sec:branchlaw}.

\begin{lemma} \label{lem:welf}
The function 
\begin{align*}
    \tilde{f}^{\mathrm{sph}}_{\mathcal{U}}: U_{\mathrm{sph}}(\mathbb{Z}_{p}) \to \Lambda_{\mathcal{U}} \\
    \overline{n}\ell u q \mapsto k_{\mathcal{U}}^{G}(\ell) 
\end{align*}
for $\overline{n}\ell \in \overline{Q}_{G}(\Zp), q \in Q_{H}^{0}(\Zp)$ is well-defined.
\end{lemma}
\begin{proof}
 For $i = 1,2$, let $q_{i} \in Q_{H}^{0}(\mathbb{Z}_{p})$ and $\overline{n}_{i}\ell_{i} \in \overline{Q}_{G}(\mathbb{Z}_{p})$, and suppose $\overline{n}_{1}\ell_{1} uq_{1} = \overline{n}_{2}\ell_{2}uq_{2}$, then rearranging we get 
\begin{equation} \label{eq:grpmess} \tag{$\dagger$}
    (\overline{n}_{2}\ell_{2}u)^{-1}\overline{n}_{1}\ell_{1} u \in Q_{H}^{0}(\mathbb{Z}_{p}).
\end{equation}
We thus have that 
$$
    g := u^{-1}\ell_{2}^{-1}\overline{n}_{2}^{-1}\overline{n}_{1}\ell_{1}u \in Q^{0}_{H}(\mathbb{Z}_{p}) \cap u^{-1}\overline{Q}_{G}(\mathbb{Z}_{p})u.
$$
By (B) of \S \ref{sec:branchlaw} the projection of $g$ to $L_{G}$ lands in $L_{G}^{0}$ and thus lies in the kernel of $k_{\mathcal{U}}^{G}$; this projection coincides with $\ell_{2}\ell_{1}^{-1}$, thus
$$
    k_{\mathcal{U}}^{G}(\ell_{1}) = k_{\mathcal{U}}^{G}(\ell_{2}).
$$
\end{proof}
\begin{definition}
Define 
\begin{align*}
   \xi: U_{\mathrm{sph}}(\mathbb{Z}_{p}) &\to \mathfrak{S}_{G}, \\
    \overline{n}\ell u q &\mapsto \ell \ \mathrm{mod} \ L_{G}^{0}(\mathbb{Z}_{p}),
\end{align*}
which is well-defined by the proof of Lemma \ref{lem:welf}.
\end{definition}

\begin{definition}
 Define an action of $A^{-}$ on $U_{\mathrm{Bru}}^{G}(\mathbb{Z}_{p})$ by 
$$
    a * \overline{n}\ell n := \overline{n}\ell a na^{-1}.
$$
\end{definition}
\begin{lemma} \label{lem:shrinkage}
For any $a \in A^{--}$ we have 
$$
    \left(a *U_{\mathrm{Bru}}^{G}(\mathbb{Z}_{p}) \right)u \subset U_{\mathrm{Sph}}(\mathbb{Z}_{p}).
$$
\end{lemma}
\begin{proof}
It suffices to check this on $\mathcal{F}_{G}(\mathbb{Z}_{p})$. Furthermore, as  $U_{\mathrm{Sph}}$ is Zariski open in the smooth projective variety $\mathcal{F}_{G}$, it suffices to show that the inclusion holds mod $p$. We compute
\begin{align*}
     \left(a *U_{\mathrm{Bru}}^{G}(\mathbb{Z}_{p}) \right)u &=  \overline{Q}_{G}(\mathbb{Z}_{p})a N_{G}(\mathbb{Z}_{p}) a^{-1}u \\
     &\equiv \overline{Q}_{G}(\mathbb{F}_{p})u \ &&\mathrm{mod} \ p \\
     & \in U_{\mathrm{Sph}}(\mathbb{F}_{p}),
\end{align*}
where the second equality follows from the fact that $a N_{G}(\mathbb{Z}_{p}) a^{-1} \equiv 1 \ \mathrm{mod} \ p$.
\end{proof}

\begin{definition} \label{def:sphericalvector}
Define 
$$
    f^{\mathrm{sph}}_{\mathcal{U}}: U_{\mathrm{Bru}}^{G}(\Zp) \to \Lambda_{\mathcal{U}}
$$
by extending the restriction of $\tilde{f}^{\mathrm{sph}}_{\mathcal{U}}$ to $(U_{\mathrm{sph}} \cap U_{\mathrm{Bru}}^{G})(\Zp)$ by zero. 
\end{definition}

\begin{lemma}
The function $f^{\mathrm{sph}}_{\mathcal{U}}$satisfies the following properties:
\begin{enumerate}
\item It is contained in $A^{\mathrm{an}}_{\mathcal{U}, m}$ whenever $\mathcal{U} \subset \mathcal{W}_{m}$
\item It is invariant under right translation by $Q_{H}^{0} \cap u^{-1}U_{n}u$.
\end{enumerate}
\end{lemma}
\begin{proof}
Part (1) follows from the fact that 
$$
    \tilde{f}^{\mathrm{sph}}_{\mathcal{U}} = k_{\mathcal{U}}^{G} \circ \xi 
$$
is the composition of an algebraic map with an $m$-analytic map. For part (2) we note that the homogeneity property of $f_{\mathcal{U}}^{\mathrm{sph}}$ and the Iwahori decomposition of $J_{G}$ means that for any $g \in U_{\mathrm{Bru}}^{G}(\Zp)$ there is $j \in J_{G}$ such that $f_{\mathcal{U}}^{\mathrm{sph}}(g) = f_{\mathcal{U}}^{\mathrm{sph}}(j)$, so it suffices to show that the restriction of $f_{\mathcal{U}}^{\mathrm{sph}}$ to $J_{G}$ is invariant under $X := Q_{H}^{0}(\Zp) \cap u^{-1}U_{n}u$. Since $X \subset Q_{H}^{0} \cap J_{G}$ (recalling that we are assuming $u \in N_{G}(\Zp)$) then right translation by $x \in X$ preserves both $U_{\mathrm{sph}}(\Zp)$ and $J_{G}$ and thus $U_{\mathrm{sph}}(\Zp) \cap J_{G}$ and $J_{G} \backslash U_{\mathrm{sph}}(\Zp) \cap J_{G}$. Let $g \in U^{G}_{\mathrm{Bru}}(\Zp)$, then if $g \in U_{\mathrm{sph}}(\Zp) \cap J_{G}$ there is $\overline{n} \in \overline{N}_{G}(\Zp), \ell \in L_{G}(\Zp)$ and $q \in Q_{H}^{0}(\Zp)$ such that $g = \overline{n}\ell u q$ and for $x \in X$ we have $f_{\mathcal{U}}^{\mathrm{sph}}(gx) = f_{\mathcal{U}}^{\mathrm{sph}}(\overline{n}\ell uqx) = k_{\mathcal{U}}^{G}(\ell)$ since $x \in Q_{H}^{0}(\Zp)$. On the other hand, if $g \notin U_{\mathrm{sph}}(\Zp) \cap J_{G}$ then $xg \notin U_{\mathrm{sph}}(\Zp) \cap J_{G}$ and thus $f_{\mathcal{U}}^{\mathrm{sph}}(gx) = f_{\mathcal{U}}^{\mathrm{sph}}(g) = 0$ and we are done. 
\end{proof}

\subsection{Analytic branching maps}

Let $\mathcal{U} \subset \mathcal{W}_{m}$ be a wide-open disc, let $\lambda \in X^{\bullet}_{+}(S_{G}/S_{G}^{0})^{Q_{H}^{0}} \cap \mathcal{U}$ and let $\mu \in X^{\bullet}_{+}(S_{H}/S_{H}^{0})$ be the associated weight for $H$. We construct maps  
$$
    H^{i}_{\mathrm{Iw}}(Y_{H}(Q_{H}^{0} \cap u^{-1}U_{n}u), \mathcal{O}) \to H^{i + c}(Y_{G}(V_{n}), A^{\mathrm{an}}_{\mathcal{U}, m})
$$
interpolating the maps 
$$
    H^{i}(Y_{H}(Q_{H}^{0} \cap u^{-1}U_{n}u), \mathcal{O}) \to H^{i + c}(Y_{G}(V_{n}), A_{\lambda, m}^{\mathrm{an}})
$$
induced by the branching map 
$$
    \left(\mathcal{P}^{H}_{\mu}\right)^{\vee} \to V_{\lambda}^{G} \hookrightarrow A_{\lambda, m}^{\mathrm{an}}
$$

The element $f_{\mathcal{U}}^{\mathrm{sph}} \in A^{\mathrm{an}}_{\mathcal{U}, m}$ is invariant under the action of $Q_{H}^{0} \cap u^{-1}U_{n}u$, and thus defines an element of $H^{0}_{\mathrm{Iw}}(Y_{G}(Q_{H}^{0} \cap u^{-1}U_{n}u), A^{\mathrm{an}}_{\mathcal{U}, m})$. This allows us to define a map 
$$
    \mathrm{br}_{\mathcal{U}}: H^{i + c}_{\mathrm{Iw}}(Y_{G}(Q_{H}^{0} \cap u^{-1}U_{n}u), \mathcal{O}) \to H^{i + c}_{\mathrm{Iw}}(Y_{G}(Q_{H}^{0} \cap u^{-1}U_{n}u), A^{\mathrm{an}}_{\mathcal{U}, m})
$$
which can be described in the following way: let $Q_{H}^{0}(p^{k})$ be the points of $G(\Zp)$ which reduce to $Q_{H}^{0}(\Z/p^{k}\Z)$ mod $p^{k}$ so that $Q_{H}^{0}(p^{k}) \cap u^{-1}U_{n}u$ define a cofinal system of neighbourhoods of $Q_{H}^{0} \cap u^{-1}U_{n}u$. Writing $f_{\mathcal{U}, k}^{\mathrm{sph}}$ for the mod $p^{k}$ reduction of $f_{\mathcal{U}}^{\mathrm{sph}}$ we note that $f^{\mathrm{sph}}_{\mathcal{U}, k}$ is invariant under $Q_{H}^{0}(p^{k}) \cap u^{-1}U_{n}u$ and thus defines an element of $H^{0}(Y_{G}(Q_{H}^{0}(p^{k}) \cap u^{-1}U_{n}u), A^{\mathrm{an}}_{\mathcal{U}, m}/p^{k})$. The above map is given by taking the limit over the maps 
$$
    H^{i + c}(Y_{G}(Q_{H}^{0}(p^{k}) \cap u^{-1}U_{n}u), \mathcal{O}/p^{k}) \to H^{i + c}(Y_{G}(Q_{H}^{0}(p^{k}) \cap u^{-1}U_{n}u), A^{\mathrm{an}}_{\mathcal{U}, m}/p^{k})
$$
given by taking the cup product with $f^{\mathrm{sph}}_{\mathcal{U}, k}$.  Repeating this construction with $f_{\lambda}^{\mathrm{sph}}$for $\lambda \in X^{\bullet}_{+}(S_{G}/S_{G}^{0})^{Q_{H}^{0}}$ in place of $f_{\mathcal{U}}^{\mathrm{sph}}$ gives a map 
$$
    \mathrm{br}_{\lambda}: H^{i + c}_{\mathrm{Iw}}(Y_{G}(Q_{H}^{0} \cap u^{-1}U_{n}u), \mathcal{O}) \to H^{i + c}_{\mathrm{Iw}}(Y_{G}(Q_{H}^{0} \cap u^{-1}U_{n}u), A^{\mathrm{an}}_{\lambda, m}) 
$$
Classes pushed forward from $Y_{H}$ via the above maps do not interpolate algebraic classes; we will need to twist by $\tau^{n}u$. This is precisely the twist used in \cite{loefflerspherical} to give classes satisfying norm relations up to an application of $U'_{p}$. 
\begin{definition}
For $k \geq n \geq 1$ define open compact subgroups 
$$
    K_{k,n}^{G} := \{g \in G(\Zp): \tau^{n}u^{-1}gu\tau^{-n} \in Q_{H}^{0} \cap u^{-1}U_{n}u \ \mathrm{mod} \ p^{k}\} \subset V_{n},
$$
and write $K_{\infty, n}^{G} := \{g \in G(\Zp): \tau^{n}u^{-1}gu\tau^{-n} \in Q_{H}^{0} \cap u^{-1}U_{n}u \}$. 
\end{definition}
\begin{lemma}
Let $i \geq 0$, $\mathcal{U} \subset \mathcal{W}_{m}$ be a wide-open disc and let $\lambda \in  X^{\bullet}_{+}(S_{G}/S_{G}^{0})^{Q_{H}^{0}}$ be such that $\lambda\vert_{\mathfrak{S}_{G}} \in \mathcal{U}$. The diagram
$$
\begin{tikzcd}[column sep=huge]
    H^{i}_{\mathrm{Iw}}(Y_{G}(Q_{H}^{0} \cap u^{-1}U_{n}u), \mathcal{O}) \arrow[r, "\tau^{n}u \circ \mathrm{br}_{\mathcal{U}} \circ \iota_{*}"] \arrow[dr, "\tau^{n}u \circ \mathrm{br}_{\lambda} \circ \iota_{*}"] & H^{i}_{\mathrm{Iw}}(Y_{G}(K^{G}_{\infty, n}), A^{\mathrm{an}}_{\mathcal{U}, m}) \arrow[d, "\rho_{\lambda}"] \\
    & H^{i}_{\mathrm{Iw}}(Y_{G}(K^{G}_{\infty, n}), A^{\mathrm{an}}_{\lambda, m})
\end{tikzcd}
$$
commutes, where the unlabelled arrows are those constructed in the discussion above and we recall that $\rho_{\lambda}: A^{\mathrm{an}}_{\mathcal{U}, m} \to A^{\mathrm{an}}_{\lambda, m}$ is the map given by composition with $\lambda$.
\end{lemma}
\begin{proof}
This reduces to showing that the functions $\rho_{\lambda}(\tau^{n}uf^{\mathrm{sph}}_{\mathcal{U}}) = \tau^{n}u\rho_{\lambda}(f^{\mathrm{sph}}_{\mathcal{U}})$ and $\tau^{n}uf^{\mathrm{sph}}_{\lambda}$ are equal as elements of $A^{\mathrm{an}}_{\lambda, m}$. By construction they are equal upon restriction to $U_{\mathrm{sph}}(\Zp) \cap U_{\mathrm{Bru}}^{G}(\Zp)$ and since for $\overline{n}\ell n \in U_{\mathrm{Bru}}^{G}(\Zp)$ we have $\overline{n}\ell\tau^{n}n\tau^{-n}u \in U_{\mathrm{sph}}(\Zp)$ by Lemma \ref{lem:shrinkage} then 
$$
    \tau^{n}u\rho_{\lambda}\left(f^{\mathrm{sph}}_{\mathcal{U}}\right)(\overline{n}\ell n) = \rho_{\lambda}\left(f^{\mathrm{sph}}_{\mathcal{U}}\right)(\overline{n}\ell \tau^{n} n \tau^{-n}u) = f_{\lambda}^{\mathrm{sph}}(\overline{n}\ell\tau^{n}n\tau^{-n}u) = \tau^{n}uf^{\mathrm{sph}}_{\lambda}(\overline{n}\ell n)
$$
as required. 
\end{proof}

\begin{lemma}
Let $\lambda \in X_{+}^{\bullet}(S_{G})^{Q_{H}^{0}}$ and let $\mu$ be the associated $H$-weight. The diagram 
$$
\begin{tikzcd}
    H^{i}_{\mathrm{Iw}}(Y_{H}(Q_{H}^{0} \cap u^{-1}U_{n}u), \mathcal{O}) \arrow[r, "\iota_{*}"] \arrow[d, "\mathrm{mom}_{\mu}"] &  H^{i + c}_{\mathrm{Iw}}(Y_{G}(Q_{H}^{0} \cap u^{-1}U_{n}u), \mathcal{O}) \arrow[d, ""] \\
     H^{i}_{\mathrm{Iw}}(Y_{H}(Q_{H}^{0} \cap u^{-1}U_{n}u), V_{\mu, \mathcal{O}}^{H}) \arrow[r, "\iota_{*} \circ \mathrm{br}_{\lambda}"] & H^{i + c}_{\mathrm{Iw}}(Y_{G}(Q_{H}^{0} \cap u^{-1}U_{n}u), A_{\lambda, m}^{\mathrm{an}})
\end{tikzcd}
$$
commutes, where $\mathrm{mom}_{\mu}$ is the map constructed in \cite[Definition 5.1.1]{loeffler2021spherical} given by taking the limit over cup products with the mod $p^{k}$ highest weight vector of $V_{\mu, \mathcal{O}}^{H}$.
\end{lemma}
\begin{proof}
Modulo $p^{k}$ we have, for $z \in  H^{i}(Y_{H}(\iota^{-1}(Q_{H}^{0}(p^{k}) \cap u^{-1}U_{n}u)), \mathcal{O}/p^{k})$
\begin{align*}
   \iota_{*} \circ \mathrm{br_{\lambda}}\left( \mathrm{mom}_{\mu}\left(z\right)\right) &= (\iota_{*} \circ \mathrm{br}_{\lambda})(z \cup f^{H}_{\mu, k}) \\
   & = \iota_{*}\left(z \cup \iota^{*}f^{\mathrm{sph}}_{\lambda, k}\right) \\
   &= \iota_{*}(z) \cup f^{\mathrm{sph}}_{\lambda, k},
\end{align*}
and taking the limit over $k$ gives us the result. 
\end{proof}
Putting this all together we obtain pushforward maps 
$$
     \iota_{\mathcal{U},n, *}  := \mathrm{pr}^{K^{G}_{\infty, n}}_{V_{n}} \circ 
 \tau^{n}_{*} \circ u_{*} \circ \mathrm{br}_{\mathcal{U}} \circ  \iota_{*} : H^{i}_{\mathrm{Iw}}(Y_{H}(Q_{H}^{0} \cap u^{-1}U_{n}u), \mathcal{O}) \to  H^{i + c}(Y_{G}(V_{n}), A^{\mathrm{an}}_{\mathcal{U}, m}) 
$$
interpolating the maps 
$$
     \iota_{\lambda, *} = \mathrm{pr}^{K^{G}_{\infty, n}}_{V_{n}} \circ 
 \tau^{n}_{*} \circ u_{*} \circ \mathrm{br}_{\lambda} \circ  \iota_{*}: H^{i}_{\mathrm{Iw}}(Y_{H}(Q_{H}^{0} \cap u^{-1}U_{n}u), V^{H}_{\mu_{ \lambda}, \mathcal{O}}) \to  H^{i + c}(Y_{G}(V_{n}), A^{\mathrm{an}}_{\lambda, m})
$$
for all $\lambda \in X^{\bullet}_{+}(S_{G})^{Q_{H}^{0}} \cap \mathcal{U}$. Note that we have dropped the Iw subscript denoting Iwasawa cohomology since Iwasawa cohomology of an open-compact subgroup is just the cohomology of that subgroup.

By replacing the local system $A^{\mathrm{an}}_{\mathcal{U}, m}$ with the lisse \'etale sheaves $\mathscr{A}^{\mathrm{Iw}}_{\mathcal{U}, m}$, the construction of the previous section goes through in \'etale cohomology: Let $\Sigma$ be as in \S \ref{sec:coho}, then, making sure to keep track of twists, we obtain pushforward maps 
$$
     \iota_{\mathcal{U},n, *}: H^{i}_{\mathrm{Iw}, \acute{et}}(Y_{H}(Q_{H}^{0} \cap u^{-1}U_{n}u)_{\Sigma}, \mathcal{O}) \to  H^{i + c}(Y_{G}(V_{n})_{\Sigma}, \mathscr{A}^{\mathrm{Iw}}_{\mathcal{U}, m}(c)) 
$$
interpolating the maps 
$$
     \iota_{\lambda, *}: H^{i}_{\mathrm{Iw}, \acute{e}t}(Y_{H}(Q_{H}^{0} \cap u^{-1}U_{n}u)_{\Sigma}, V^{H}_{\mu_{\lambda}, \mathcal{O}}) \to  H^{i + c}(Y_{G}(V_{n})_{\Sigma}, \mathscr{A}^{\mathrm{Iw}}_{\lambda, m}(c))
$$
for all $\lambda \in X^{\bullet}_{+}(S_{G})^{Q_{H}^{0}} \cap \mathcal{U}$. Similarly, the construction works for absolute \'etale cohomology although we won't use it in this article. In practice we shall suppress the \textit{\'et} subscript in Iwasawa \'etale cohomology as it will be obvious from context (or won't matter) which we are using.

 These pushforward maps fit into the framework of \cite{loefflerspherical} and so as a direct corollary to \cite[Proposition 4.5.2]{loefflerspherical} we obtain the following norm compatibility.
\begin{theorem} \label{thm:normrelation}
Suppose we have classes $z^{H}_{n} \in H^{i}_{\mathrm{Iw}}(Y_{H}(Q_{H}^{0} \cap u^{-1}U_{n}u), \mathcal{O})$ for all $n \geq 1$ satisfying the compatibility relation 
$$
    \mathrm{pr}^{Q_{H}^{0} \cap u^{-1}U_{n + 1}u}_{Q_{H}^{0} \cap u^{-1}U_{n}u, *}(z_{n + 1}^{H}) = z_{n}^{H}.
$$
Define 
$$
    \xi^{G}_{\mathcal{U}, n} := \iota_{\mathcal{U}, n, *}\left(z_{n}^{H}\right) \in H^{i + c}(Y_{G}(V_{n}), A^{\mathrm{an}}_{\mathcal{U}, m}).
$$
These classes satisfy the norm-relation
$$
    \left(\mathrm{pr}^{n + 1}_{n}\right)_{*}\left(\xi^{G}_{\mathcal{U}, n + 1}\right) = U'_{p}\cdot \xi^{G}_{\mathcal{U}, n}.
$$
\end{theorem}

\section{Complexes of Banach modules} \label{sec:complex}
\subsection{Slope decompositions}
We recall the theory of slope decompositions for compact operators on non-archimedean Banach modules, our main references being \cite{Ashsteve} and \cite{urbaneigen}. Let $K/\Qp$ be a complete field extension and let $R$ be a Banach $K$-algebra
\begin{definition}
 Consider $F \in R\{\{T\}\}$ and $h \in \mathbb{R}_{\geq 0}$. We say that $F$ has a slope $\leq h$ factorisation if we have a factorisation
$$
    F = Q \cdot S
$$
where $Q$ is a polynomial and $S$ satisfies $S(0) = 1$ (is \textit{Fredholm} in the language of \cite{Ashsteve}), such that 
\begin{enumerate}
    \item Every slope of $Q$ is $\leq h$
    \item $S$ has slope $> h$
    \item $p^{h}$ is in the interval of convergence of $S$. 
\end{enumerate}
Such a factorisation is unique if it exists. For a polynomial $P(t) \in R[t]$ define $P^{*}(t) = t^{\mathrm{deg}(P)}P(1/t)$. 
\end{definition}
\begin{definition} \label{def:slope}
Let $M$ be an $R$-module equipped with an $R$-linear endomorphism $u$. For $h \in \mathbb{Q}$ we say that $M$ has a \textit{$\leq h$-slope decomposition} if it decomposes as a direct sum 
$$
    M = M^{u \leq h} \oplus M^{u > h}
$$
such that
\begin{itemize}
    \item Both summands are $u$-stable.
    \item $M^{u \leq h}$ is finitely generated over $R$.
    \item For every $m \in M^{u \leq h}$ there is a polynomial $Q \in R[t]$ of slope $\leq h$  with $Q^{*}(0)$ a multiplicative unit, such that $Q^{*}(u)m = 0$.
    \item For any polynomial $Q \in R[t]$ of slope $\leq h$ with $Q^{*}(0)$ a multiplicative unit, the map 
    $$
    Q^{*}(u): M^{u > h} \to M^{u > h}
    $$
    is an isomorphism. 
\end{itemize}
\end{definition}
If such a decomposition exists it is unique and $u$ acts invertibly on $M^{u \leq h}$. 
Let $M$ be a projective Banach $R$-module with an action of $\mathfrak{U}_{p}^{--}$ by compact operators. For $u \in \mathfrak{U}_{p}^{--}$ let $F_{u} \in R\{\{t\}\}$ denote the Fredholm determinant of $u$ acting on $M$. The following theorem \cite[Theorem 2.3.8]{urbaneigen} follows directly from results of Coleman, Serre and Buzzard: 

\begin{theorem} \label{thm:1}
Let $R, M$ be as above and suppose that $M$ is equipped with an $R$-linear compact operator $u$. If we have a prime decomposition $F_{u}(T) = Q(T)S(T)$ in $R\{\{T\}\}$ with $Q$ a polynomial such that $Q(0) = 1$ and $Q^{*}(0) \in R^{\times}$ then there exists $R_{Q}(T) \in TR\{\{T\}\}$ whose coefficients are polynomials in the coefficients of $Q$ and $S$ and we have a decomposition of $M$:
$$
    M = N_{u}(Q) \oplus F_{u}(Q)
$$
into closed $R$ submodules satisfying
\begin{itemize}
    \item The projector on $N_{u}(Q)$ is given by $R_{Q}(u)$. 
    \item $Q^{*}(u)$ annihilates $N_{u}(Q)$.
    \item $Q^{*}(u)$ is invertible on $F_{u}(Q)$. 
\end{itemize}
If $R$ is Noetherian then $N_{u}(Q)$ is projective of finite rank and 
$$
    \det(1 - tu \mid N_{u}(Q)) = Q(t).
$$
\end{theorem}

When the decomposition $F_{u} = QS$ is a slope $\leq h$ factorisation then the decomposition in the above theorem is a slope $\leq h$ factorisation and $M^{u \leq h} = N_{u}(Q)$. In this case we write $e_{u}^{\leq h} := R_{Q}(u)$.

\subsection{Slope decompositions on cohomology}
Let $K = K_{p}K^{p} \subset \mathcal{G}(\mathbb{A}_{f})$ be a neat open compact subgroup with $K_{p} \subset J_{G}$ admitting an Iwahori decomposition in the sense that the product map 
$$
    \left(\overline{N}_{1} \cap K_{p}\right) \times \left(L_{G}(\Zp) \cap K_{p}\right) \times \left(N_{G}(\Zp) \cap K_{p}\right) \to K_{p}
$$
is a bijection, and let $R$ be a $\mathbb{Q}_{p}$ Banach algebra.
\begin{definition}
Let $\mathcal{T}_{K_{p}} \subset G(\Qp)$ be the monoid generated by $K_{p}$ and $A^{-}$. For an $R[\mathcal{T}_{K_{p}}]$-module $M$ let 
$$
    \mathscr{C}^{\bullet}(Y_{G}(K), M)
$$
be the `Borel--Serre' complex defined in \cite[Section 2.1]{hansen2017universal} whose cohomology computes $H^{\bullet}(Y_{G}(K), M)$ (as $R$-modules).
 We let $R\Gamma(Y_{G}(K), M)$ be the image of the above complex in the derived category of Banach $R$-modules. 
\end{definition}

Suppose $M$ is an orthonormalisable Banach $R$-module and with a continuous action of $A^{-}$. Then we can define an action of the Hecke algebra $\mathbb{T}_{S,p}^{-}$ on the complex $\mathscr{C}^{\bullet}(Y_{G}(K), M)$ via its interpretation as an algebra of double coset operators. Suppose further that $A^{--}$ acts compactly on $M$. Then the action of $\mathfrak{U}_{p}^{--}$ on $ \mathscr{C}^{\bullet}(Y_{G}(K), M)$ acts compactly on the total complex $\oplus_{i}  \mathscr{C}^{i}(Y_{G}(K), M)$. We refer to the following proposition from  \cite[2.3.3]{hansen2017universal}:
\begin{proposition} \label{prop:shrinkslope}
Let $R$ be an affinoid algebra. If $C^{\bullet}$ is a complex of projective Banach $R$-algebras equipped with an $R$-linear compact operator $u$, then for any $x \in \mathrm{Sp}(R)$ and $h \in \mathbb{Q}_{\geq 0}$ there is an affinoid subdomain  $\mathrm{Sp}(R') \subset \mathrm{Sp}(R)$ such that $x \in \mathrm{Sp}(R')$ and such that the complex $C^{\bullet} \hat{\otimes}_{R}R'$ admits a slope $\leq h$ decomposition for $u$ and $(C^{\bullet} \hat{\otimes}_{R} R')^{u \leq h}$ is a complex of finite flat $R'$-modules.
\end{proposition}

We will also need the following easy technical lemma.

\begin{lemma} \label{lem:shrinklem}
Let $N \subset M$ be an inclusion of projective Banach $R$-modules with $R$-linear compact operator $u$ such that 
$$
    uM \subset N
$$
Suppose further that $N, M$ admit slope $\leq h$ decompositions, then for $h \in \Q_{\geq 0}$ we have 
$$
    M^{u \leq h} = N^{u \leq h}.
$$
Moreover, if $e^{\leq h}_{u}$ is the slope $\leq h$ idempotent on $M$ associated to $u$ by Theorem \ref{thm:1} then 
$$
    e^{\leq h}_{u}N = N^{u\leq h}
$$
i.e. $e^{\leq h}_{u}$ is a slope $\leq h$ idempotent for $N$.
\end{lemma}
\begin{proof}
By \cite[Theorem 4.1.2(c)\&(d)]{Ashsteve} we have a slope $\leq h$ decomposition on $M/N$ such that 
$$
    0 \to N^{u \leq h} \to M^{u \leq h} \to (M/N)^{u \leq h} \to 0
$$
is an exact sequence. However, $u(M/N) = 0$ by our hypothesis so $u$ has infinite slope on $M/N$ and thus $(M/N)^{u \leq h} = 0$ for any $h \in \Q_{\geq 0}$ (as $u$ must be invertible on finite slope parts) whence $N^{u \leq h} = M^{u \leq h}$.

For the statement involving the idempotent it's easy to see that for any idempotent operator $\phi$ on $M$ preserving $N$ we have $\phi(N) = \phi(M)$ and that $e^{\leq h}_{u}$ preserves $N$ is immediate since $$
    e^{\leq h}_{u}N \subset M^{u \leq h} = N^{u \leq h}. 
$$
\end{proof}

By Proposition \ref{prop:shrinkslope} we can shrink $\mathcal{V}$ to an affinoid $\mathcal{V}'$ also containing $x_{0}$ and such that the complex admits a slope $\leq h$ decomposition for any $u \in \mathfrak{U}^{--}_{p}, h \in \mathbb{Q}_{\geq 0}$ with projector $e_{u}^{\leq h} \in \mathcal{O}(\mathcal{V}')\{\{u\}\}$. In this case by Lemma \ref{lem:shrinklem} we have
$$
     \mathscr{C}^{\bullet}(Y_{G}(K), A^{\mathrm{an}}_{\mathcal{V}', m})^{u \leq h} \cong \mathscr{C}^{\bullet}(Y_{G}(K), A^{\mathrm{an}}_{\mathcal{V}', m + 1})^{u \leq h}.
$$
We will need slope decompositions for coefficents defined over a wide-open disc in weight space. 

\begin{lemma}
Let $\mathcal{U} \subset \mathcal{W}_{G}$ be a wide open disc and let $x_{0} \in \mathcal{U}$. There is a wide open disc $x_{0} \in \mathcal{U}' \subset \mathcal{U}$ such that the complex $  \mathscr{C}^{\bullet}(Y_{G}(K), A^{\mathrm{an}}_{\mathcal{U}', m})$ admits a slope $\leq h$ decomposition.
\end{lemma}
\begin{proof}
As explained above, there is an affinoid $\mathcal{V} \subset \mathcal{U}$ containing $x_{0}$ over which $ \mathscr{C}^{\bullet}(Y_{G}(K), A^{\mathrm{an}}_{\mathcal{V}, m})$ admits a slope $\leq h$ decomposition. By shrinking we can assume that $\mathcal{V}$ is a closed disc centered on $x_{0}$. Let $\mathcal{U}'$ be the wide-open disc given by taking the `interior' wide-open disc of this affinoid disc. Since an orthonormal basis of $A^{\mathrm{an}}_{\mathcal{V}, m}$ gives an orthonormal basis of $A^{\mathrm{an}}_{\mathcal{U}', m}$ and the Banach norm on $\Lambda_{G}(\mathcal{U}')[1/p]$  restricts to the Gauss norm on $\mathcal{O}(\mathcal{V})$ we get a slope $\leq h$ decomposition on $\mathscr{C}^{\bullet}(Y_{G}(K), A^{\mathrm{an}}_{\mathcal{U}', m})$ and all the above results hold in this case. Note in particular that the projector $e_{u}^{\leq h}$ is still in $\mathcal{O}(\mathcal{V})\{\{u\}\}$ when computing the decomposition for $\mathcal{U}'$.
\end{proof}
There is a comparison isomorphism from \'etale to Betti cohomology. In particular, for any $i\geq 0$ and any profinite $\Zp[K]$-module $M$ we have a canonical isomorphism 
$$
    j^{(i)}_{M}: H^{i}(\overline{Y}_{G}(K_{p}), \mathscr{M}) \cong H^{i}(Y_{G}(K_{p}), M),
$$
one can take the perspective that we are endowing the Betti cohomology group $H^{i}(Y_{G}(K_{p}), M)$ with a Galois action. 
\begin{lemma} \label{lem:crush}
Let $h \in \Q_{\geq 0}$ and let $\mathcal{U}$ be a wide-open disc such that $\mathscr{C}^{\bullet}(Y_{G}(K), A^{\mathrm{an}}_{\mathcal{U}, m})$ admits a slope $\leq h$ decomposition. Then the slope $\leq h$ total cohomology $H^{\bullet}(Y_{G}(K), A^{\mathrm{an}}_{\mathcal{U}, m + 1})^{u\leq h}$ is a Galois-stable direct summand of the etale cohomology group $H^{\bullet}(\overline{Y}_{G}(K), \mathscr{A}^{\mathrm{Iw}}_{\mathcal{U}, m + 1})$ as a $\Lambda_{\mathcal{U}}$-module under the Betti-\'etale comparison isomorphism.
\end{lemma}
\begin{proof}
For $m \geq 0$ write 
\begin{align*}
    M_{m} :=& \mathscr{C}^{\bullet}(Y_{G}(K), A^{\mathrm{an}}_{\mathcal{U}, m}) \\
    I_{m} :=& \mathscr{C}^{\bullet}(Y_{G}(K), A^{\mathrm{Iw}}_{\mathcal{U}, m}),
\end{align*}
so that we have natural inclusions 
\begin{equation}\label{eq:aniw}
    M_{m} \subset I_{m + 1} \subset M_{m + 1}.
\end{equation}
The key point is that by fixing a chain homotopy defining the chain complex corestriction map we can take the action of $u$ on $M_{m}, I_{m + 1}$ to be the restriction of the action of $u$ on $M_{m + 1}$, so we can apply Lemma \ref{lem:shrinklem}. Since $uM_{m + 1} \subset M_{m}$ it follows from \eqref{eq:aniw} that 
$$
    uI_{m + 1} \subset  M_{m}.
$$
By Theorem \ref{thm:1} there is are idempotents $e_{u, m}^{\leq h} \in uR\{\{u\}\}$ for each $m \geq 0$ such that 
$$
    e_{u, m}^{\leq h}M_{m + 1} = I_{m + 1}^{u \leq h} = M_{m}^{u \leq h}
$$
where the second equality is Lemma \ref{lem:shrinklem}. Moreover, by Lemma \ref{lem:shrinklem} we have that 
$$
    e_{u, m}^{\leq h} M_{m} = M_{m}^{u \leq h}
$$
(i.e. $e_{u, m}^{\leq h}$ does not depend on $m$) from which we can infer that 
$$
   M_{m}^{u \leq h} = e^{\leq h}I_{m + 1}
$$ 
is a direct summand of $I_{m + 1}$ and thus $H^{\bullet}(M_{m})^{u \leq h}$ is a direct summand of the total cohomology $H^{\bullet}(I_{m})$ by functoriality. It follows that 
$$
    j_{A^{\mathrm{an}}_{\mathcal{U}, m}}^{(\bullet)}\left(H^{\bullet}(Y_{G}(K_{p}), A^{\mathrm{an}}_{\mathcal{U}, m})\right) \subset  j_{A^{\mathrm{an}}_{\mathcal{U}, m}}^{(\bullet)}\left(H^{\bullet}(Y_{G}(K_{p}), A^{\mathrm{Iw}}_{\mathcal{U}, m})\right) = H^{\bullet}(\overline{Y}_{G}(K_{p}), \mathscr{A}^{\mathrm{Iw}}_{\mathcal{U}, m})
$$
is a direct summand. 
To show Galois-stability we note that for each $i$
$$
    H^{i}(\overline{Y}_{G}(K), \mathscr{A}^{\mathrm{Iw}}_{\mathcal{U}, m}) = \varprojlim_{n}H^{i}(\overline{Y}_{G}(K), \mathscr{A}^{\mathrm{Iw}}_{\mathcal{U}, m}/\mathrm{Fil}^{n})
$$
as Galois modules, with each $H^{\bullet}(\overline{Y}_{G}(K), \mathscr{A}^{\mathrm{Iw}}_{\mathcal{U}, m}/\mathrm{Fil}^{n})$ finite. Since $\tilde{e}_{u}^{\leq h} := j_{A^{\mathrm{an}}_{\mathcal{U}, m}}^{(\bullet)} \circ e_{u}^{\leq h} \circ \left(j_{A^{\mathrm{an}}_{\mathcal{U}, m}}^{(\bullet)}\right)^{-1}$ can be represented by a polynomial in $u$ mod $p^{n}$ and since the Hecke operators commute with the Galois action we see that for $g \in G_{\Q}$
$$
    g \cdot \tilde{e}^{\leq h} = \tilde{e}^{\leq h} \cdot g \ \mathrm{mod} \ p^{n}
$$
for all $n$ and by taking the limit over $n$ we get the result. 
\end{proof}
\subsection{Refined slope decompositions and classicality} \label{sec:refinedslope}

As in \cite[Section 3.5]{salazar2021parabolic} we consider a more refined slope decomposition. For $i = 1, \ldots, n$ let $Q_{G, i}^{\mathrm{max}}$ denote the maximal  parabolic subgroups of $G$ containing $Q_{G}$. These correspond to a subset $\{\alpha_{1}, \ldots, \alpha_{n}\}$ of the simple roots of $G$ and by taking $a_{i} \in A^{-}$ such that $v(\alpha_{i}(a_{i})) > 0$ and $v(\alpha_{j}(a_{i})) = 0$ for $j \neq i$ we can associate Hecke operators $U_{i} \in \mathfrak{U}_{p}^{-}$ as the image of $a_{i}$ under the isomorphism $\mathbb{Z}_{p}[A^{-}/A(\mathbb{Z}_{p})] \cong \mathfrak{U}_{p}^{-}$.

 Let $\textbf{h} := (h_{1}, \ldots, h_{n}) \in \mathbb{Q}_{\geq 0}^{n}$. Suppose a Banach $R$ module $M$ admits a slope $\leq h_{\mathrm{aux}}$ decomposition with respect to the compact operator $U_{0} := U_{1}\cdots U_{n} \in \mathfrak{U}_{p}^{--}$ for some $h_{\mathrm{aux}} > \prod_{i}h_{i}$, so that $M^{\leq h_{\mathrm{aux}}}$ is a finite projective Banach $R$-module. In particular, the whole of $A^{-}$ acts compactly on $M^{\leq h_{\mathrm{aux}}}$ and supposing that the Fredholm series $F_{i}$ admit slope $h_{i}$ decompositions for each $i$ then we can define 
$$
    M^{\leq \textbf{h}} = \cap_{i} (M^{\leq h_{\mathrm{aux}}})^{\leq h_{i}}.
$$
\begin{lemma} \label{lem:multidec}
For $0 \leq i \leq n$ let $\textbf{h}^{(i)} = (h_{1}, \ldots, h_{i})$ and suppose the $U_{i}$ act compactly on $M$. Set 
$$
    M^{\leq \textbf{h}^{(i + 1)}} 
   = (M^{\leq\textbf{h}^{(i)}})^{U_{i + 1}\leq  \textbf{h}_{i + 1}}.
$$
Then 
$$
     M^{\leq \textbf{h}^{(n)}}  =  M^{\leq \textbf{h}}.
$$
\end{lemma}
\begin{proof}

 It suffices to prove the following statement: Suppose $M$ is a Banach module equipped with two compact operators $U_{1}, U_{2}$ whose Fredholm determinants $F_{1}, F_{2}$ admit slope $h_{1}, h_{2}$ factorisations respectively. Then 
 $$
    M^{U_{1} \leq h_{1}} \cap  M^{U_{2} \leq h_{2}} =  (M^{U_{1} \leq h_{1}})^{U_{2} \leq h_{2}}
 $$
 Suppose $F_{2} = Q_{2}S_{2}$ is the slope $\leq h_{2}$ factorisation and $\tilde{F}_{2} = \tilde{Q}_{2}\tilde{S}_{2}$ is a slope factorisation of the Fredholm determinant $\tilde{F}_{2}$ of $U_{2}$ restricted to $M^{U_{1} \leq h_{1}}$. Then $\tilde{Q}_{2}$ divides $Q_{2}$ so for $m \in  (M^{U_{1} \leq h_{1}})^{U_{2} \leq h_{2}}$ we have $Q^{*}_{2}m = 0$ and thus $m \in  M^{U_{1} \leq h_{1}} \cap  M^{U_{2} \leq h_{2}}$. 
 
 Conversely suppose $m \in  M^{U_{1} \leq h_{1}} \cap  M^{U_{2} \leq h_{2}}$. Then in particular $m \in  M^{U_{1} \leq h_{1}}$ and so we can write $m = m_{2} + n$ where $m_{2} \in  (M^{U_{1} \leq h_{1}})^{U_{2} \leq h_{2}}$ and $n$ is in the complement $ (M^{U_{1} \leq h_{1}})^{U_{2} > h_{2}}$. As $m \in  M^{U_{2} \leq h_{2}}$ we have $Q^{*}_{2}(U_{2})m = 0$ but also $Q^{*}_{2}(U_{2})m_{2} = 0$ by the same argument as in the first inclusion, so $Q^{*}_{2}(U_{2})n = 0$ and since $Q^{*}_{2}$ is a slope $\leq h_{2}$ polynomial and $Q^{*}_{2}(0)$ is a multiplicative unit then $n = 0$ by Definition \ref{def:slope}. 
\end{proof}
\begin{corollary}
Let $M$ be a projective Banach $R$-module with an action of an $R$-linear compact operator $u$. Then the module $M^{\leq \textbf{h}}$ is a finite projective $R$-module and a direct summand of $M$ with projector $e^{\leq \textbf{h}} \in R\{\{U_{1}, \ldots, U_{n}\}\}$.
\end{corollary}
\begin{proof}
The above lemma states that we can obtain the refined slope decomposition $M^{\leq \textbf{h}}$ on $M^{\leq h_{\mathrm{aux}}}$ by iteratively taking a finite number of slope decompositions so the result follows from Theorem \ref{thm:1}.
\end{proof}
\begin{definition}
Set $h_{i}^{\mathrm{crit}} := -(\langle \lambda, \alpha_{i} \rangle + 1)v(\alpha(a_{i}))$.
\end{definition}

We say $\textbf{h}$ is \textit{non-critical} if for all $i = 1, \ldots, n$ we have $h_{i} < h^{\mathrm{crit}}_{i}$. 
\begin{theorem} \label{thm:classical}
For $\textbf{h} \in \mathbb{Q}_{\geq 0}^{n}$ non-critical, $\lambda \in X^{\bullet}_{+}(S_{G})$ and any $m \geq 0$ there is a quasi-isomorphism:
$$
    R\Gamma(Y_{G}(K), A_{\lambda,m}^{\mathrm{an}})^{\leq \textbf{h}} \cong  R\Gamma(Y_{G}(K), V^{G}_{\lambda})^{\leq \textbf{h}}. 
$$
\end{theorem}
\begin{proof}
 This is proved for compactly supported cohomology with coefficients in modules of distributions in \cite[Theorem 4.4]{salazar2021parabolic} using the dual of a parabolic locally analytic BGG complex. The same proof applied without taking the dual adapts easily to our setting using this complex.
\end{proof}
More generally we introduce the following notion:
\begin{definition} \label{def:noncrit}
Let $\mathfrak{m} \subset \mathbb{T}^{-}_{S, p}[1/p]$ be a maximal ideal such that $R\Gamma(Y_{G}(K), A^{\mathrm{an}}_{\lambda, m})[1/p]_{\mathfrak{m}} \neq 0$ . We say that $\mathfrak{m}$ is \textit{non-critical} if the map $V_{\lambda}^{G} \to A_{\lambda,m}^{\mathrm{an}}[1/p]$ induces a quasi-isomorphism
$$
    R\Gamma(Y_{G}(K), A_{\lambda,m}^{\mathrm{an}})[1/p]_{\mathfrak{m}}\cong  R\Gamma(Y_{G}(K), V^{G}_{\lambda})_{\mathfrak{m}}. 
$$
\end{definition}
We end with a variation of Proposition \ref{prop:shrinkslope}
\begin{lemma}
Let $R$ be an affinoid algebra and let $C^{\bullet}$ be a complex of projective Banach $R$-algebras equipped with a continuous $R$-linear action of $\mathfrak{U}_{p}^{-}$ such that $\mathfrak{U}_{p}^{--}$ acts via compact operators. Then for any $x \in \mathrm{Sp}(R)$ and $\textbf{h} \in \mathbb{Q}_{\geq 0}^{n}$ there is an affinoid subdomain  $\mathrm{Sp}(R') \subset \mathrm{Sp}(R)$ such that $x \in \mathrm{Sp}(R')$ and such that the complex $C^{\bullet} \hat{\otimes}_{R}R'$ admits a slope $\leq \textbf{h}$ decomposition and $(C^{\bullet} \hat{\otimes}_{R} R')^{\leq \textbf{h}}$ is a complex of finite flat $R'$-modules.
\end{lemma}
\begin{proof}
We know that by Proposition \ref{prop:shrinkslope} there is an affinoid $\mathrm{Sp}(R_{0})$ containing $x$ and such that $U_{0}$ admits a slope $\leq h_{\mathrm{aux}}$ decomposition on $\oplus_{i} C^{i}$ and affinoids $\mathrm{Sp}(R_{i})$ such that $U_{i}$ admits a slope decomposition on $(\oplus_{i} C^{i})^{\leq h_{\mathrm{aux}}}$ and  $x \in \mathrm{Sp}(R_{i})$ for each $i$. Taking the intersection of these subsets gives us the required affinoid. The finite flatness follows from the above corollary.
\end{proof}
\begin{definition}
If $\mathcal{U}$ is a wide-open disc such that $H^{\bullet}(Y_{G}(K), A^{\mathrm{an}}_{\mathcal{U}, m})$ admits a slope $\leq \textbf{h}$ decomposition for some (and therefore any) $m \geq 0$ we say that $\mathcal{U}$ is \textit{$\textbf{h}$-adapted}.
\end{definition}
\subsection{Control theorem}

We prove control results for the cohomology of locally symmetric spaces. 
\begin{lemma}
Let $\lambda \in \mathcal{W}_{m}(\Qp)$ be an algebraic weight and $\mathcal{U} \subset \mathcal{W}_{m}$ a wide-open disc containing $\lambda$, then there is a quasi-isomorphism
$$
    R\Gamma(Y_{G}(K), A^{\mathrm{an}}_{\mathcal{U}, m})[1/p] \otimes^{L}_{\Lambda_{\mathcal{U}}[1/p]} \Qp \sim R\Gamma(Y_{G}(K), A^{\mathrm{an}}_{\lambda, m})[1/p].
$$
where $\Qp$ is given the structure of a $\Lambda_{\mathcal{U}}$-module via the homomorphism $\Lambda_{\mathcal{U}} \to \Qp$ given by evaluation at $\lambda$. 
\end{lemma}
\begin{proof}
This follows from the fact that 
$$
    A^{\mathrm{an}}_{\mathcal{U}, m}[1/p] \otimes_{\Lambda_{\mathcal{U}}[1/p]} \Qp = A^{\mathrm{an}}_{\lambda, m}[1/p]. 
$$ 
\end{proof}
\begin{corollary} \label{cor:specseq}
For $\textbf{h} \in \mathbb{Q}_{\geq 0}^{n}$ non-critical, an $\textbf{h}$ adapted wide-open disc $\mathcal{U} \subset \mathcal{W}_{m}$ and algebraic $\lambda \in \mathcal{U}$, there is a quasi-isomomorphism
$$
     R\Gamma(Y_{G}(K), A^{\mathrm{an}}_{\mathcal{U}, m})[1/p]^{\leq \textbf{h}} \otimes^{L}_{\Lambda_{\mathcal{U}}[1/p]} \Qp \sim R\Gamma(Y_{G}(K), V_{\lambda, \mathcal{O}})[1/p]^{\leq \textbf{h}}.
$$
\end{corollary}
\begin{proof}
This is an immediate corollary of Theorem \ref{thm:classical} and the previous lemma.
\end{proof}

More generally, for \textit{any} $\textbf{h} \in \Q^{n}_{ \geq 0}$ and a non-critical maximal ideal $\mathfrak{m} \subset \mathcal{T}_{\lambda, \textbf{h}}$ (i.e. $\mathfrak{m}$ pulls back to a non-critical maximal ideal under the surjection $\mathbb{T}^{-}_{S, p}[1/p] \to \mathcal{T}_{\lambda, \textbf{h}}$) lifting to a maximal ideal $\mathfrak{M} \subset \mathcal{T}_{\mathcal{U}, \textbf{h}}$, we have 
$$
     R\Gamma(Y_{G}(K), A^{\mathrm{an}}_{\mathcal{U}, m})[1/p]_{\mathfrak{M}}^{\leq \textbf{h}} \otimes^{L}_{\Lambda_{\mathcal{U}}[1/p]} \Qp \sim R\Gamma(Y_{G}(K), V_{\lambda, \mathcal{O}})[1/p]_{\mathfrak{m}}^{\leq \textbf{h}}.
$$
\subsection{Vanishing results}
Let $\mathcal{U}$ be an $\textbf{h}$-adapted wide-open disc. 
\begin{definition}
For $\textbf{h} \in \mathbb{Q}^{n}_{\geq 0}$ set
\begin{align*}
    \mathcal{T}_{\mathcal{U}, \textbf{h}} &= \mathrm{im}\left( \mathbb{T}^{-}_{S,p}\otimes \Lambda_{\mathcal{U}}[1/p] \to \mathrm{End}_{\Lambda_{\mathcal{U}}[1/p]}(R\Gamma(Y_{G}(K), A^{\mathrm{an}}_{\mathcal{U}, m})[1/p]^{\leq \textbf{h}} )\right ) \\
     \mathcal{T}_{\lambda, \textbf{h}} &= \mathrm{im}\left( \mathbb{T}^{-}_{S,p}[1/p] \to \mathrm{End}_{\mathcal{O}}(R\Gamma (Y_{G}(K), A_{\lambda, m}^{\mathrm{an}})[1/p]^{\leq \textbf{h}}) \right )
\end{align*}
\end{definition}

\begin{lemma} 
The natural map 
$$
    r_{\lambda}: \mathcal{T}_{\mathcal{U}, \textbf{h}} \to \mathcal{T}_{\lambda, \textbf{h}}
$$
induces a bijection
\begin{equation} \label{eq:bij}
    \mathrm{Specm}(\mathcal{T}_{\lambda, \textbf{h}}) \to \mathrm{Specm}(\mathcal{T}_{\mathcal{U}, \textbf{h}}/ \mathrm{ker}(r_{\lambda}))
\end{equation}
\end{lemma}
\begin{proof}
This is \cite[Theorem 6.2.1(ii)]{Ashsteve}.
\end{proof}

\begin{lemma} \label{lem:disappear}
Suppose $\lambda \in \mathcal{U} \cap X^{\bullet}_{+}(S_{G}/S_{G}^{0})^{Q_{H}^{0}}$ and  $\mathfrak{m}_{\lambda} \subset \mathcal{T}_{\lambda,\textbf{h}}$ is a maximal ideal such that
$$
    R\Gamma(Y_{G}(K), A^{\mathrm{an}}_{\lambda, m})[1/p]^{\leq \textbf{h}}_{\mathfrak{m}_{\lambda}}
$$
has cohomology concentrated in some degree $j \geq 0$. If $\mathfrak{M}_{\mathcal{U}}$ is the image of $\mathfrak{m}_{\lambda}$ under the identification \eqref{eq:bij} then 
$$
    R\Gamma(Y_{G}(K), A^{\mathrm{an}}_{\mathcal{U}, m})_{\mathfrak{M}_{\mathcal{U}}}^{\leq \textbf{h}}
$$
is quasi-isomorphic to a complex of projective $(\mathcal{T}_{\mathcal{U}, \textbf{h}})_{\mathfrak{M}_{\mathcal{U}}}$ modules with cohomology concentrated in degree $j$.
\end{lemma}
\begin{proof}
This follows from the Lemma 2.9 of \cite{barrera2021p}.
\end{proof}

\section{Classes in Galois cohomology} \label{sec:classesingalois}
We give a recipe for mapping \'etale classes into Galois cohomology. From now on we assume that $\mathcal{G}, \mathcal{H}$ admit compatible Shimura data and set $q:= \frac{1}{2}\mathrm{dim}_{\R}Y_{G}(K_{p})$ for some (hence any) choice of open-compact subgroup $K_{p} \subset G(\Zp)$, where as usual we leave the choice of prime-to-$p$ level subgroup implicit. 
\subsection{Bits of eigenvarieties and families of Galois representations} \label{sec:eigenvar}

Let $\textbf{h} \in \mathbb{Q}^{n}_{\geq 0}$. Consider the total cohomology $H^{\bullet}(Y_{G}(K), A^{\mathrm{an}}_{\mathcal{U},m})[1/p]^{\leq \textbf{h}}$ for an $\textbf{h}$-adapted wide-open disc $\mathcal{U}$.
\begin{definition}
For $\mathcal{U}, \textbf{h}$ as above define $ \mathcal{E}_{\mathcal{U}, \textbf{h}}$ to be the quasi-Stein rigid space with global sections
\begin{align*}
    \mathcal{O}(\mathcal{E}_{\mathcal{U}, \textbf{h}}) := \mathcal{T}_{\mathcal{U}, \textbf{h}} \otimes_{\Lambda_{\mathcal{U}}[1/p]} \mathcal{O}(\mathcal{U}).
\end{align*}
The structure morphism
$$
    \textbf{w}: \mathcal{E}_{\mathcal{U}, \textbf{h}} \to \mathcal{W}_{G}
$$
is finite and we refer to it as the \textit{weight map}.
\end{definition}
A point $x \in \mathcal{E}_{\mathcal{U}, \textbf{h}}$ corresponds to a maximal ideal $\mathfrak{m}_{x}$ of the Hecke algebra $\mathbb{T}_{S, p}^{-}\otimes \mathcal{O}(\mathcal{U})$ such that the eigensystem 
$$
    \mathbb{T}_{S, p}^{-} \otimes \mathcal{O}(\mathcal{U}) \to \mathbb{T}_{S, p}^{-} \otimes \mathcal{O}(\mathcal{U})/\mathfrak{m}_{x} =: L
$$
occurs in $H^{\bullet}(Y_{G}(K), A^{\mathrm{an}}_{\textbf{w}(x), m})[1/p]^{\leq \textbf{h}} \otimes L$ i.e.
$$
    \left(H^{\bullet}(Y_{G}(K), A^{\mathrm{an}}_{\textbf{w}(x)})[1/p]^{\leq \textbf{h}} \otimes L\right)_{\mathfrak{m}_{x}} \neq 0.
$$
By a mild abuse of notation we write $x \in \mathcal{E}_{\mathcal{U}, \textbf{h}}(L)$, identifying $x$ with the morphism $\operatorname{Spm}(L) \to \mathcal{E}_{\mathcal{U}, \textbf{h}}$ mapping the unique point of $\operatorname{Spm}(L)$ to $x$. 
\begin{definition}
Let $\textbf{h} \in \Q^{n}_{\geq 0}$ and  $x \in  \mathcal{E}_{\mathcal{U}, \textbf{h}}(L)$.
\begin{itemize}
\item We say $x$ is \textit{classical} if $\textbf{w}(x)$ is the restriction of a dominant algebraic character of $T_{G}$ and the associated eigensystem occurs in $H^{\bullet}(Y_{G}(K), V^{G}_{\textbf{w}(x)}) \otimes_{\Qp}L$. 
\item We say $x$ is of non-critical slope if $\textbf{h}_{x} = \left(v_{p}\left(U_{1}(x)\right), \ldots, v_{p}\left(U_{n}(x)\right)\right)$ is a non-critical slope.
\item We say that $x$ is non-critical if it has classical weight and the corresponding maximal ideal $\mathfrak{m}_{x} \subset \mathcal{T}_{\mathcal{U}, \textbf{h}}$ is non-critical. 
\end{itemize}
\end{definition}
\begin{remark}
Theorem \ref{thm:classical} says that points $x \in \mathcal{E}_{\mathcal{U}, \textbf{h}}$ for whom  $\textbf{h}_{x} = \left(v_{p}\left(U_{1}(x)\right), \ldots, v_{p}\left(U_{n}(x)\right)\right)$ is a non-critical slope are non-critical and that non-critical points are classical. 
\end{remark}
\begin{definition} \label{def:reallynice}
We say a classical point $x \in \mathcal{E}_{\mathcal{U}, \textbf{h}}(L)$ is \textit{really nice} if it is non-critical and
$$
    \left (H^{\bullet}(Y_{G}(K), V^{G}_{\textbf{w}(x)})[1/p]\otimes_{\Qp}L\right)_{\mathfrak{m}_{x}} = \left(H^{q}(Y_{G}(K), V^{G}_{\textbf{w}(x)})[1/p] \otimes_{\Qp} L \right)_{\mathfrak{m}_{x}},
$$ 
i.e. the localisation at $\mathfrak{m}_{x}$ is concentrated in the middle degree. The localised cohomology is then a finite free $\mathcal{O}(\mathcal{E}_{\mathcal{U}, \textbf{h}})_{\mathfrak{m}_{x}}$-module and the weight map is \'etale at $x$.
\end{definition}
A quasi-Stein rigid space $X$ with structure sheaf $\mathcal{O}$ has, by definition, an admissible covering by an increasing sequence of affinoid subspaces 
$$
    X_{1} \subset \cdots X_{n} \subset X_{n + 1} \subset \cdots \subset X.
$$
such that the induced maps $\mathcal{O}(X_{n + 1}) \to \mathcal{O}(X_{n})$ have dense image for $n \geq 1$. An $\mathcal{O}(X)$-module $M$ is called \textit{coadmissible} if $M = \varprojlim_{n} M_{n}$, where $M_{n}$ is a finitely generated module over $\mathcal{O}(X_{n})$ and the natural maps $M_{n + 1} \otimes \mathcal{O}(X_{n}) \to M_{n}$ coming from the inverse system, are isomorphisms. There is an equivalence of categories between coadmissible $\mathcal{O}(X)$-modules and coherent sheaves on $X$. 
\begin{proposition}
There is a complex of coherent sheaves $M^{\bullet}_{\mathcal{U},h}$ over $\mathcal{E}_{\mathcal{U}, \textbf{h}}$ whose global sections are given by 
$$
     \mathscr{C}^{\bullet}(Y_{G}(K), A^{\mathrm{an}}_{\mathcal{U}, m})[1/p]^{\leq \textbf{h}} \otimes_{\Lambda_{\mathcal{U}}[1/p]} \mathcal{O}(\mathcal{U}). 
$$
\end{proposition}
\begin{proof}
Since $\mathcal{U}$ is quasi-Stein it admits an admissible covering
$$
    \mathcal{U}_{1} \subset \cdots \subset \mathcal{U}
$$
by affinoid subspaces satisfying the condition that the maps $\mathcal{O}(\mathcal{U}_{n + 1}) \to \mathcal{O}(\mathcal{U}_{n})$ have dense image. Let $u \in A^{--}$ and write 
\begin{align*}
    M &:= \mathscr{C}^{\bullet}(Y_{G}(K), A^{\mathrm{an}}_{\mathcal{U}, m})[1/p]\otimes_{\Lambda_{\mathcal{U}}[1/p]} \mathcal{O}(\mathcal{U}).
\end{align*}
Note that $M^{\leq \textbf{h}} =  \mathscr{C}^{\bullet}(Y_{G}(K), A^{\mathrm{an}}_{\mathcal{U}, m})[1/p]^{\leq \textbf{h}} \otimes_{\Lambda_{\mathcal{U}}[1/p]} \mathcal{O}(\mathcal{U})$ since taking slope $\leq \textbf{h}$ parts commutes with taking tensor products with trivial $u$-modules. 
Let 
$$
M_{n} :=  M \otimes_{\mathcal{O}(\mathcal{U})} \mathcal{O}(\mathcal{U}_{n}) =  \mathscr{C}^{\bullet}(Y_{G}(K), A^{\mathrm{an}}_{\mathcal{U}, m}\otimes_{\Lambda_{\mathcal{U}}}\mathcal{O}(\mathcal{U}_{n}))
$$
and define an inverse system $\{M_{n}, \phi_{n + 1}: M_{n + 1} \to M_{n}\}_{n \geq 1}$ where $\phi_{n}$ are the natural maps. Then $M = \varprojlim_{n} M_{n}$ as $\mathcal{O}(\mathcal{U})[u]$-modules because the $\phi_{n}$ are $u$-equivariant. By Theorem \ref{thm:1} each $M_{n}^{\leq \textbf{h}}$ is a finitely generated $\mathcal{O}(\mathcal{U}_{n})$-module and thus the inverse system $\{M_{n}^{\leq \textbf{h}}\}_{n \geq 1}$ defines a coherent sheaf on $\mathcal{O}(\mathcal{U})$ with global sections $\varprojlim_{n} M_{n}^{\leq \textbf{h}}$. We claim that $M^{\leq \textbf{h}} = \varprojlim_{n} M_{n}^{\leq \textbf{h}}$. We have a direct sum decomposition 
$$
    M \cong \varprojlim_{n} M_{n} = \varprojlim_{n} \left(M_{n}\right)^{\leq \textbf{h}} \oplus \varprojlim_{n} \left(M_{n}\right)^{> \textbf{h}}
$$
Using that $u$-equivariant morphisms preserve slope decompositions we see that the natural maps $M^{\leq \textbf{h}} \to M_{n}^{ > \textbf{h}}$ and $M^{> \textbf{h}} \to M_{n}^{ \leq \textbf{h}}$ are zero and thus the natural maps $M^{\leq \textbf{h}} \hookrightarrow M \to \varprojlim_{n} M_{n}^{ \leq \textbf{h}}$ and $M^{ > \textbf{h}} \hookrightarrow M \to \varprojlim_{n} M_{n}^{ > \textbf{h}}$ are injective and thus isomorphisms since $M = M^{\leq \textbf{h}} \oplus M^{ > \textbf{h}}$. We thus have a complex of coherent sheaves $\tilde{M}^{\bullet}_{\mathcal{U}, \textbf{h}}$ over $\mathcal{U}$ and we define 
$$
    M^{\bullet}_{\mathcal{U}, \textbf{h}} := \textbf{w}^{*}\tilde{M}^{\bullet}_{\mathcal{U}, \textbf{h}}.
$$
\end{proof}
\begin{proposition} \label{prop:vanish}
Let $x \in  \mathcal{E}_{\mathcal{U}, \textbf{h}}(L)$ be a really nice point.
 Then there is an affinoid neighbourhood $x \in \mathcal{V} \subset \mathcal{E}_{\mathcal{U}, \textbf{h}}$ such that the complex of sheaves
$$
   M^{\bullet}_{\mathcal{U}, \textbf{h}}\vert_{\mathcal{V}}
$$
has cohomology concentrated in degree $q$.
\end{proposition}
\begin{proof}
 By Lemma \ref{lem:disappear} the stalk of $M^{\bullet}_{\mathcal{U}, \textbf{h}}$ at $x$ is quasi-isomorphic to a complex concentrated in degree $q$. By coherence we can find an affinoid $\mathcal{V} \subset \mathcal{U}$ containing $x$ such that 
$$
    M^{\bullet}_{\mathcal{U}, \textbf{h}}\vert_{\mathcal{V}}
$$
has cohomology concentrated in degree $q$.
\end{proof}

\begin{definition} \label{def:righteous}
We say an affinoid $\mathcal{V} \subset \mathcal{E}_{\mathcal{U}, \textbf{h}}$ is \textit{righteous} if the restriction of the weight map to $\mathcal{V}$ is an isomorphism onto its image and $M^{\bullet}_{\mathcal{U}, \textbf{h}}\vert_{\mathcal{V}}$ has cohomology concentrated in degree $q$. 
\end{definition}

Clearly a subaffinoid of a righteous affinoid is itself righteous. 
\begin{lemma}
If $x \in \mathcal{E}_{\mathcal{U}, \textbf{h}}$ is really nice then it has an affinoid neighbourhood $\mathcal{V}$ which is righteous. 
\end{lemma}
\begin{proof}
This follows immediately from the weight map being \'etale at really nice points and Proposition \ref{prop:vanish}.
\end{proof}

\begin{definition} \label{def:colemanfam}
Suppose $x \in \mathcal{E}_{\mathcal{U}, \textbf{h}}$ is a really nice point and let $\underline{\Pi} \subset \mathcal{E}_{\mathcal{U}, \textbf{h}}$ be a neighbourhood of $x$ such that
$$
    \textbf{w}\vert_{\underline{\Pi}}: \underline{\Pi} \cong \textbf{w}\left(\underline{\Pi}\right) =: \mathcal{U}(\underline{\Pi})
$$
i.e. the weight map restricts to an isomorphism.
We call such a neighbourhood a \textit{Coleman family} passing through $x$. 
\end{definition}

Suppose now that $\underline{\Pi}\subset  \mathcal{E}_{\mathcal{U}, \textbf{h}}$ is a Coleman family passing through a really nice point $x \in \mathcal{E}_{\mathcal{U}, \textbf{h}}$. By shrinking $\mathcal{U}$ we can assume that $\underline{\Pi}$ is a connected component of $\mathcal{E}_{\mathcal{U}, \textbf{h}}$ \footnote{ \textit{c.f.} \cite[Lemma and Definition 7.6.1]{bellaicheeigenbook}. The idea is that since the weight map $\textbf{w}$ is finite flat the spaces $\mathcal{E}_{\mathcal{V}, \textbf{h}}$ for $\textbf{w}(x) \in \mathcal{V} \subset \mathcal{U}$ form a basis of neighbourhoods of the finite set $\textbf{w}^{-1}\left(\textbf{w}(x)\right)$. By taking $\mathcal{V}$ sufficiently small we can ensure that the connected component of $\mathcal{E}_{\mathcal{V}, \textbf{h}}$ containing $x$ contains no other elements of $\textbf{w}^{-1}\left(\textbf{w}(x)\right)$.}.
Let $\tilde{f} \in \mathcal{T}_{\mathcal{U}, \textbf{h}}$ be an idempotent satisfying
$$
    \tilde{f} \cdot \mathcal{T}_{\mathcal{U}, \textbf{h}} \cong \mathcal{O}(\underline{\Pi}) 
$$
and $\tilde{f}\cdot H^{\bullet}(Y_{G}(K), A^{\mathrm{an}}_{\mathcal{U}, m})[1/p]^{\leq \textbf{h}} =  H^{q}(Y_{G}(K), A^{\mathrm{an}}_{\mathcal{U}, m})[1/p]^{\leq \textbf{h}} \otimes \mathcal{O}(\underline{\Pi})$ and
let $f \in \mathbb{T}^{-}_{S, p} \hat{\otimes} \mathcal{O}(\underline{\Pi})$ be a lift of $\tilde{f}$.
\begin{definition} \label{def:galrep}
Define an $\mathcal{O}(\mathcal{U}(\underline{\Pi}))$-linear Galois representation
$$
    W_{\underline{\Pi}} := H^{q}(Y_{G}(K), A^{\mathrm{Iw}}_{\mathcal{U}, m})[1/p]^{\leq \textbf{h}} \hat{\otimes}_{\mathcal{T}_{\mathcal{U}, h}}\mathcal{O}(\underline{\Pi}) = f \cdot H^{q}(Y_{G}(K), A^{\mathrm{Iw}}_{\mathcal{U}, m})[1/p]^{\leq \textbf{h}}.
$$
where we let $G_{E}$ (recall that $E$ is the reflex field of $Y_{G}$) act on $H^{q}(Y_{G}(K), A^{\mathrm{Iw}}_{\mathcal{U}, m})[1/p]^{\leq \textbf{h}}$ via the comparison isomorphism $j^{(q)}_{A^{\mathrm{Iw}}_{\mathcal{U}, m}}$ recalled in Section \ref{sec:complex}.
\end{definition}
This Galois representation is a direct summand of $H^{q}(Y_{G}(K), A^{\mathrm{Iw}}_{\mathcal{U}, m})[1/p]^{\leq \textbf{h}}$ with projector $f$.
\begin{lemma}
For $L/\Qp$, $x \in \mathcal{E}_{\mathcal{U}, \textbf{h}}(L)$ a really nice point and $\underline{\Pi}$ a Coleman family passing through $x$, we have 
$$
    W_{\underline{\Pi}} \otimes_{\mathcal{O}(\mathcal{U}(\underline{\Pi}))} L \cong H^{q}(Y_{G}(K), \mathscr{V}^{G}_{\textbf{w}(x)})_{\mathfrak{m}_{x}} \otimes L =: W_{x}.
$$
\end{lemma}

\begin{proof}
The left hand side is isomorphic to the localisation of $H^{q}(Y_{G}(K), A^{\mathrm{Iw}}_{w(x), m})[1/p]^{\leq \textbf{h}}$ at $\mathfrak{m}_{x}$ and this is equal to the right hand side by definition of non-criticality, identifying Betti and \'etale cohomology under the comparison isomorphisms $j^{(\bullet)}_{A^{\mathrm{Iw}}_{\mathcal{U}, m}}$ and $j^{(\bullet)}_{V^{G}_{\textbf{w}(x)}}$.
\end{proof}

\subsection{Abel--Jacobi maps} \label{sec:abeljac}

Let $K = K^{p}K_{p} \subset \mathcal{G}(\mathbb{A}_{f})$ be a neat open compact subgroup with $K_{p} \subset J_{G}$ and admitting an Iwahori decomposition and let $\Sigma$ be a set of primes of the reflex field $E$ such that we have an integral model $Y_{G}(K_{p})_{\Sigma}$ of $Y_{G}(K_{p})$ over $\mathcal{O}_{E}[\Sigma^{-1}]$. As usual let $\mathcal{U} \subset \mathcal{W}_{m}$ be a wide-open disc and let $\lambda \in \mathcal{U}$ be a dominant algebraic character. Let $\textbf{h} \in \mathbb{Q}_{\geq 0}^{d}$ be a non-critical slope and $x \in \mathcal{E}_{\mathcal{U}, \textbf{h}}$ a really nice point satisfying $\textbf{w}(x) = \lambda$ and with a Coleman family $\underline{\Pi}$ passing through it. There is an `Abel--Jacobi' map 
$$
    H^{q + 1}(Y_{G}(K)_{\Sigma}, \mathscr{V}_{\lambda}^{G})_{\mathfrak{m}_{x}} \to H^{1}(\mathcal{O}_{E}[\Sigma^{-1}], W_{x})   
$$
constructed in the following way: The Hochschild--Serre spectral sequence 
\begin{equation} \label{eq:hochse}
    H^{i}(\mathcal{O}_{E}[\Sigma^{-1}], H^{j}(\overline{Y}_{G}(K), \mathscr{V}^{G}_{\lambda})) \implies H^{i + j}(Y_{G}(K)_{\Sigma}, \mathscr{V}^{G}_{\lambda})
\end{equation}
has an edge map $H^{q + 1}(Y_{G}(K)_{\Sigma}, \mathscr{V}^{G}_{\lambda}) \to H^{q + 1}(\overline{Y}_{G}(K), \mathscr{V}^{G}_{\lambda})^{G_{\Q}}$. Denoting the kernel of this edge map by $H^{q + 1}(Y_{G}(K)_{\Sigma}, \mathscr{V}^{G}_{\lambda})_{0}$ there is an induced map $H^{q + 1}(Y_{G}(K)_{\Sigma}, \mathscr{V}^{G}_{\lambda})_{0} \to H^{1}(\mathcal{O}_{E}[\Sigma^{-1}], H^{q}(\overline{Y}_{G}(K), \mathscr{V}^{G}_{\lambda}))$. Localising \eqref{eq:hochse} at $\mathfrak{m}_{x}$ kills the terms $H^{i}(\mathcal{O}_{E}[
\Sigma^{-1}], H^{j}(\overline{Y}_{G}(K), \mathscr{V}^{G}_{\lambda}))$ for $i \geq 0$ and $j \neq q$ (by our assumption that $x$ is really nice) and thus we obtain a `classical' Abel--Jacobi map
$$
    \mathrm{AJ}^{\mathrm{cl}}_{\mathfrak{m}_{x}}: H^{q + 1}(Y_{G}(K)_{\Sigma}, \mathscr{V}^{G}_{\lambda})_{\mathfrak{m}_{x}} \to H^{1}(\mathcal{O}_{E}[\Sigma^{-1}], W_{x}).
$$
\begin{remark} \label{rem:rn}
The definition of $\mathrm{AJ}^{\mathrm{cl}}_{\mathfrak{m}_{x}}$ only requires that localisation at $\mathfrak{m}_{x}$ kills the groups $H^{i}(\overline{Y}_{G}(K), \mathscr{V}^{G}_{\lambda})$ outside of the middle degree i.e. we do not require any assumptions on the local behaviour of the eigenvariety at the point $x$. 
\end{remark}
By the exact same process we can define an `analytic' Abel--Jacobi map 
$$
    \mathrm{AJ}^{\mathrm{an}}_{x}: H^{q + 1}(Y_{G}(K)_{\Sigma}, \mathscr{A}^{\mathrm{Iw}}_{\lambda,m})_{\mathfrak{m}_{x}} \to H^{1}(\mathcal{O}_{E}[\Sigma^{-1}], W_{x})
$$
where we have identified $W_{x} = H^{q}(\overline{Y}_{G}(K), \mathscr{V}^{G}_{\lambda})_{\mathfrak{m}_{x}} \cong H^{q}(\overline{Y}_{G}(K), \mathscr{A}_{\lambda, m}^{\mathrm{Iw}})[1/p
]_{\mathfrak{m}_{x}}$
The main theorem of this section is the following
\begin{theorem}
Let $W_{\underline{\Pi}}$ be the Galois representation associated to $\underline{\Pi}$ as in Definition \ref{def:galrep}. There is an `overconvergent Abel--Jacobi map'
$$
    \mathrm{AJ}_{\underline{\Pi}}^{\mathrm{oc}}:H^{q + 1}(Y_{G}(K)_{
    \Sigma}, \mathscr{A}^{\mathrm{Iw}}_{\mathcal{U},m})[1/p] \to H^{1}(\mathcal{O}_{E}[\Sigma^{-1}], W_{\underline{\Pi}}),
$$
such that the following diagram commutes 
\begin{equation} \label{fig:abjac}
    \begin{tikzcd}
        H^{q + 1}(Y_{G}(K)_{
    \Sigma}, \mathscr{A}^{\mathrm{Iw}}_{\mathcal{U},m})[1/p] \arrow[r, "\mathrm{AJ}_{\underline{\Pi}}^{\mathrm{oc}}"] \arrow[d, ""] &  H^{1}(\mathcal{O}_{E}[\Sigma^{-1}], W_{\underline{\Pi}}) \arrow[d, ""] \\ 
     H^{q + 1}(Y_{G}(K)_{\Sigma}, \mathscr{A}_{\lambda, m}^{\mathrm{Iw}} )[1/p]_{\mathfrak{m}_{x}}\arrow[r, "\mathrm{AJ}_{x}^{\mathrm{an}}"] &  H^{1}(\mathcal{O}_{E}[\Sigma^{-1}], W_{x})
    \end{tikzcd}
\end{equation}
where the first vertical map is the natural specialisation map and the second is induced by the composition 
$$
    W_{\underline{\Pi}} \to H^{q}(\overline{Y}_{G}(K), \mathscr{A}_{\lambda, m}^{\mathrm{Iw}})[1/p
]_{\mathfrak{m}_{x}} \cong W_{x}. 
$$
\end{theorem}
As before the Hochschild--Serre spectral sequence gives a map
$$
    \psi_{\mathcal{U}}: H^{q + 1}(Y_{G}(K)_{\Sigma}, \mathscr{A}^{\mathrm{Iw}}_{\mathcal{U},m})[1/p]_{0} \to H^{1}(\mathcal{O}_{E}[\Sigma^{-1}], H^{q}(\overline{Y}_{G}(K), \mathscr{A}^{\mathrm{Iw}}_{\mathcal{U}, m})[1/p])
$$
where $H^{q + 1}(Y_{G}(K)_{\Sigma}, \mathscr{A}^{\mathrm{Iw}}_{\mathcal{U},m})[1/p]_{0}$ is the kernel of the base-change map
$$
    H^{q + 1}(Y_{G}(K)_{\Sigma}, \mathscr{A}^{\mathrm{Iw}}_{\mathcal{U},m})[1/p] \to H^{q + 1}(\overline{Y}_{G}(K), \mathscr{A}^{\mathrm{Iw}}_{\mathcal{U},m})[1/p].
$$
 
\begin{lemma}
Let $e^{\leq \textbf{h}} \in \Lambda_{\mathcal{U}}\{\{\mathfrak{U}_{p}^{-}\}\}$ be the slope $\leq \textbf{h}$ projector on $\mathscr{C}^{\bullet}(\overline{Y}_{G}(K), A^{\mathrm{Iw}}_{\mathcal{U}, m})[1/p]$, then 
$$
    e^{\leq \textbf{h}}H^{\bullet}(\overline{Y}_{G}(K), \mathscr{A}^{\mathrm{Iw}}_{\mathcal{U}, m})[1/p] = H^{\bullet}(\overline{Y}_{G}(K), \mathscr{A}^{\mathrm{Iw}}_{\mathcal{U}, m})[1/p]^{\leq \textbf{h}}. 
$$
\end{lemma}
\begin{proof}
  Since the differentials $d$ are continuous and $\mathfrak{U}_{p}$-equivariant and since $e^{\leq \textbf{h}}$ converges on each $\mathscr{C}^{i}(Y_{G}(K), A^{\mathrm{Iw}}_{\mathcal{U}, m})$ then $d(e^{\leq \textbf{h}} x) = e^{\leq \textbf{h}}d(x)$ and so $e^{\leq \textbf{h}}$ preserves cocycles and coboundaries and by Lemma \ref{lem:shrinklem}
  $$
    e^{\leq \textbf{h}}Z(Y_{G}(K), A^{\mathrm{Iw}}_{\mathcal{U}, m})[1/p] = Z(Y_{G}(K), A^{\mathrm{Iw}}_{\mathcal{U}, m})[1/p]^{\leq \textbf{h}}
  $$
  and
  $$
  e^{\leq \textbf{h}}B(\overline{Y}_{G}(K), A^{\mathrm{Iw}}_{\mathcal{U}, m})[1/p] = B(\overline{Y}_{G}(K), A^{\mathrm{Iw}}_{\mathcal{U}, m})[1/p]^{\leq \textbf{h}}
  $$
  and the result follows by comparing Betti and \'etale cohomology. 
\end{proof}
\begin{proposition} \label{prop:abjac}
Shrinking $\mathcal{U}$ if necessary, there is an element 
$$
    T\in \mathbb{T}^{-}_{S, p}\hat{\otimes} \Lambda_{\mathcal{U}}
$$
such that 
$$
T \cdot H^{q + 1}(Y_{G}(K)_{\Sigma}, \mathscr{A}^{\mathrm{Iw}}_{\mathcal{U},m})[1/p] \subset H^{q + 1}(Y_{G}(K)_{\Sigma}, \mathscr{A}^{\mathrm{Iw}}_{\mathcal{U},m})[1/p]_{0}.
$$
and whose image in $\mathcal{T}_{\mathcal{U}, \textbf{h}}$ is the element $\tilde{f}$ from Section \ref{sec:eigenvar}.
\end{proposition}
\begin{proof}
The series $e^{\leq \textbf{h}}$ is a $p$-adic limit of polynomials. Indeed, by construction $e^{\leq \textbf{h}}$ converges on an affinoid $\mathcal{V} \times \mathbb{A}^{n}_{\mathrm{rig}}$ containing $\mathcal{U} \times \mathbb{A}^{n}_{\mathrm{rig}}$ and thus we can write $e^{\leq \textbf{h}}$ as a $p$-adic limit of $\mathcal{O}(\mathcal{V})$-coefficient polynomials in the operators $U_{1}, \ldots, U_{n}$. Since the terms in this series tend to $0$ $p$-adically we can find $C \in \Q$ such that  $\left(e^{\leq \textbf{h}}\right)^{\circ} = p^{C}e^{\leq \textbf{h}}$ is optimally integrally normalised in the sense that $e^{\leq \textbf{h}} \in \mathcal{O}(\mathcal{V})^{\circ}\{\{U_{1}, \ldots, U_{n} \}\} \subset \Lambda_{\mathcal{U}}\{\{U_{1}, \ldots, U_{n}\}\}$ and $e^{\leq \textbf{h}} \not\equiv 0 \ \mathrm{mod} \ p$.  

We claim that $\left(e^{\leq \textbf{h}}\right)^{\circ}$ acts on $H^{q + 1}(Y_{G}(K)_{\Sigma}, A^{\mathrm{Iw}}_{\mathcal{U},m})$. Indeed, it suffices to show that $H^{q + 1}(Y_{G}(K)_{\Sigma}, A^{\mathrm{Iw}}_{\mathcal{U},m})$ is $p$-adically complete, but this follows from the fact that $p^{n} \in \mathrm{Fil}_{\mathcal{U}, m}^{n}$ for all $n \geq 0$ and 
$$
    H^{q + 1}(Y_{G}(K)_{\Sigma}, \mathscr{A}^{\mathrm{Iw}}_{\mathcal{U},m}) = \varprojlim_{n}H^{q + 1}(Y_{G}(K)_{\Sigma}, \mathscr{A}^{\mathrm{Iw}}_{\mathcal{U},m}/\mathrm{Fil}_{\mathcal{U}, m}^{n}).
$$
Since $\left(e^{\leq \textbf{h}}\right)^{\circ} $ is $\Lambda_{\mathcal{U}}$-linear, this action induces an action on $H^{q + 1}(Y_{G}(K)_{\Sigma}, \mathscr{A}^{\mathrm{Iw}}_{\mathcal{U},m})[1/p]$. Rescaling by $p^{-C}$ we obtain an action of $e^{\leq \textbf{h}}$. Since the base-change map is Hecke equivariant and $p$-adically-continuous we obtain a map 
$$
    e^{\leq\textbf{h}}H^{q + 1}(Y_{G}(K)_{\Sigma}, \mathscr{A}^{\mathrm{Iw}}_{\mathcal{U},m}) \to H^{q + 1}(\overline{Y}_{G}(K), \mathscr{A}^{\mathrm{Iw}}_{\mathcal{U},m})^{\leq \textbf{h}}.
$$
 As in Section \ref{sec:eigenvar}, shrinking $\mathcal{U}$ if necessary, we let $f \in \mathbb{T}^{-}_{S, p} \otimes \Lambda_{\mathcal{U}}$ be an element mapping to the idempotent inducing the projection $\mathcal{T}_{\mathcal{U}, \textbf{h}} \to \mathcal{O}(\underline{\Pi}) \cong \mathcal{O}(\mathcal{U})$. We then have
$$
    H^{q + 1}(\overline{Y}_{G}(K), \mathscr{A}^{\mathrm{Iw}}_{\mathcal{U},m})[1/p]^{\leq \textbf{h}}\otimes_{\mathcal{T}_{\mathcal{U}, \textbf{h}}} \mathcal{O}(\underline{\Pi})  =  f \cdot H^{q + 1}(\overline{Y}_{G}(K), \mathscr{A}^{\mathrm{Iw}}_{\mathcal{U},m})[1/p]^{\leq \textbf{h}} = 0.
$$
  Setting $T = f \cdot e^{\leq \textbf{h}} $, we thus have:
$$
    T \cdot H^{q + 1}(Y_{G}(K)_{\Sigma}, \mathscr{A}^{\mathrm{Iw}}_{\mathcal{U},m})[1/p] \subset  H^{q + 1}(Y_{G}(K)_{\Sigma}, \mathscr{A}^{\mathrm{Iw}}_{\mathcal{U},m})[1/p]_{0}.
$$
Moreover, since $e^{\leq \textbf{h}}$ maps to the identity in $\mathcal{T}_{\mathcal{U}, \textbf{h}}$ we have that $T$ maps to the image of $f$ which is $\tilde{f}$ by definition.
\end{proof}

\begin{definition} \label{def:abeljacobi}
We define the \textit{slope $\leq \textbf{h}$ overconvergent Abel--Jacobi map} 
$$
    \mathrm{AJ}^{\mathrm{oc}}_{\underline{\Pi}}: H^{q + 1}(Y_{G}(K)_{\Sigma}, \mathscr{A}^{\mathrm{Iw}}_{\mathcal{U},m})[1/p] \to H^{1}(\mathcal{O}_{E}[\Sigma^{-1}], W_{\underline{\Pi}})
$$
to be the map $x \mapsto \psi_{\mathcal{U}}\left(T \cdot x\right)$.
\end{definition}

It remains to show the compatibility of $\mathrm{AJ}_{\underline{\Pi}}^{\mathrm{oc}}$ with $\mathrm{AJ}_{x}^{\mathrm{an}}$. Let 
$$
\phi_{x}^{(q + 1)}:  H^{q + 1}(Y_{G}(K)_{\Sigma}, \mathscr{A}^{\mathrm{Iw}}_{\mathcal{U},m})[1/p] \to  H^{q + 1}(Y_{G}(K)_{\Sigma}, \mathscr{A}^{\mathrm{Iw}}_{\lambda,m})[1/p]_{\mathfrak{m}_{x}}
$$
be the composition of natural specialisation map with the map to the localisation at $\mathfrak{m}_{x}$ and let 
$$
    \phi^{(q)}_{x}: H^{1}(\mathcal{O}_{E}[\Sigma^{-1}
    ], W_{\underline{\Pi}}) \to H^{1}(\mathcal{O}_{E}[\Sigma^{-1}], W_{x})
$$
be the map induced by the natural specialisation map $W_{\underline{\Pi}} \to W_{x}$. The notation $\phi^{(\bullet)}_{x}$ refers to the fact that both of these maps are induced on different degrees of cohomology by the same specialisation map $A^{\mathrm{Iw}}_{\mathcal{U},m} \to A_{\lambda, m}^{\mathrm{Iw}}$ and as a result of this the diagram 
$$
    \begin{tikzcd}
        H^{q + 1}(Y_{G}(K)_{
    \Sigma}, \mathscr{A}^{\mathrm{Iw}}_{\mathcal{U},m})[1/p]_{0} \arrow[r, "\mathrm{AJ}"] \arrow[d, ""] &  H^{1}(\mathcal{O}_{E}[\Sigma^{-1}], H^{q}(\overline{Y}_{G}(K)_{
    \Sigma}, \mathscr{A}^{\mathrm{Iw}}_{\mathcal{U},m})[1/p]) \arrow[d, ""] \\ 
     H^{q + 1}(Y_{G}(K)_{\Sigma}, \mathscr{A}_{\lambda, m}^{\mathrm{Iw}} )_{0}\arrow[r, "\mathrm{AJ}"] &  H^{1}(\mathcal{O}_{E}[\Sigma^{-1}], H^{q}(\overline{Y}_{G}(K), \mathscr{A}_{\lambda, m}^{\mathrm{Iw}} ))
    \end{tikzcd}
$$
commutes by functoriality of the Hochschild--Serre spectral sequence, with the bottom map inducing $\mathrm{AJ}_{x}^{\mathrm{an}}$ on the localisation at $\mathfrak{m}_{x}$. We thus have
\begin{align*}
    \phi^{(q)}_{x}\left(\mathrm{AJ}_{\underline{\Pi}}^{\mathrm{oc}}(z)\right) &= \phi^{(q)}_{x}\left(\mathrm{AJ}\left(\left(f \cdot e^{\leq \textbf{h}}\right)z\right)\right) \\
    &= \mathrm{AJ}_{x}^{\mathrm{an}}\left(\phi_{x}^{(q + 1)}\left((f \cdot e^{\leq \textbf{h}}) z\right)\right) \\
    &= \mathrm{AJ}_{x}^{\mathrm{an}}\left(\phi_{x}^{(q + 1)}\left(z\right)\right)
\end{align*}
where the last line follows because the Hecke operator $f \cdot e^{\leq \textbf{h}}$ acts trivially on $W_{x}$ and thus also on its Galois cohomology. 
\subsection{Compatibility of Abel-Jabcobi maps and norm compatible classes} \label{sec:compat}

We show that our constructions allow us to construct norm compatible elements in Galois cohomology. More specifically, throughout the constructions of the previous sections we had to routinely shrink the subspaces of $\mathcal{W}_{G}$ over which we were working. We show that, by using the `abelian tower' construction of \cite[Section 4.6]{loefflerspherical} we can apply this shrinking uniformly in a tower of abelian extensions of the reflex field $E$. 

Let $(G, H)$ be a spherical pair and let $C_{H} = H/H^{\mathrm{der}}$ denote the maximal torus quotient of $H$. The construction of \textit{op. cit.} uses modified data $(\tilde{G}, \tilde{H}) = (G \times C_{H}, H)$ in order to construct classes over towers of Shimura varieties associated to neat open-compact subgroups with $p$-part $J_{G} \times C_{H,n} \subset \tilde{G}(\Zp)$, where $C_{H,n} = \{c \in C_{H}: c \equiv \ 1 \ \mathrm{mod} \ p^{n}\}$. We consider the (finite) Shimura set given by
$$
    \Delta_{n} = \mathcal{C}_{H}(\Q)\backslash \mathcal{C}_{H}(\A)/C_{H,n} \cdot \mathcal{C}^{p} \cdot \mathcal{C}_{H}(\R)^{\dagger}
$$
where $\mathcal{C}_{H}$ is the $\Q$-group $\mathcal{H}/\mathcal{H}^{\mathrm{der}}$, $\mathcal{C}^{p}$ is the maximal compact subgroup of $\mathcal{C}_{H}(\A^{(p)}_{f})$ and $\mathcal{C}_{H}(\R)^{\dagger}$ is the subset of the components of $\mathcal{C}(\R)$ in the image of a maximal compact subgroup $K_{\mathcal{H}, \infty} \subset \mathcal{C}(\R)$. The Galois action on $\Delta_{n}$ is given by a character 
$$
    \kappa_{n}: \mathrm{Gal}(\overline{E}/E)^{\mathrm{ab}} \to \Delta_{n}
$$
determined by the Shimura datum for $\tilde{G}$ which cuts out an abelian extension $E_{n}/E$. We write $M_{n} \subset \Delta_{n}$ for the image of $\kappa_{n}$. For any $\mathcal{O}$-linear \'etale sheaf $\mathscr{F}$ we have a Galois equivariant isomorphism
$$
    H^{i}(\overline{Y}_{\tilde{G}}(J_{G} \times C_{H, n}), \mathscr{F}) \cong H^{i}(\overline{Y}_{G}(J_{G}), \mathscr{F}) \otimes_{\mathcal{O}} \mathcal{O}[\Delta_{n}].
$$
\begin{proposition} \label{prop:ajcompat}
Let $\textbf{h} \in \Q^{d}_{\geq 0}$, $\mathcal{U}' \subset \mathcal{W}_{G}$ a wide-open disc adapted to $\textbf{h}$, and let $x \in \mathcal{E}_{\mathcal{U}', \textbf{h}}$ be a really nice point with a Coleman family $\underline{\Pi}$ passing through it. There is a wide-open disc $\mathcal{U} \subset \mathcal{U}'$ such that for all $n \geq 1$ we can construct Abel--Jacobi maps 
$$
    \mathrm{AJ}_{\underline{\Pi}, n}^{\mathrm{oc}}: H^{q + 1}(Y_{\tilde{G}}(J_{G} \times C_{H, n})_{\Sigma}, A^{\mathrm{Iw}}_{\mathcal{U}, m})[1/p] \to H^{1}(\mathcal{O}_{E_{n}}[\Sigma^{-1}], W_{\underline{\Pi}}) \otimes_{\mathcal{O}[M_{n}]}\mathcal{O}[\Delta_{n}]
$$
fitting into the commutative diagram \eqref{fig:abjac}. 
\end{proposition}
\begin{proof}
In Section \ref{sec:abeljac} we construct the maps $\mathrm{AJ}_{\underline{\Pi}}^{\mathrm{oc}}$ by constructing a Hecke operator annihilating the cohomology group $H^{q + 1}(Y_{G}(K), \mathscr{A}_{\mathcal{U}, m}^{\mathrm{Iw}})$, possibly after shrinking $\mathcal{U}$. Let $\tilde{\mathbb{T}}^{-}_{S, p}$ denote the Hecke algebra given by the restricted tensor product of spherical Hecke algebras away from $S \cup \{p\}$ and $A^{-}_{\tilde{G}}$ at $p$. In order to prove the proposition it suffices to find a Hecke operator $T \in \tilde{\mathbb{T}}^{-}_{S, p} \hat{\otimes} \Lambda_{\mathcal{U}}$ which annihilates the cohomology group $H^{q + 1}(Y_{G}(J_{G} \times C_{H, n}), \mathscr{A}_{\mathcal{U}, m}^{\mathrm{Iw}})$ for all $n \geq 1$, possibly after shrinking $\mathcal{U}$. But since $H^{q + 1}(Y_{G}(J_{G} \times C_{H, n}), \mathscr{A}_{\mathcal{U}, m}^{\mathrm{Iw}}) \cong H^{q + 1}(Y_{G}(J_{G}), \mathscr{A}_{\mathcal{U}, m}^{\mathrm{Iw}}) \otimes_{\mathcal{O}} \mathcal{O}[\Delta_{n}]$ and there is an isomorphism $\mathbb{T}^{-}_{S, p} \otimes \mathbb{T}(C_{H}) \cong \tilde{\mathbb{T}}^{-}_{S, p}$ where $\mathbb{T}(C_{H})$ consists of Hecke operators on $C_{H}$, and the natural inclusion $G \to \tilde{G}$ induces a map $\mathbb{T}^{-}_{S, p}  \to \tilde{\mathbb{T}}^{-}_{S, p}$ sending $X \mapsto X \otimes 1$. The image of the natural map 
$$
    \mathbb{T}^{-}_{S, p} \hat{\otimes} \Lambda_{\mathcal{U}} \to \tilde{\mathbb{T}}^{-}_{S, p} \hat{\otimes} \Lambda_{\mathcal{U}} \to \mathrm{End}_{\Lambda_{\mathcal{U}}}( H^{q + 1}(Y_{G}(J_{G}), \mathscr{A}_{\mathcal{U}, m}^{\mathrm{Iw}}) \otimes_{\mathcal{O}} \mathcal{O}[\Delta_{n}])
$$
preserves each summand $ H^{q + 1}(Y_{G}(J_{G}), \mathscr{A}_{\mathcal{U}, m}^{\mathrm{Iw}}) \otimes_{\mathcal{O}} \mathcal{O}[\delta]$ corresponding to a class $[\delta] \in \Delta_{n}$ and acts via its usual action on $ H^{q + 1}(Y_{G}(J_{G}), \mathscr{A}_{\mathcal{U}, m}^{\mathrm{Iw}})$. The upshot of this is that we can apply the construction of Proposition \ref{prop:abjac} to obtain (after possibly shrinking $\mathcal{U}$) a Hecke operator $T \in \mathbb{T}^{-}_{S, p} \hat{\otimes} \Lambda_{\mathcal{U}}$ annihilating $H^{q + 1}(Y_{G}(J_{G}), \mathscr{A}_{\mathcal{U}, m}^{\mathrm{Iw}})$ whose image in $\tilde{\mathbb{T}}^{-}_{S, p} \hat{\otimes} \Lambda_{\mathcal{U}}$ therefore annihilates $H^{q + 1}(Y_{G}(J_{G}), \mathscr{A}_{\mathcal{U}, m}^{\mathrm{Iw}}) \otimes_{\mathcal{O}} \mathcal{O}[\Delta_{n}]$. Since we only applied the construction of Proposition \ref{prop:abjac} at level $K^{p}\cdot J_{G}$ this construction is independent of $n$.
\end{proof}
We have a parabolic subgroup $\tilde{Q}_{G} = Q_{G} \times C_{H} \subset \tilde{G}$ with Levi decomposition $\tilde{Q}_{G} = \tilde{L}_{G} \times N_{G} = \left(L_{G} \times C_{H}\right) \times N_{G}$ where we have embedded $N_{G} \to \tilde{G}$ by mapping $n \mapsto (n, 1)$. If we define $\tilde{L}_{G, n} = L_{G, n} \times C_{H, n}$ and $\tilde{V}_{n} = \overline{N}_{n} \times \tilde{L}_{G, n} \times N_{G}$ then $\tilde{V}_{n}$ and $J_{G} \times C_{H, n}$ both have Iwahori decompositions with respect to $\tilde{Q}_{G}$ and so both sides of the degeneracy maps
$$
    \phi_{n}: H^{i}(Y_{\tilde{G}}(\tilde{V}_{n})_{\Sigma}, \mathscr{A}^{\mathrm{Iw}}_{\mathcal{U}, m}) \to H^{i}(Y_{\tilde{G}}(J_{G} \times C_{H, n})_{\Sigma}, \mathscr{A}^{\mathrm{Iw}}_{\mathcal{U}, m})
$$
receive an action of $\mathfrak{U}^{-}_{p}$ and the maps $\phi_{n}$ are equivariant for this action.
\begin{proposition}
With notation as above, suppose we have a collection of elements $z^{H}_{n} \in  H^{q - c + 1}_{\mathrm{Iw}}(Y_{H}(Q_{H}^{0} \cap u^{-1}U_{n}u) , \mathcal{O})$ which are compatible under the natural degeneracy maps 
$$
     H^{q - c + 1}_{\mathrm{Iw}}(Y_{H}(Q_{H}^{0} \cap u^{-1}U_{n + 1}u) , \mathcal{O}) \to  H^{q - c + 1}_{\mathrm{Iw}}(Y_{H}(Q_{H}^{0} \cap u^{-1}U_{n}u) , \mathcal{O}),
$$
where as usual $c = \mathrm{dim}_{\R}Y_{G} - \mathrm{dim}_{\R}Y_{H}$. Then by Proposition \ref{prop:ajcompat} for non-critical $\textbf{h} \in \Q_{\geq 0}^{d}$ it is possible to choose $\mathcal{U}$ such that for all $n \geq 1$ the Abel--Jacobi map $\mathrm{AJ}^{\mathrm{oc}}_{\underline{\Pi}, n}$ is well defined. Writing 
$$
    z_{\underline{\Pi}, n}^{\mathrm{Gal}} = \mathrm{AJ}^{\mathrm{oc}}_{\underline{\Pi}, n} \circ \phi_{n} \circ \iota_{\mathcal{U}, n, *}\left(z^{H}_{n} \right) \in H^{1}(\mathcal{O}_{E_{n}}[
    \Sigma^{-1}], W_{\underline{\Pi}}(c))  \otimes_{\mathcal{O}[M_{n}]}\mathcal{O}[\Delta_{n}],
$$
then, letting $\beta_{n  +1}: \Delta_{n + 1} \to \Delta_{n}$ denote the natural quotient map, we have
$$
    \left(\mathrm{cores}^{n + 1}_{n} \otimes \beta_{n + 1}\right)\left(z_{\mathcal{U}, n + 1}^{\mathrm{Gal}} \right) = U'_{p}\cdot z_{\mathcal{U}, n}^{\mathrm{Gal}}.
$$
\end{proposition} 
\begin{proof}
    The classes $\xi_{\mathcal{U}, n} := \iota_{\mathcal{U}, n, *}\left(z^{H}_{n} \right)$ satisfy the relation 
    $$
        \mathrm{pr}^{n + 1}_{n, *}\left(\xi_{\mathcal{U}, n + 1}\right) = U'_{p} \cdot \xi_{\mathcal{U}, n}
    $$
    by Theorem \ref{thm:normrelation}. Since the maps $\phi_{n}$ are $\mathfrak{U}'_{p}$-equivariant we have $\mathrm{pr}^{n + 1}_{n, *} \circ \phi_{n + 1}(\xi_{\mathcal{U}, n + 1}) = U'_{p} \cdot \phi_{n}(\xi_{\mathcal{U}, n})$. Finally, the degeneracy map 
    $$
        H^{i}(Y_{G}(J_{G} \times C_{H, n + 1})_{\Sigma}, \mathscr{A}^{\mathrm{Iw}}_{\mathcal{U}, m}) \to H^{i}(Y_{G}(J_{G} \times C_{H, n + 1})_{\Sigma}, \mathscr{A}^{\mathrm{Iw}}_{\mathcal{U}, m})
    $$
    corresponds under the Hochschild--Serre spectral sequence to the map 
    $$
        H^{1}(E, W_{\underline{\Pi}} \otimes \mathcal{O}[\Delta_{n + 1}]) \to H^{1}(E, W_{\underline{\Pi}} \otimes \mathcal{O}[\Delta_{n}])
    $$
    induced by $\beta_{n + 1}$. Thus it suffices to show that this map is indeed $\mathrm{cores}^{n + 1}_{n} \otimes \beta_{n + 1}$.
    As a $G_{E}$-set we have $\Delta_{n} = M_{n} \times M_{n} \backslash \Delta_{n}$ and thus $H^{1}(E, W_{\underline{\Pi}} \otimes \mathcal{O}[\Delta_{n}]) = H^{1}(E, W_{\underline{\Pi}} \otimes_{\mathcal{O}} \mathcal{O}[M_{n}]) \otimes_{\mathcal{O}}\mathcal{O}[M_{n}\backslash \Delta_{n}]$. We have $W_{\underline{\Pi}} \otimes_{\mathcal{O}}\mathcal{O}[M_{n}] = \mathrm{Ind}_{G_{E}}^{G_{E_{n}}}\left(W_{\underline{\Pi}}\right)$ and the map $M_{n + 1} \to M_{n}$ induced by restricting $\beta_{n + 1}$ corresponds to the natural restriction map $\mathrm{Gal}(E_{n + 1}/E) \to \mathrm{Gal}(E_{n}/E)$. Thus the map 
    $$
        H^{1}(E, W_{\underline{\Pi}} \otimes \mathcal{O}[\Delta_{n + 1}]) \to H^{1}(E, W_{\underline{\Pi}} \otimes \mathcal{O}[\Delta_{n}])
    $$
    induced by $\beta_{n + 1}$ corresponds to a map 
    $$
        H^{1}(E, \mathrm{Ind}_{G_{E}}^{G_{E_{n + 1}}}\left(W_{\underline{\Pi}}\right)) \otimes_{\mathcal{O}} \mathcal{O}[M_{n + 1} \backslash \Delta_{n + 1}] \to  H^{1}(E, \mathrm{Ind}_{G_{E}}^{G_{E_{n}}}\left(W_{\underline{\Pi}}\right)) \otimes_{\mathcal{O}} \mathcal{O}[M_{n} \backslash \Delta_{n}] 
    $$
    given by the tensor product of the natural map $\mathrm{Ind}_{G_{E}}^{G_{E_{n + 1}}}\left(W_{\underline{\Pi}}\right) \to \mathrm{Ind}_{G_{E}}^{G_{E_{n}}}\left(W_{\underline{\Pi}}\right)$ with $\beta_{n + 1}$. Applying the Shapiro's lemma isomorphism on both sides, this map becomes
    $$
        \mathrm{cores}_{n + 1}^{n} \otimes \beta_{n + 1}: H^{1}(E_{n + 1}, W_{\underline{\Pi}}) \otimes_{\mathcal{O}} \mathcal{O}[M_{n} \backslash \Delta_{n}]\to H^{1}(E_{n}, W_{\underline{\Pi}}) \otimes_{\mathcal{O}} \mathcal{O}[M_{n} \backslash \Delta_{n}],
    $$
    as desired.
\end{proof}
Let $\mathfrak{m}_{\underline{\Pi}}$ be the maximal ideal of $ \mathbb{T}_{S, p}^{-} \hat{\otimes} \Lambda_{\mathcal{U}}$ given by the kernel of the map 
$$
    \mathbb{T}^{-}_{S, p} \hat{\otimes} \Lambda_{\mathcal{U}} \to \mathcal{O}(\underline{\Pi})^{\circ}/\mathfrak{m}_{\mathcal{U}} \cong \Lambda_{\mathcal{U}}/\mathfrak{m}_{\mathcal{U}}.
$$
We can write $\mathfrak{m}_{\underline{\Pi}}$ as a product $\mathfrak{m}_{S} \cdot \mathfrak{m}_{p}$ of maximal ideals $\mathfrak{m}_{S} \subset  \mathbb{T}_{S} \hat{\otimes} \Lambda_{\mathcal{U}}$, $\mathfrak{m}_{p} \subset \mathfrak{U}^{-}_{p} \hat{\otimes} \Lambda_{\mathcal{U}}$.
In order to prove Proposition \ref{prop:bound} we need to make the following assumption: 
\begin{equation} \label{ass:bex} \tag{$\ddagger$}
\text{\textbf{Assumption:} the torsion subgroup of $\overline{H}^{q + 1}(\overline{Y}_{G}(J_{G}), \mathscr{A}^{\mathrm{Iw}}_{\mathcal{U}, m})_{\mathfrak{m}_{S}}$ has bounded exponent.}
\end{equation}
Since this assumption is rather opaque we prove the following lemma relating it to a more natural statement.
\begin{lemma} \label{lem:torvan}
Let $K = K_{p} \times \prod_{\ell \in S}G(\Z_{\ell}) \times K^{pS} \subset G(\Zp) \times \prod_{\ell \in S}G(\Z_{\ell}) \times G(\A_{f}^{pS})$ be an open-compact subgroup and suppose $\mathfrak{m}_{S}$ is such that 
$$
    H^{i}(\overline{Y}_{G}(K), \F)_{\mathfrak{m}_{S}} = 0 
$$
for $i \neq q$, where $\F := \Lambda_{\mathcal{U}}/\mathfrak{m}_{\mathcal{U}}$, a finite field. Then \eqref{ass:bex} holds and, in addition, $H^{q}(\overline{Y}_{G}(K), \mathscr{A}_{\mathcal{U}, m}^{\mathrm{Iw}}) _{\mathfrak{m}_{S}}$ is a flat $\Lambda_{\mathcal{U}}$-module. 
\end{lemma}
\begin{remark}
    The hypothesis of this lemma is designed to resemble a `torsion vanishing theorem' for `generic' (c.f. \cite[Definition 1.1]{yang2025generic}) non-Eisenstein maximal ideals of the spherical Hecke algebra, of the sort proved in, for example, \cite{carianischolze}, \cite{hamann2023torsion}, \cite{caraiani2023etale} and most recently \cite{yang2025generic} where the conjecture is proved for abelian type Shimura varieties up to a mild condition on the primes involved. 
\end{remark}
\begin{proof}
    Let $\lambda \in \mathcal{U}$ be any character. We first claim that 
    $$
        H^{i}(\overline{Y}_{G}(K), \mathscr{A}_{\lambda, m}^{\mathrm{Iw}}/(p))_{\mathfrak{m}_{S}} = 0
    $$
    for all $i \neq q$. Let $K_{p}(p) = \{k \in K_{p}: k \equiv 1 \ \mathrm{mod} \ p\}$ and let $K(p) = K_{p}(p) \times K^{p}$. Then $K(p) \subset K$ is a normal subgroup  which acts trivially on the $\F$-module $A_{\lambda, m}^{\mathrm{Iw}}/(p)$ and since $H^{i}(\overline{Y}_{G}(K(p)), \mathscr{A}_{\lambda, m}^{\mathrm{Iw}}/(p)) \cong H^{i}(\overline{Y}_{G}(K(p)), \F) \otimes_{\F} A_{\lambda, m}^{\mathrm{Iw}}/(p)$ we deduce that $H^{i}(\overline{Y}_{G}(K(p)), \mathscr{A}_{\lambda, m}^{\mathrm{Iw}}/(p)) _{\mathfrak{m}_{S}} = 0$ for $i \neq q$. By considering the spectral sequence 
    $$
        E^{i, j}_{2}: H^{i}(K/K(p), H^{j}(\overline{Y}_{G}(K(p)), \mathscr{A}_{\lambda, m}^{\mathrm{Iw}}/(p))_{\mathfrak{m}_{S}} ) \implies H^{i + j}(\overline{Y}_{G}(K), \mathscr{A}_{\lambda, m}^{\mathrm{Iw}}/(p)) _{\mathfrak{m}_{S}}
    $$
    we deduce that $H^{i}(\overline{Y}_{G}(K), \mathscr{A}_{\lambda, m}^{\mathrm{Iw}}/(p)) _{\mathfrak{m}_{S}} = 0$ for $i < q$. If we let $\left(A_{\lambda, m}^{\mathrm{Iw}}/(p)\right)^{\vee}$ denote the $\F$-dual of $A_{\lambda, m}^{\mathrm{Iw}}/(p)$ then this module is also stabilised by $K(p)$ and we can deduce the following vanishing in compactly supported cohomology $H^{i}_{c}(\overline{Y}_{G}(K), \left(\mathscr{A}_{\lambda, m}^{\mathrm{Iw}}\right)^{\vee})_{\mathfrak{m}_{S}} = 0$ for $i < q$ since $H^{i}_{c}(\overline{Y}_{G}(K), \F)_{\mathfrak{m}_{S}} = 0$ for $i \neq q$ by Poincar\'e duality. Another application of Poincar\'e duality gives that 
    $$
        H^{i}(\overline{Y}_{G}(K), \mathscr{A}_{\lambda, m}^{\mathrm{Iw}}/(p)) _{\mathfrak{m}_{S}} \cong \mathrm{Hom}_{\F}\left(H^{2q - i}_{c}(\overline{Y}_{G}(K), \left(\mathscr{A}_{\lambda, m}^{\mathrm{Iw}}\right)^{\vee})_{\mathfrak{m}_{S}}, \F\right)
    $$
    from which we deduce that $H^{i}(\overline{Y}_{G}(K), \mathscr{A}_{\lambda, m}^{\mathrm{Iw}}/(p)) _{\mathfrak{m}_{S}} = 0$ for $i \neq q$. Noting that $A_{\mathcal{U}, m}^{\mathrm{Iw}}/\mathfrak{m}_{\mathcal{U}} = A_{\lambda, m}^{\mathrm{Iw}}/(p)$, we consider the spectral sequence 
    $$
    E^{i, j}_{2}: \mathrm{Tor}_{-i}^{\Lambda_{\mathcal{U}}}(H^{i}(\overline{Y}_{G}(K), \mathscr{A}_{\mathcal{U}, m}^{\mathrm{Iw}})_{\mathfrak{m}_{S}}, \Lambda_{\mathcal{U}}/\mathfrak{m}_{\mathcal{U}}) \implies H^{i + j}(\overline{Y}_{G}(K), \mathscr{A}_{\lambda, m}^{\mathrm{Iw}}/(p)) _{\mathfrak{m}_{S}}.
    $$
    Since $H^{2q}(\overline{Y}_{G}(K), \mathscr{A}_{\lambda, m}^{\mathrm{Iw}}/(p)) _{\mathfrak{m}_{S}} = 0$ we immediately see that $H^{2q}(\overline{Y}_{G}(K), \mathscr{A}_{\mathcal{U}, m}^{\mathrm{Iw}}) _{\mathfrak{m}_{S}} \otimes_{\Lambda_{\mathcal{U}}}\Lambda_{\mathcal{U}}/\mathfrak{m}_{\mathcal{U}} = 0$ and thus $H^{2q}(\overline{Y}_{G}(K), \mathscr{A}_{\mathcal{U}, m}^{\mathrm{Iw}}) _{\mathfrak{m}_{S}} = 0$ by the topological Nakayama lemma (see the first theorem of \cite[\S3]{balister1997note} for the statement we use here), where we use that cohomology with $\mathscr{A}^{\mathrm{Iw}}_{\mathcal{U}, m}$-coefficients is profinite. Then $$
    \mathrm{Tor}_{-i}^{\Lambda_{\mathcal{U}}} (H^{2q}(\overline{Y}_{G}(K), \mathscr{A}_{\mathcal{U}, m}^{\mathrm{Iw}}) _{\mathfrak{m}_{S}} , \Lambda_{\mathcal{U}}/\mathfrak{m}_{\mathcal{U}}) = 0
    $$
    for all $i$ and so the row $E^{i, 2q}_{2}$ vanishes. We can iterate the above process to get that 
    $$
        H^{i}(\overline{Y}_{G}(K), \mathscr{A}_{\mathcal{U}, m}^{\mathrm{Iw}}) _{\mathfrak{m}_{S}} = 0
    $$
    for $i > q$. Since we have an injection 
    $$
        \mathrm{Tor}_{1}^{\Lambda_{\mathcal{U}}}(H^{q}(\overline{Y}_{G}(K), \mathscr{A}_{\mathcal{U}, m}^{\mathrm{Iw}}) _{\mathfrak{m}_{S}} ,\Lambda_{\mathcal{U}}/\mathfrak{m}_{\mathcal{U}}) \hookrightarrow H^{q - 1}(\overline{Y}_{G}(K), \mathscr{A}_{\lambda, m}^{\mathrm{Iw}}/(p)) _{\mathfrak{m}_{S}} = 0
    $$
    we have $\mathrm{Tor}_{1}^{\Lambda_{\mathcal{U}}}(H^{q}(\overline{Y}_{G}(K), \mathscr{A}_{\mathcal{U}, m}^{\mathrm{Iw}}) _{\mathfrak{m}_{S}} ,\Lambda_{\mathcal{U}}/\mathfrak{m}_{\mathcal{U}}) = 0$ and thus by the local criterion for flatness we see that $\mathrm{Tor}_{1}^{\Lambda_{\mathcal{U}}}(H^{q}(\overline{Y}_{G}(K), \mathscr{A}_{\mathcal{U}, m}^{\mathrm{Iw}}) _{\mathfrak{m}_{S}} ,\Lambda_{\mathcal{U}}/\mathfrak{m}_{\mathcal{U}}) = 0$ for all $i > 0$ and thus $H^{q}(\overline{Y}_{G}(K), \mathscr{A}_{\mathcal{U}, m}^{\mathrm{Iw}}) _{\mathfrak{m}_{S}}$ is a flat $\Lambda_{\mathcal{U}}$-module. It now follows (again from topological Nakayama) that $H^{q - 1}(\overline{Y}_{G}(K), \mathscr{A}_{\mathcal{U}, m}^{\mathrm{Iw}}) _{\mathfrak{m}_{S}} = 0$ and we can iterate the process to get that $H^{i}(\overline{Y}_{G}(K), \mathscr{A}_{\mathcal{U}, m}^{\mathrm{Iw}}) _{\mathfrak{m}_{S}} = 0$ for $i < q$ as well. 
\end{proof} 
\begin{proposition} \label{prop:bound}
Suppose the maximal ideal $\mathfrak{m}_{S}$ of $\mathbb{T}_{S} \hat{\otimes} \Lambda_{\mathcal{U}}$  satisfies assumption \footnote{We consider the action of $\mathbb{T}_{S} \hat{\otimes} \Lambda_{\mathcal{U}}$ on mod $p$ cohomology via the quotient map $\mathbb{T}_{S} \hat{\otimes} \Lambda_{\mathcal{U}} \to\mathbb{T}_{S} \otimes \Lambda_{\mathcal{U}}/\mathfrak{m}_{\mathcal{U}}$} \eqref{ass:bex}. Then there is $D \in \Q$ independent of $n \geq 1$ such that
$$
    z^{\mathrm{Gal}}_{\underline{\Pi}, n} \in p^{D}H^{1}(\mathcal{O}_{E_{n}, \Sigma}, T_{\underline{\Pi}}(c)) \otimes_{\mathcal{O}} \mathcal{O}[M_{n} \backslash \Delta]
$$
considered as a submodule of $H^{1}(\mathcal{O}_{E_{n}, \Sigma}, W_{\underline{\Pi}}(c)) \otimes_{\mathcal{O}} \mathcal{O}[M_{n} \backslash \Delta]$, where $T_{\underline{\Pi}}$ is the image of $H^{q}(\overline{Y}_{G}(J_{G}), \mathscr{A}^{\mathrm{Iw}}_{\mathcal{U}, m})$ in $W_{\underline{\Pi}}$.
\end{proposition}
\begin{proof}
Write $H^{q + 1, \circ}_{\Sigma, n} := H^{q + 1}(Y_{\tilde{G}}(J_{G} \times C_{n})_{\Sigma}, \mathscr{A}^{\mathrm{Iw}}_{\mathcal{U}, m})$ and $H^{q + 1}_{\Sigma, n} = H^{q + 1}(Y_{\tilde{G}}(J_{G} \times C_{n})_{\Sigma}, \mathscr{A}^{\mathrm{Iw}}_{\mathcal{U}, m})[1/p]$. The lattice given by the image of the natural map $f: H^{q + 1, \circ}_{\Sigma, n} \to H^{q + 1}_{\Sigma, n}$ is preserved by the Hecke algebra $\mathbb{T}_{S, p}^{-} \hat{\otimes}\Lambda_{\mathcal{U}}$. Let $C \in \Q$ be such that $T \in p^{C}\left(\mathbb{T}_{S, p}^{-} \hat{\otimes}\Lambda_{\mathcal{U}}\right)$. By functoriality of the Hochschild--Serre spectral sequence we have a diagram
\begin{equation} \label{fig:bound}
\begin{tikzcd} 
    H^{q + 1, \circ}_{\Sigma, n} \arrow[r, "f"] & H^{q + 1}_{\Sigma, n} \arrow[d, "T"] \\
    H^{q + 1, \circ}_{\Sigma,n, 0} \arrow[u, ""] \arrow[d, ""] \arrow[r, "g"] & H^{q + 1}_{\Sigma,n, 0} \arrow[d, ""] \\
    H^{1}(\mathcal{O}_{E_{n}}[\Sigma^{-1}], T_{\underline{\Pi}}) \otimes_{\mathcal{O}}\mathcal{O}[M_{n}\backslash\Delta]\arrow[r, "h"] & H^{1}(\mathcal{O}_{E_{n}}[\Sigma^{-1}], W_{\underline{\Pi}}) \otimes_{K}K[M_{n}\backslash\Delta],
\end{tikzcd}
\end{equation}
where $H^{q + 1, \circ}_{\Sigma, n, 0}, H^{q + 1}_{\Sigma, n, 0}$ denotes the cohomologically trivial parts of $H^{q + 1, \circ}_{\Sigma, n}, H^{q + 1}_{\Sigma, n}$ respectively, and $g$ the restriction of $f$ to these submodules, $h$ is induced by the natural map $T_{\underline{\Pi}} \to W_{\underline{\Pi}}$, the unlabelled upward vertical arrow is the natural inclusion and the unlabelled downward vertical arrows are the natural maps arising from the Hochschild--Serre spectral sequence. We claim that $T$ sends $f\left(H^{q + 1, \circ}_{\Sigma, n}\right)$ into $p^{D}g\left( H^{q + 1, \circ}_{\Sigma,n, 0}\right)$ for $D \in \Z$ independent of $n$. Since $p^{-C}T$ preserves $ H^{q + 1, \circ}_{\Sigma,n}$, given $x \in  H^{q + 1, \circ}_{\Sigma,n}$, and writing $p^{-C}T(x) = y \in  H^{q + 1, \circ}_{\Sigma,n}$, we have $f(y) \in  H^{q + 1}_{\Sigma,n, 0}$. Consider the commutative diagram 
$$
\begin{tikzcd}
     H^{q + 1, \circ}_{\Sigma,n}/H^{q + 1, \circ}_{\Sigma,n, 0} \arrow[d, hook, "BC^{\circ}"] \arrow[r, "f"] &  H^{q + 1}_{\Sigma,n}/H^{q + 1}_{\Sigma,n, 0} \arrow[d, hook, "BC"] \\ 
     \overline{H}_{n}^{q + 1, \circ} \arrow[r, ""] & \overline{H}^{q + 1}_{n}
\end{tikzcd}
$$
where $\overline{H}_{n}^{q + 1, \circ} := H^{q + 1}(\overline{Y}_{G}(J_{G} \times C_{n}), \mathscr{A}^{\mathrm{Iw}}_{\mathcal{U}, m})$, $\overline{H}_{n}^{q + 1} := \overline{H}_{n}^{q + 1, \circ}[1/p]$, and $BC, BC_{0}$ are induced by the base change maps. We deduce that there is $N \in \Z$ such that $p^{N}y \in H^{q + 1, \circ}_{\Sigma,n, 0}$ and moreover that $N$ is bounded by the exponent of the torsion subgroup $\overline{H}_{n}^{q + 1, \circ}$, which is finite by our assumption \eqref{ass:bex} and independent of $n$ since $\overline{H}^{q + 1, \circ}_{n} \cong \oplus_{M_{n} \backslash \Delta}\overline{H}^{q + 1}_{0}$. Thus by taking $N \gg 0$ we ensure that for any $x$ as above we have $p^{-C}T\left(f(x)\right) = p^{-C - N}g\left(p^{N}y\right)$ with $p^{N}y \in H^{q + 1, \circ}_{\Sigma,n, 0}$. Taking $D = -C - N$ we conclude that $f\left(H^{q + 1, \circ}_{\Sigma,n}\right) \subset p^{D}g\left(H^{q + 1, \circ}_{\Sigma,n, 0}\right)$ with $D$ independent of $n$. The result follows from the commutativity of the bottom square of  \eqref{fig:bound}.

\end{proof}
\section{Criterion for \texorpdfstring{$Q_{H}^{0}$}{TEXT}-admissibility and extra variables} \label{sec:superf}
\subsection{Criterion for \texorpdfstring{$Q_{H}^{0}$}{TEXT}-admissibility} \label{sec:admiss}
Let $\mu \in X_{+}^{\bullet}(S_{H}), \lambda \in X_{+}^{\bullet}(S_{G})$ and suppose there is an injective $H$-map
$$
    V^{H}_{\mu} \to V^{G}_{\lambda}.
$$
Is there a character $\chi \in X^{\bullet}(G)$ such that $\lambda + \chi \in X^{\bullet}_{+}(S_{G})^{Q_{H}^{0}}$? 
\begin{lemma} \label{lem:admiss}
Suppose $\mu$ is trivial on $Q_{H}^{0} \cap G^{\mathrm{der}}$ and the maximal torus quotient $C_{G}$ of $G$ is split. Then there is $\chi_{\mu} \in X^{\bullet}(G)$ such that $(V_{\lambda}^{G} \otimes \chi_{\mu}^{-1})^{Q_{H}^{0}} \neq \{0\}$.
\end{lemma}
\begin{proof}
We have an injection 
$$
    Q_{H}^{0}/(Q_{H}^{0} \cap G^{\mathrm{der}}) \to C_{G}
$$
and thus $ Q_{H}^{0}/(Q_{H}^{0} \cap G^{\mathrm{der}})$ is a split torus. The restriction of $\mu$ to $Q_{H}^{0}$ lifts (non-uniquely) to a character $\chi_{\mu}$ of $G$ whence the result follows.
\end{proof}

The assumptions in Lemma \ref{lem:admiss} won't hold in every case, so we describe a process to increase the number of $Q_{H}^{0}$ admissible weights by substituting the pair $(G, H)$ for a slightly modified pair $(\tilde{G}, \tilde{H})$.

Assume $S^{0}_{H}$ satisfies Milne's assumption (SV5). This is equivalent to the real points of the subgroup 
$$
    (S_{H}^{0})^{a} = \cap_{\chi \in X^{\bullet}(S_{H})} \mathrm{Ker}(\chi)
$$
being compact \footnote{We emphasise again that we are only imposing this assumption for convenience.}. 
For an algebraic group $M$ define $\tilde{M} := M \times S_{H}^{0}$ and let 
$$
    Q_{\tilde{H}}^{0} := \{(h, \overline{h}) \in Q_{H}^{0} \times S_{H}^{0}\}
$$
where $\overline{h}$ denotes the image of $h$ in $S_{H}^{0}$. Since this is the kernel of the map $\tilde{Q}_{H} \to S_{H}^{0}$ given by $(q, s) \mapsto \overline{q}s^{-1}$ it is a mirabolic subgroup of $\tilde{Q}_{H}$. We have that $\tilde{\mathcal{F}} := \tilde{\overline{Q}}_{G} \backslash \tilde{G} \cong \mathcal{F}$ and its easy to see that $ Q_{\tilde{H}}^{0} $ has an open orbit on $\tilde{\mathcal{F}}$. A character $\mu \in X^{\bullet}(S_{H})$ induces a character $\tilde{\mu} \in X^{\bullet}( Q_{\tilde{H}}^{0})$ given by 
$$
    (h, \overline{h}) \mapsto \mu(\overline{h})
$$
which corresponds to $\mu$ under the isomorphism $Q_{H}^{0} \cong  Q_{\tilde{H}}^{0}$ sending $h$ to $(h, \overline{h})$. What's more, $\tilde{\mu}$ admits an extension to a character $\chi_{\mu} \in X^{\bullet}(\tilde{G})$ by simply taking for $(g,s) \in \tilde{G}$ 
$$
    \chi_{\mu}(g, s) = \mu(s).
$$
Thus 
$$
    (V^{G}_{\lambda} \otimes \chi_{\mu}^{-1})^{Q_{\tilde{H}}^{0}} \neq \{0\}
$$
i.e. the weight $\lambda - \chi_{\mu}$ is $Q_{\tilde{H}}^{0}$-admissible.

\section{Congruences and distribution valued Galois cohomology classes} \label{sec:congruences}
\subsection{Congruences between classes with varying twists}
 We show that twists of our classes by global characters of $G$ satisfy congruences modulo powers of $p$. Such congruence relations are crucial in the construction of \textit{analytic Euler systems} in the sense of \cite[Definition 3.2.1]{buyukboduk2018iwasawa} and $p$-adic $L$-functions as they allow one to extend cyclotomic variation over a full copy of the $\mathrm{GL}_{1}$-weight space.

 Let $\chi \in X^{\bullet}(G)$ be a character whose restriction to $L_{G}^{0}$ is trivial. Since $\mathcal{W}_{G}(L)$ has a group struture, for any wide-open disc $\mathcal{U}$ we make sense of the translated subset $\mathcal{U} + \chi$. We show that elements constructed via the maps 
 $$
    \iota_{\mathcal{U} + \chi, *}: H^{i}_{\mathrm{Iw}}(Y_{H}(Q_{H}^{0} \cap u^{-1}U_{n}u), \mathcal{O}) \to  H^{i + c}(Y_{G}(V_{n}), \mathscr{A}^{\mathrm{an}}_{\mathcal{U}, m} \otimes \chi) 
 $$
 satisfy a congruence relation as the classes vary over powers of $\chi$. This congruence relation will be the crucial technical input in constructing classes which interpolate twists for all such $\chi \in X^{\bullet}(G)$. This congruence relation was first observed for algebraic coefficients in \cite[Theorem 5.2.1]{loeffler2021spherical} where it was used to show that such classes vary in ordinary families by taking the limit over the level tower. In our case we do away with the ordinarity hypothesis and thus incur denominators as we `climb the tower'; we will control the ensuing growth using $p$-adic analytic methods as in \cite{recoleman}. 

 Since for any $n \geq 1$ the mod $p^{n}$ restriction of $\chi$ to $V_{n}$ is trivial we obtain an isomorphism
$$
    H^{i}(Y_{G}(K_{n}), \mathscr{A}^{\mathrm{Iw}}_{\mathcal{U}, m} \otimes \chi /p^{n}) \cong H^{i}(Y_{G}(K_{n}), \mathscr{A}^{\mathrm{Iw}}_{\mathcal{U}, m}/p^{n})
$$
for any $K_{n} \subset V_{n}$, given by taking the cup product with the canonical mod $p^{n}$ basis element $e_{\chi, n} \in H^{0}(Y_{G}(K_{n}), \mathcal{O}(\chi^{-1})/p^{n})$.

From now on let $\chi_{1}, \ldots, \chi_{d} \in X^{\bullet}(G)$ be a collection of characters of $G$ whose restriction to $L_{G}$ descends to a character on $S_{G}/S_{G}^{0}$. Given a tuple $\underline{j} = (j_{1}, \ldots, j_{d}) \in \Z^{d}_{\geq 0}$ write $\underline{\chi}^{\underline{j}}:= \chi_{1}^{j_{1}}\cdots\chi_{d}^{j_{d}}$ and for $m \gg 0$ let 
$$
    f^{\mathrm{sph}, [\underline{j}]}_{\mathcal{U}, m} \in A_{\mathcal{U}, m}^{\mathrm{an}}\otimes\underline{\chi}^{\underline{j}}
$$
denote the function defined in Definition \ref{def:sphericalvector} for the wide-open disc $\mathcal{U} + \underline{\chi}^{\underline{j}}$. 

\begin{proposition}
Let $h \geq 0$ and let $h_{1} + \cdots + h_{d} = h$ be a partition of $h$. Then for any $n \geq 0$ classes $f^{\mathrm{sph}, [\underline{j}]}_{\mathcal{U}, m}$ satisfy the following congruence:
$$
     \sum_{\substack{j_{i} \leq h_{i}\\ i = 1, \ldots, d}}\binom{h_{1}}{j_{1}}\cdots\binom{h_{d}}{j_{d}}(-1)^{\sum_{i = 1}^{d}j_{i}}\tau^{n}uf_{\mathcal{U}, m}^{\mathrm{sph}, [\underline{j}]}\underline{\chi}^{-\underline{j}} \equiv \ 0 \ \mathrm{mod} \ p^{hn}
$$
\end{proposition}
\begin{proof}
Let $\overline{n} \ell n \in U_{\mathrm{Bru}}^{G}(\Zp)$ and let $\overline{n}'\ell'uq \in U_{\mathrm{Sph}}(\Zp)$ be such that $\overline{n}\ell \tau^{n} n \tau^{-n}u = \overline{n}'\ell'uq$. We note that $q \in Q_{H}^{0} \cap u^{-1}\overline{Q}_{G}u \ \mathrm{mod} \ p^{n}$ so that $\chi_{i}(q) \equiv \ 1 \ \mathrm{mod} \ p^{n}$ by assumption. Then
\begin{align*}
    \sum_{\substack{j_{i} \leq h_{i}\\ i = 1, \ldots, d}}\binom{h_{1}}{j_{1}}\cdots\binom{h_{d}}{j_{d}}(-1)^{\sum_{i = 1}^{d}j_{i}}\tau^{n}uf_{\mathcal{U}, m}^{\mathrm{sph}, [\underline{j}]}\underline{\chi}^{-\underline{j}}(\overline{n} \ell n) &= \sum_{\substack{j_{i} \leq h_{i}\\ i = 1, \ldots, d}}\binom{h_{1}}{j_{1}}\cdots\binom{h_{d}}{j_{d}}(-1)^{\sum_{i = 1}^{d}j_{i}}f_{\mathcal{U}, m}^{\mathrm{sph}, [\underline{j}]}\underline{\chi}^{-\underline{j}}(\overline{n}\ell\tau^{n} n \tau^{-n} u) \\
    &=\sum_{\substack{j_{i} \leq h_{i}\\ i = 1, \ldots, d}}\binom{h_{1}}{j_{1}}\cdots\binom{h_{d}}{j_{d}}(-1)^{\sum_{i = 1}^{d}j_{i}}f_{\mathcal{U}, m}^{\mathrm{sph}, [\underline{j}]}\underline{\chi}^{-\underline{j}}(\overline{n}'\ell'uq) \\
    &= k_{\mathcal{U}}^{G}(\ell')\underline{\chi}^{\underline{j}}(\ell')\sum_{\substack{j_{i} \leq h_{i}\\ i = 1, \ldots, d}}\binom{h_{1}}{j_{1}}\cdots\binom{h_{d}}{j_{d}}(-1)^{\sum_{i = 1}^{d}j_{i}}\underline{\chi}^{-\underline{j}}(\ell'q) \\
    &= k_{\mathcal{U}}^{G}(\ell')\sum_{\substack{j_{i} \leq h_{i}\\ i = 1, \ldots, d}}\binom{h_{1}}{j_{1}}\cdots\binom{h_{d}}{j_{d}}(-1)^{\sum_{i = 1}^{d}j_{i}}\underline{\chi}^{-\underline{j}}(q) \\
    &= k_{\mathcal{U}}^{G}(\ell')\prod_{i = 1}^{d}\left(1 - \chi_{i}^{-1}(q)\right)^{h_{i}}.
\end{align*}
Since $1 - \chi_{i}^{-1}(p)$ vanishes mod $p^{n}$ the whole expression vanishes mod $p^{hn}$. 
 \end{proof}
\begin{definition}
Suppose we have 
$$
z^{H}_{\infty, n} \in H^{i}_{\mathrm{Iw}}(Y_{H}(Q_{H}^{0} \cap u^{-1}U_{n}u), \mathcal{O})
$$
and let $z^{G}_{\infty, n} = \iota_{*}\left(z^{H}_{\infty, n}\right)$.
For a tuple $\underline{j} = (j_{1}, \ldots , j_{d}) \in \Z_{\geq 0}^{d}$ define 
$$
    \xi^{G, [\underline{j}]}_{\mathcal{U}, n} = \tau^{n}_{*} \circ u_{*} \varprojlim_{k}(z^{G}_{\infty, n} \cup f^{\mathrm{sph}, [\underline{j}]}_{\mathcal{U}, m, k}) \in H^{i + c}(Y_{G}(K^{G}_{\infty, n}), A_{\mathcal{U}, m}^{\mathrm{Iw}}\otimes \underline{\chi}^{\underline{j}}).
$$
This is simply the class $\xi_{\mathcal{U} + \underline{\chi}^{\underline{j}}}^{G}$ constructed in Section \ref{sec:LABL}.
\end{definition}

We note that $K^{G}_{\infty, n} \subset V_{n}$ and fixes $\tau^{n}uf^{\mathrm{sph}}_{\mathcal{U}, m}$ for any wide-open disc $\mathcal{U}$. Let $h \geq 1$ and write $\mathrm{Res}_{n}^{hn}$ for the restriction map induced by the inclusion 
$$
    K^{G}_{\infty, hn} \subset K^{G}_{\infty, n}.
$$
\begin{lemma} \label{lem:cong}
Let $\chi_{1}, \ldots, \chi_{d}$ be as above and for $h \geq 1$ let $\underline{h} \in \Z_{\geq 0}^{d}$ denote the $d$-tuple with every entry equal to $h$. Then for any $n \geq 1$ and any partition $h = h_{1} + \cdots + h_{d}$ with $h_{i} \geq 0$ we have 
$$
    \sum_{\substack{j_{i} \leq h_{i}\\ i = 1, \ldots, d}}\binom{h_{1}}{j_{1}}\cdots\binom{h_{d}}{j_{d}}(-1)^{\sum_{i = 1}^{d}j_{i}}\mathrm{Res}_{n}^{hn}\left(\xi^{G,[\underline{j}]}_{\mathcal{U}, n}\right) \cup e_{\underline{j}} \equiv \ 0 \ \mathrm{mod} \ p^{hn}
$$
as an element of $H^{i + c}(Y_{G}(K^{G}_{\infty, hn}), A_{\mathcal{U},m}^{\mathrm{Iw}})$, where $e_{\underline{j}}$ is the canonical element of $\mathcal{O}\otimes\chi^{-\underline{j}} = \mathcal{O}\cdot e_{\underline{j}}$. 
\end{lemma}
\begin{proof}
The class $\xi_{\mathcal{U}, n}^{G, [\underline{j}]}$ is the image of a class
$$
    \xi^{G}_{\infty, n} \in H^{i + c}_{\mathrm{Iw}}(Y_{G}(K_{\infty, n}^{G}), \mathcal{O})
$$
under the map given by taking the limit over the mod $p^{k}$ cup products with $\tau^{n}uf_{\mathcal{U}, m, k}^{\mathrm{Sph}, [\underline{j}]}$ at level $K_{k, n}^{G}$. We then have for any $k \geq hn$
\begin{align*}
     \sum_{\substack{j_{i} \leq h_{i}\\ i = 1, \ldots, d}}\binom{h_{1}}{j_{1}}\cdots\binom{h_{d}}{j_{d}}(-1)^{\sum_{i = 1}^{d}j_{i}}&\mathrm{Res}_{n}^{hn}\left(\xi^{G,[\underline{j}]}_{\mathcal{U}, n}\right) \cup e_{\underline{j}} \\ &\equiv \sum_{\substack{j_{i} \leq h_{i}\\ i = 1, \ldots, d}}\binom{h_{1}}{j_{1}}\cdots\binom{h_{d}}{j_{d}}(-1)^{\sum_{i = 1}^{d}j_{i}}\mathrm{Res}_{n}^{hn}\left(\xi^{G,[\underline{j}]}_{\infty, n} \cup \tau^{n}uf_{\mathcal{U}, k}^{\mathrm{sph}}\right) \cup e_{\underline{j}} \ \mathrm{mod} \ p^{hn} \\ 
     &\equiv \sum_{\substack{j_{i} \leq h_{i}\\ i = 1, \ldots, d}}\binom{h_{1}}{j_{1}}\cdots\binom{h_{d}}{j_{d}}(-1)^{\sum_{i = 1}^{d}j_{i}}\mathrm{Res}_{n}^{hn}\left(\xi^{G}_{\infty, n}\right)  \cup \tau^{n}uf_{\mathcal{U}, k}^{\mathrm{sph}}\cup e_{\underline{j}} \\
     &\equiv \sum_{\substack{j_{i} \leq h_{i}\\ i = 1, \ldots, d}}\binom{h_{1}}{j_{1}}\cdots\binom{h_{d}}{j_{d}}(-1)^{\sum_{i = 1}^{d}j_{i}}\mathrm{Res}_{n}^{hn}\left(\xi^{G}_{\infty, n}\right)  \cup \tau^{n}uf_{\mathcal{U}, k}^{\mathrm{sph}}\underline{\chi}^{-\underline{j}} \\
     &\equiv \mathrm{Res}_{n}^{hn}\left(\xi^{G}_{\infty, n}\right)  \cup\sum_{\substack{j_{i} \leq h_{i}\\ i = 1, \ldots, d}}\binom{h_{1}}{j_{1}}\cdots\binom{h_{d}}{j_{d}}(-1)^{\sum_{i = 1}^{d}j_{i}} \tau^{n}uf_{\mathcal{U}, k}^{\mathrm{sph}}\underline{\chi}^{-\underline{j}}
     \equiv 0.
\end{align*}
\end{proof}

\subsection{Analytic classes in Galois cohomology}

We show how one can use the above congruences to construct distribution-valued Galois cohomology classes interpolating cyclotomic twists  with full variation. We assume that the spherical pair $(G, H)$ admits compatible Shimura data and we fix a set of primes $\Sigma$ not dividing the level, so that $Y_{G}, Y_{H}$ extend to smooth integral models over $\mathcal{O}_{E, \Sigma}$. 

We utilise the `abelian tower' construction of \cite{loefflerspherical}, recalled in Section \ref{sec:compat}.  Recall that we define
$$
    \tilde{G} = G \times C_{H}, 
$$
where $H \to C_{H}$ is the maximal torus quotient of $H$, and 
$$
    \Delta_{n} = \mathcal{C}_{H}(\Q)\backslash \mathcal{C}_{H}(\A)/C_{H, n} \cdot \mathcal{C}^{p} \cdot \mathcal{C}_{H}(\R)^{\dagger}.
$$
We set $\Delta_{\infty} = \varprojlim_{n}\Delta_{n}$. These are abelian $p$-adic Lie groups. There is a Galois character 
$$
    \kappa_{n}: \mathrm{Gal}(\overline{E}/E)^{\mathrm{ab}} \to \Delta_{n}
$$
defined using the Shimura cocharacter $\mathrm{Res}_{E/\Q}\mathbb{G}_{m} \to C_{H}$. Taking the limit over $n$ gives a character $\kappa_{\infty}: \mathrm{Gal}(\overline{E}/E)^{\mathrm{ab}} \to \Delta_{\infty}$. Let $\Gamma$ denote the image of $\kappa_{\infty}$ in $\Delta_{\infty}$ and let $E_{\infty}$ be the corresponding extension of the reflex field $E$. We write $\Gamma_{n} \subset \Gamma$ for the kernel of $\Gamma \to \Delta_{n}$ and let $E_{n}$ denote the corresponding extension of $E$. We have $E_{\infty} = \cup_{n \geq 0}E_{n}$. 

Suppose we are given elements $z^{H}_{n} \in H^{i}_{\mathrm{Iw}}(Y_{H}(Q_{H}^{0}\cap u^{-1}U_{n}u)_{\Sigma}, \mathcal{O})$ compatible as we vary $n$. Let $ \chi_{1}, \ldots, \chi_{d}$ be characters of $\tilde{G}$ which are trivial on $L_{G}^{0}$ and we impose that these characters define a local coordinate on $\Delta_{\infty}$ i.e. there is an open-compact subgroup $\Delta^{0}_{\infty} \subset \Delta_{\infty}$ such that 
$$
    (\chi_{1}, \ldots, \chi_{d}): \Delta^{0}_{\infty} \cong \Zp^{d}.
$$
For $\underline{j} = (j_{1}, \ldots, j_{d}) \in \Z^{d}_{\geq 0}$ write $\underline{\chi}^{\underline{j}} = \chi_{1}^{j_{1}}\cdots\chi_{d}^{j_{d}}$ and let $\mathcal{U} \subset \mathcal{W}_{m}$ be a wide-open disc for some $m\geq 1$. We can now form the class 
$$
    \xi_{\mathcal{U}, n}^{[\underline{j}]} = \left(\tau^{n}_{*} \circ u_{*} \circ \iota^{[\underline{j}]}_{\mathcal{U}, *}\right)(z_{n}^{H})\in H^{i + c}_{\mathrm{Iw}}(Y_{G}(V_{n})_{\Sigma}, \mathscr{A}^{\mathrm{Iw}}_{\mathcal{U}, m}\otimes\underline{\chi}^{\underline{j}}),
$$
where $\iota_{\mathcal{U}, *}^{[\underline{j}]}$ is, following our usual notational convention, the map $\iota_{\mathcal{U} + \underline{\chi}^{\underline{j}}, n, *}$ constructed in Section \ref{sec:LABL}. 
\begin{definition}
Let $\Delta_{1} \subset \Delta_{2}$ be profinite abelian groups. We define
$$
    \mathrm{Coind}_{\Delta_{1}}^{\Delta_{2}}(L) = \mathrm{Hom}_{\mathcal{O}[[\Delta_{1}]] \otimes L}(\mathcal{O}[[\Delta_{2}]] \otimes L, L).
$$
These groups admit a natural action of $\Delta_{2}/\Delta_{1}$ given by 
$$
    (\delta_{2} \cdot f)(\delta_{1}) = f(\delta_{2}^{-1}\delta_{1}),
$$
for $\delta_{i} \in \Delta_{2}, i = 1,2$.
\end{definition}

\begin{remark}
When $[\Delta_{2}: \Delta_{1}]$ is finite this is just the usual coinduced module occurring in the cohomological statement of Shapiro's lemma, a fact we will make use of below. 
\end{remark}
Let $F$ be a field and let the Galois group $G_{F}$ act trivially on $\mathrm{Coind}_{\Delta_{1}}^{\Delta_{2}}(L)$.   There is a natural $G_{F}$-equivariant map $\mathrm{CoInd}_{\Delta_{1}}^{\Delta_{2}}(L) \to L$ given by evaluating at the identity; this map becomes an isomorphism upon restriction to $\left(\mathrm{CoInd}^{\Delta_{2}}_{\Delta_{1}}(L)\right)^{\Delta_{2}}$. Given a class $z \in H^{1}(F, M \otimes \mathrm{CoInd}_{\Delta_{1}}^{\Delta_{2}}(L))$ for some $G_{F}$-module $M$ we write $z(1) \in H^{1}(F, M)$ for the image under the evaluation-at-$1$ map.

 Since $V_{n} \subset J_{G} \times C_{H, n}$ we can push these classes forward to obtain classes (which we abusively refer to with the same notation)
\begin{align*}
    \xi_{\mathcal{U}, n}^{[\underline{j}]} \in &H^{i + c}(Y_{G}(J_{G} \times C_{H, n})_{\Sigma}, \mathscr{A}^{\mathrm{Iw}}_{\mathcal{U}, m}\otimes\underline{\chi}^{\underline{j}})
\end{align*} 

Let $\Pi$ be a cuspidal automorphic representation of $G(\A)$ admitting $p$-stabilisation which defines a really nice $L$-valued point $x$ on the $Q_{G}$-eigenvariety in the sense of Definition \ref{def:reallynice}. Suppose $\underline{\Pi}$ is a Coleman family (in the sense of Definition \ref{def:colemanfam}) passing through $x$, fibred over $\mathcal{U}$, where we allow the case that $\mathcal{U} = \{\lambda\}$ is a single algebraic weight.  We can thus project to the $\underline{\Pi}$-eigenspace $W_{\underline{\Pi}}$ of $H^{i  + c}(Y_{G}(J_{G})_{\Sigma}, \mathscr{A}^{\mathrm{Iw}}_{\mathcal{U},m})$ on which $U_{p}'$ acts via its value $\alpha(\underline{\Pi}) \in \Lambda_{\mathcal{U}}$ under the identification $\mathcal{O}(\underline{\Pi})^{\leq 1} \cong \Lambda_{\mathcal{U}}$ induced by the weight map. 

Applying the Abel--Jacobi map constructed in Section \ref{sec:abeljac} we obtain Galois cohomology classes
$$
     \xi_{\underline{\Pi}, n}^{[\underline{j}]} \in H^{1}(E, W_{\underline{\Pi}}\otimes_{L}\mathrm{Coind}_{1}^{\Delta_{n}}(L)(\kappa_{n})\otimes \underline{\chi}^{\underline{j}}) = H^{1}(E_{n}, W_{\underline{\Pi}}\otimes_{L}\mathrm{Coind}_{\Gamma/\Gamma_{n}}^{\Delta_{n}}(L)\otimes \underline{\chi}^{\underline{j}})
$$
such that the image of $ \xi_{\underline{\Pi}, n + 1}^{[\underline{j}]}$ under $\mathrm{cores}^{E_{n + 1}}_{E_{n}}$ composed with the map on cohomology induced by the natural map $\mathrm{Coind}_{\Gamma/\Gamma_{n}}^{\Delta_{n + 1}}(L) \to \mathrm{Coind}_{\Gamma/\Gamma_{n}}^{\Delta_{n}}(L)$ is given by $\alpha(\underline{\Pi})\xi_{\underline{\Pi}, n}^{[\underline{j}]}$. Note that the cohomology group $H^{1}(E_{n}, W_{\underline{\Pi}}\otimes_{L}\mathrm{Coind}_{\Gamma/\Gamma_{n}}^{\Delta_{n}}(L) \otimes \underline{\chi}^{\underline{j}})$ is naturally acted upon by $ \Delta_{n}$ with the action of the subgroup $\kappa_{n}: \Gamma/\Gamma_{n} \hookrightarrow \Delta_{n}$ coinciding with the natural action of $\Gamma$.  Let $\mathcal{V} \subset \mathcal{U}$ be an affinoid subset so that $W_{\underline{\Pi}}\vert_{\mathcal{V}} = W_{\underline{\Pi}} \hat{\otimes}_{\mathcal{O}(\mathcal{U})}\mathcal{O}(\mathcal{V})$ is a Banach module over $\mathcal{O}(\mathcal{V})$. Set 
$$
   \hat{W}_{\underline{\Pi}}\vert_{\mathcal{V}} = \mathrm{Hom}_{\mathcal{O}(\mathcal{V})}(\mathrm{Hom}_{\mathcal{O}(\mathcal{V})}(W_{\underline{\Pi}}\vert_{\mathcal{V}}, \mathcal{O}(\mathcal{V})), \mathcal{O}(\mathcal{V})),
$$
the reflexive hull of $W_{\underline{\Pi}}\vert_{\mathcal{V}}$. This is a locally free $\mathcal{O}(\mathcal{V})$-module whose specialisations at really nice points coincide with those of $W_{\underline{\Pi}}\vert_{\mathcal{V}}$. Shrinking $\mathcal{V}$ if necessary, we assume that $\hat{W}_{\underline{\Pi}}\vert_{\mathcal{V}}$ is a free $\mathcal{O}(\mathcal{V})$-module, so that we can apply the results of Appendix \ref{app:A}. Set $\lambda = v_{\mathcal{O}(\mathcal{V})}(\alpha(\underline{\Pi})\vert_{\mathcal{V}})$. We let $\hat{\xi}_{\underline{\Pi}, n}^{[\underline{j}]}\vert_{\mathcal{V}}$ be the image of $\xi_{\underline{\Pi}, n}^{[\underline{j}]}$ in $H^{1}(E_{n}, \hat{W}_{\underline{\Pi}}\vert_{\mathcal{V}}\otimes_{L}\mathrm{Coind}_{\Gamma/\Gamma_{n}}^{\Delta_{n}}(L) \otimes \underline{\chi}^{\underline{j}})$.
\begin{definition}
We set 
$$
    x_{n}(\underline{j}) = \alpha(\underline{\Pi})\vert_{\mathcal{V}}^{-n}\hat{\xi}_{\underline{\Pi}, n}^{[\underline{j}]}\vert_{\mathcal{V}}. 
$$
\end{definition}
For every $n \geq 1$ and every $\underline{j} = (j_{1}, \ldots, j_{d})$ we have an injective $\Delta_{\infty}$-equivariant map 
$$
\mathrm{Coind}_{\Gamma/\Gamma_{n}}^{\Delta_{n}}(L) \otimes \underline{\chi}^{\underline{j}} \hookrightarrow \mathrm{Coind}_{\Gamma}^{\Delta_{\infty}}(L)^{\Delta_{\infty} = \underline{\chi}^{\underline{j}}}
$$
defined in the following way: the module $\mathrm{Coind}_{\Gamma}^{\Delta_{\infty}}(L)$ has a natural `twisting' element $\mathrm{Tw}^{[\underline{j}]}$ given by sending a group-like element $[\delta] \in \mathcal{O}[[\Delta_{\infty}]]\otimes L$ to $\underline{\chi}^{\underline{j}}(\delta)$. The above map is then given by pulling back under the natural reduction map $\mathcal{O}[[\Delta_{\infty}]] \to \mathcal{O}[\Delta_{n}]$ and multiplying by $\mathrm{Tw}^{[\underline{j}]}$. By applying the map on cohomology induced by this map with restriction to $E_{\infty}$ we obtain
$$
    x_{n}(\underline{j}) \in H^{1}(E_{\infty}, \hat{W}_{\underline{\Pi}}\vert_{\mathcal{V}}\hat{\otimes}\mathrm{Coind}_{\Gamma}^{\Delta_{\infty}}(L))^{\Delta_{\infty} = \underline{\chi}^{\underline{j}}}
$$
which satisfy 
$$
    \sum_{\delta \in p^{n}\Delta^{0}_{\infty}/p^{n + 1}\Delta^{0}_{\infty}}\underline{\chi}^{\underline{j}}(\delta)\delta \cdot x_{n + 1}(\underline{j}) = x_{n}(\underline{j}). 
$$
\begin{definition}
Let $A$ be a Banach $L$-algebra, $M$ a Banach $A$-module, and let $\Delta$ be an abelian $p$-adic Lie group with fixed open-compact subgroup $\Delta^{0}$ isomorphic to $\Zp^{d}$ for some $d \in \Z_{> 0}$.
\begin{enumerate}
\item Define the $A$-Banach space $\mathscr{C}^{h}(\Delta, L)$ to be the space of functions which are analytic on the subsets $a + p^{h}\Delta^{0}$ for all $a \in \Delta$, equipped with the norm given by taking the supremum over all such subsets. Define 
$$
    \mathscr{C}^{\ell a}(\Delta, L) := \varinjlim_{h}\mathscr{C}^{h}(\Delta, L)
$$
equipped with the inductive limit topology.
\item For $\lambda \in \R_{\geq 0}$ define a Banach space for $d \geq 1$
$$
    \mathscr{C}_{\lambda}\left(\Zp^{d}, L\right)
$$
to be $\hat{\bigotimes}_{i = 1}^{d}\mathscr{C}_{\lambda}\left(\Zp, L\right)$ where $\mathscr{C}_{\lambda}\left(\Zp, L\right)$ is the the Banach space of continuous functions of `classe $\lambda$' as defined in \cite[\S I.5.]{colmez}. In general define 
$$
    \mathscr{C}_{\lambda}\left(\Delta, L\right)
$$
to be the space of continuous functions $f: \Delta \to L$ which are of class $\lambda$ on $a + \Delta_{0}$ for all $a \in \Delta$. 
\item Define 
$$
    \mathscr{D}^{\ell a}(\Delta, M) := \mathrm{Hom}_{A, \mathrm{cont}}(\mathscr{C}^{\ell a}(\Delta, L), M),
$$
the space of continuous $A$-linear functions on $ \mathscr{C}^{\ell a}(\Delta, L)$ taking values in $M$. We call this latter space the space of $M$ -valued \textit{locally analytic distributions} on $\Delta$. It does not depend on the choice of $\Delta^{0}$.
\item Define a subspace 
$$
    \mathscr{D}_{\lambda}(\Delta, M) \subset \mathscr{D}^{\ell a}(\Delta, M)
$$
to be the continuous dual of $\mathscr{C}_{\lambda}(\Delta, L)$. This is an $A$-Banach module which we give the `weak topology' of pointwise convergence on $\mathscr{C}_{\lambda}(\Delta, L)$. 
\item For $h \in \Z_{\geq 0}$ define 
$$
    \mathscr{C}^{[0, h]}(\Delta, L)
$$
to be the space of locally polynomial functions taking values in $L$. 
\item For $h \in \Z_{\geq 0}$ define 
$$
    \mathscr{D}^{[0, h]}(\Delta, M)
$$
to be the $A$ linear functions on $\mathscr{C}^{[0, h]}(\Delta, L)$ taking values in $M$.
\end{enumerate}
\end{definition}
Given an $L$-Banach space $M$ with a continuous action of $\Delta_{\infty}$, define an action of $\Delta_{\infty}$ on  $\mathscr{D}^{\ell a}(\Delta_{\infty}, M)$ for $a \in \Delta_{\infty}$, $f$ locally analytic on $\Delta_{\infty}$ by
$$
    \int_{\Delta_{\infty}}f(x)(a \cdot \mu)(x) := a\cdot \left(\int_{\Delta_{\infty}}f(a^{-1} \cdot x)\mu(x) \right).
$$
\begin{theorem} \label{thm:interpo}
There is a unique Galois  cohomology class
$$
    \mu_{\underline{\Pi}} \in H^{1}(E_{\infty}, \mathscr{D}_{\lambda}(\Delta_{\infty}, \hat{W}_{\underline{\Pi}}\vert_{\mathcal{V}}))^{\Gamma}
$$
satisfying
$$
    \int_{p^{n}\Delta^{0}_{\infty}}\underline{\chi}^{\underline{j}}(\delta)\mu_{\Pi_{\mathcal{U}}}(\delta) = x_{n}(\underline{j})(1)
$$
for all $n \geq 1$ and $\underline{j} = (j_{1}, \ldots, j_{d}) \in \Z_{\geq 0}^{d}$. 
\end{theorem}
\begin{proof}
The compatibility between the classes $x_{n}(\underline{j})$ as $n$ varies allows us to define, for $h \geq \lfloor \lambda \rfloor$, a unique cohomology class
$$
    \mu_{\mathrm{alg}}[h] \in H^{1}(E_{\infty}, \mathscr{D}^{[0, h]}(\Delta_{\infty}, \hat{W}_{\Pi_{\mathcal{V}}}\hat{\otimes} \mathrm{Coind}_{\Gamma}^{\Delta_{\infty}}(L)))^{\Delta_{\infty}}
$$
satisfying 
$$
     \int_{p^{n}\Delta^{0}}\underline{\chi}^{\underline{j}}(\delta)\mu_{\mathrm{alg}}[h](\delta) = x_{n}(\underline{j})
$$
for all $0 \leq j_{1} + \cdots + j_{d} \leq h$ and all $n \geq 1$. 
We have 
\begin{align*}
    H^{1}(E_{\infty}, \mathscr{D}^{[0,h]}(\Delta_{\infty}, \hat{W}_{\Pi_{\mathcal{V}}}) \hat{\otimes} \mathrm{Coind}^{\Delta_{\infty}}_{\Gamma}(L))^{\Delta_{\infty}} &= \left(\left(H^{1}(E_{\infty}, \mathscr{D}^{[0,h]}(\Delta_{\infty}, \hat{W}_{\Pi_{\mathcal{V}}})) \hat{\otimes} \mathrm{Coind}^{\Delta_{\infty}}_{\Gamma}(L)\right)^{\Gamma}\right)^{\Gamma\backslash\Delta_{\infty}} \\
    &=\left(H^{1}(E_{\infty}, \mathscr{D}^{[0,h]}(\Delta_{\infty}, \hat{W}_{\Pi_{\mathcal{V}}}))^{\Gamma} \hat{\otimes} \mathrm{Coind}^{\Delta_{\infty}}_{\Gamma}(L)\right)^{\Gamma\backslash\Delta_{\infty}}  \\
    &= H^{1}(E_{\infty}, \mathscr{D}^{[0,h]}(\Delta_{\infty}, \hat{W}_{\Pi_{\mathcal{V}}}))^{\Gamma},
\end{align*}
where the third equality becomes clear when one notes that 
$$
    H^{1}(E_{\infty}, \mathscr{D}^{[0,h]}(\Delta_{\infty}, \hat{W}_{\Pi_{\mathcal{V}}}))^{\Gamma} \hat{\otimes} \mathrm{Coind}^{\Delta_{\infty}}_{\Gamma}(L) = \mathrm{Hom}_{\mathcal{O}[[\Gamma]] \otimes L}(\mathcal{O}[[\Delta_{\infty}]] \otimes L, H^{1}(E_{\infty}, \mathscr{D}^{[0,h]}(\Delta_{\infty}, \hat{W}_{\Pi_{\mathcal{V}}}))^{\Gamma}),
$$
so we end up with classes 
$$
    \mu_{\mathrm{alg}}[h] \in H^{1}(E_{\infty}, \mathscr{D}^{[0,h]}(\Delta_{\infty}, \hat{W}_{\Pi_{\mathcal{V}}}))^{\Gamma}
$$
and by tracing through the above isomorphisms we see that these classes satisfy 
$$
    \int_{p^{n}\Delta^{0}_{\infty}}\underline{\chi}^{\underline{j}}(\delta)\mu_{\mathrm{alg}}[h](\delta) = x_{n}(\underline{j})(1)
$$
By Appendix \ref{app:A}, Lemma \ref{lem:cong} and Proposition \ref{prop:bound} the classes $x_{n}(\underline{j})(1)$ satisfy 
$$
       \norm{p^{-hn}\sum_{\substack{j_{i} \leq h_{i}\\ i = 1, \ldots, d}}\binom{h_{1}}{j_{1}}\cdots\binom{h_{d}}{j_{d}}(-1)^{\sum_{i = 1}^{d}j_{i}}x_{n}(\underline{j})(1)} \leq Cp^{\lfloor \lambda n \rfloor},
$$
where the norm $\norm{\cdot}$ is the one induced by the norm on $\hat{W}_{\underline{\Pi}}\vert_{\mathcal{V}}$ and so we can extend to a class 
$$
    \mu_{\underline{\Pi}}[h] \in H^{1}(E_{\infty}, \mathscr{D}_{\lambda}(\Delta_{\infty}, \hat{W}_{\underline{\Pi}}\vert_{\mathcal{V}}))^{\Gamma}
$$
uniquely defined by the interpolation property
$$
    \int_{p^{n}\Delta^{0}_{\infty}}\underline{\chi}^{\underline{j}}(\delta)\mu_{\Pi_{\mathcal{U}}}(\delta) = x_{n}(\underline{j})(1)
$$
for all $n \geq 1$ and $\underline{j} \in \Z_{\geq 0}^{d}$ satisfying $0 \leq j_{1} + \ldots + j_{d} \leq h$. Now we note that taking $h' > h$ produces a class $\mu_{\Pi_{\mathcal{V}}}[h']$ satisfying a strictly stronger interpolation property, whence we conclude (by unicity) that $\mu_{\underline{\Pi}}[h]$ is independent of $h$ and we write $\mu_{\underline{\Pi}}$ for this class. 
\end{proof}
\section{Application to $\mathrm{GSp}_{4}$} \label{sec:application}
\subsection{Eisenstein classes} \label{sec:eis}
We briefly recall the definition of Beilinson's Eisenstein classes and their $p$-adic interpolations due to Kings. The Eisenstein classes provide the primordial example of classes satisfying the relations of $z_{n}^{H}$ in Theorem \ref{thm:normrelation} and are a fundamental input to the construction of the Euler system of Kato \cite{kato} as well as the Euler system of Beilinson--Flach elements of Lei--Loeffler--Zerbes \cite{lei2014euler} and the Lemma--Flach Euler system of Loeffler--Zerbes \cite{LZGsp}. 

Set $H = \mathrm{GL}_{2}$ considered as a reductive group over $\Z$, let $T_{H}$ the maximal torus of diagonal matrices, $B_{H}$ be the Borel subgroup of upper triangular matrices with Levi decomposition $B_{H} = T_{H} \times N_{H}$, and set $Q_{H} = B_{H}$. 
\begin{definition}
For a locally profinite group $K$ define 
$$
\mathcal{S}(K, \mathbb{Z}_{p}) = \{\phi: K \to \Zp: \text{$\phi$ locally constant with compact support}\},
$$
the space of $\mathbb{Z}_{p}$-valued \textit{Schwartz functions} on $K$. If furthermore $K$ admits the structure of a topological ring with respect to its locally profinite topology, then we equip the spaces $\mathcal{S}(K^{2}, \Zp), \mathcal{S}_{0}(K^{2}, \Zp)$ with an action of $H(K)$ given by 
$$
    (h \cdot \phi)\left ( \left(k_{1}, k_{2}\right) \right) = \phi\left((k_{1}, k_{2}) \cdot h\right).
$$
We define the following subspaces:
\begin{itemize}
\item Denote by $\mathcal{S}_{0}(K, \Zp)$ the subspace of $\mathcal{S}(K^{2}, \Zp)$ consisting of of $\phi: K \to \Zp$ satisfying $\phi(1) = 0$. This subspace is preserved by the above action of $H(K)$. 
\item For an integer $c$ coprime to $6p$, let $_{c}\mathcal{S}_{0}((\mathbb{A}^{(p)}_{f} \times \mathbb{Z}_{p})^{2}, \mathbb{Z}_{p})$ denote the subspace of $\mathcal{S}_{0}((\mathbb{A}^{(p)}_{f} \times \mathbb{Z}_{p})^{2}, \mathbb{Z}_{p})$ consisting of elements of the form $\phi^{(c)} \otimes \mathrm{ch}(\mathbb{Z}_{c}^{2})$, where $\phi^{(c)}$ is a $\mathbb{Z}$-valued Schwartz function on $(\mathbb{A}_{f}^{(c)})^{2}$ and $\mathbb{Z}_{c} = \prod_{\ell \mid c}\mathbb{Z}_{\ell}$. This admits an action of $H(\A_{f}^{(pc)} \times \Zp \times \Z_{c})$ as above. 
\end{itemize}
\end{definition}

The irreducible algebraic representations of $H$ are parameterised by pairs of integers $(k, m)$ with $k \geq 0$ and are given by $V_{k, m} := \mathrm{Sym}^{k}V_{\mathrm{std}} \otimes \mathrm{det}^{m}$, where $V_{\mathrm{std}} = \Qp^{2}$ with the standard left action of $H(\Qp)$. Let $\mathscr{V}_{k}(m)$ be the \'etale sheaf associated to $V_{k, m}$ and let $\mathscr{V}_{k, \Zp}(m)$ be the \'etale sheaf associated to the minimal admissible lattice.

\begin{definition}
Let $k \geq 0$. For $c$ as above and $U \subset H(\mathbb{A}^{(pc)}_{f} \times \mathbb{Z}_{pc})$ a neat open compact subgroup define a map \begin{align*}
    {_{c}}\mathrm{Eis}^{k}: {_{c}}\mathcal{S}_{0}((\mathbb{A}^{(p)}_{f} \times \mathbb{Z}_{p})^{2}, \mathbb{Z}_{p})^{U} \to H^{1}(Y_{H}(U), \mathscr{V}_{k, \Zp}(1))
\end{align*}
where for $\phi \in _{c}\mathcal{S}_{0}((\mathbb{A}^{(p)}_{f} \times \mathbb{Z}_{p})^{2}, \mathbb{Z}_{p})^{U}$,  $_{c}\mathrm{Eis}^{ k}(\phi)$ is the element defined in \cite[Definition 3.3.6]{kings2015eisenstein} whose image in $H^{1}(Y_{H}(U), \mathscr{V}_{k}(1))$ satisfies 
$$
    _{c}\mathrm{Eis}^{k}(\phi) = \left(c^{2} - c^{-k}\begin{pmatrix} c & \\ & c \end{pmatrix}^{-1} \right)r(\mathrm{Eis}_{\mathrm{mot}, }^{k}(\phi)), 
$$
where $\mathrm{Eis}_{\mathrm{mot}}^{k}(\phi)$ is Beilinson's motivic Eisenstein class, defined in \cite[\S 3]{beilinson1986higher}, and $r$ is the \'etale regulator map. 
\end{definition}

By a deep result of Kings \cite[Proposition 3.3.5]{kings2015eisenstein} these classes interpolate as $k$ varies. Set 
$$
K_{p,n} = \{g \in \mathrm{GL}_{2}(\Zp): g \equiv \begin{psmallmatrix}
    * & * \\ & 1
\end{psmallmatrix} \ \mathrm{mod} \ p^{n}\}
$$
and fix a prime-to-$p$ level group $K^{(p)} \subset G(\A^{(p)}_{f})$ such that the product $K_{n} := K^{(p)}K_{p, n}$ is neat. Let $\phi_{n} = \phi^{(p)} \otimes \mathrm{ch}(p^{n}\Zp \times \left(1 + p^{n}\Zp\right)) \in \mathcal{S}(\A_{f}, \Z_{p})^{K_{n}}$ where $\phi^{(p)} \in \mathcal{S}(\A_{f}^{(p)}, \Zp)^{K^{(p)}}$. We write
$$
    {_{c}}\mathrm{Eis}_{\acute{e}t, n}^{k} := {_{c}}\mathrm{Eis}^{k}(\phi_{n}) \in H^{1}(Y_{H}(K_{n}), \mathscr{V}_{k, \Zp}(1)).
$$
\begin{definition} \label{def:eisiw}
Let $\Sigma$ be a set of primes including $p$ and the primes at which $K^{(p)}$ ramifies. Define the \textit{Eisenstein--Iwasawa} class to be the projective limit
$$
    {_{c}}\mathcal{EI} = \varprojlim_{n} {_{c}}\mathrm{Eis}_{n}^{0} \in H^{1}_{\mathrm{Iw}}(Y_{H}(K_{\infty})_{\Sigma}, \Zp(1))
$$
where $K_{\infty} = \cap _{n}K_{n}$. 
\end{definition}
\begin{remark}
    The class ${_{c}}\mathrm{Eis}^{0}_{n}$ coincides with the image of the \textit{Siegel unit} $_{c}g_{\phi_{n}}$  (c.f. \cite[7.1.2]{LZGsp}) under the Kummer map $\mathcal{O}(Y_{H}(K_{n}))^{\times} \to H^{1}(Y_{H}(K_{n}), \mathscr{V}_{k, \Zp}(1))$. It is clear from the construction given in \cite[Proposition 1.3]{kato} that these classes extend over the integral model $Y_{H}(K_{n})_{\Sigma}$ and so it makes sense to consider ${_{c}}\mathrm{Eis}^{0}_{n}$ as an element of $H^{1}(Y_{H}(K_{n})_{\Sigma}, \mathscr{V}_{k, \Zp}(1))$.
\end{remark}
These classes satisfy $\mathrm{mom}^{k}_{n}\left({_{c}}\mathcal{EI}\right) = {_{c}}\mathrm{Eis}_{\acute{e}t, n}^{k}$, where $\mathrm{mom}^{k}_{n}: H^{1}_{\mathrm{Iw}}(Y_{H}(K_{\infty})_{\Sigma}, \Zp(1)) \to H^{1}(Y_{H}(K_{n}), \mathscr{V}_{k, \Zp}(1))$ is given by the following composition: 
\begin{align*}
    H^{1}_{\mathrm{Iw}}(Y_{H}(K_{\infty})_{\Sigma}, \Zp(1)) &= \varprojlim_{s} H^{1}(Y_{H}(K_{s})_{\Sigma}, \Z/p^{s}\Z(1)) \\ 
    &\xrightarrow{\cup f_{k} \ \mathrm{mod} \ p^{s}} \varprojlim_{s} H^{1}(Y_{H}(K_{s})_{\Sigma}, \mathscr{V}_{k,\Zp}/p^{s}(1)) \\
    &\to \varprojlim_{s} H^{1}(Y_{H}(K_{n})_{\Sigma}, \mathscr{V}_{k,\Zp}/p^{s}(1)) \\
    &\cong  H^{1}(Y_{H}(K_{n})_{\Sigma}, \mathscr{V}_{k,\Zp}(1)) \\
    &\to H^{1}(Y_{H}(K_{n}), \mathscr{V}_{k,\Zp}(1))
\end{align*}
where the first map is projection to the $n$th layer of the inverse limit, the second is the inverse limit over cup products with the  mod $p^{n}$ reduction of the highest weight vector $f_{k} \in V_{k, 0}$, the third map is projection to level $K_{n}$ and the final map is restriction to the generic fibre.

\subsection{The Lemma--Flach Euler system}
In this section let $\Q_{n} = \Q(\zeta_{p^{n}})$, $\Q_{\infty} = \cup_{n}\Q(\zeta_{p^{n}})$ and set $\Gamma = \mathrm{Gal}(\Q_{\infty}/\Q)$, $\Gamma_{n} = \mathrm{Gal}(\Q_{\infty}/\Q_{n})$. Let 
$G = \mathrm{GSp}_{4}$, with the explicit presentation
$$
    G(R) = \{(g, \mu(g)) \in \mathrm{GL}_{4}(R) \times R^{\times}: gJg^{t} = \mu(g)J\}
$$
for $J = \begin{psmallmatrix}
    & & & 1 \\
    & & 1 & \\
    & -1 & & \\
    -1 & & & 
\end{psmallmatrix}$ and a $\Z$-algebra $R$. Set
$H = \mathrm{GL}_{2} \times _{\mathrm{GL}_{1}}\mathrm{GL}_{2}$. We consider the inclusion 
$$
    \iota: H \to G
$$
given by
$$
    \begin{psmallmatrix}
        a & b \\ c & d
    \end{psmallmatrix} \times \begin{psmallmatrix}
        a' & b' \\ c' & d'
    \end{psmallmatrix} \mapsto \begin{psmallmatrix}
        a & & & b \\ & a' & b' & \\ & c' & d' & \\ c & & & d
    \end{psmallmatrix}
$$
under which these groups form a spherical pair $(G, H)$ for the upper triangular Borel subgroup $B_{G} \subset G$. The subgroup
$$
    T_{G} := \{\begin{psmallmatrix}
        x &  &  &  \\
        & y &  &  \\
        & & \lambda y^{-1} &  \\
        & & & \lambda x^{-1}
    \end{psmallmatrix} : x, y, \lambda \in \mathbb{G}_{m}\} \subset G
$$
is a maximal torus, and we have the Levi decomposition $B_{G} = T_{G} \ltimes N_{G}$ given pictorially by 
$$
    \begin{psmallmatrix}
        * & * & * & * \\
        & * & * & * \\
        & & * & * \\
        & & & *
    \end{psmallmatrix} = \begin{psmallmatrix}
        * &  &  &  \\
        & * &  &  \\
        & & * &  \\
        & & & *
    \end{psmallmatrix} \ltimes \begin{psmallmatrix}
        1& * & * & * \\
        & 1 & * & * \\
        & & 1 & * \\
        & & & 1
    \end{psmallmatrix}
$$
Similarly, let $B_{H} = T_{H} \ltimes N_{H}$ be the Levi decomposition of the upper triangular Borel of $H$ given by 
$$
    \begin{pmatrix}
        * & * \\ & * 
    \end{pmatrix} \times \begin{pmatrix}
        * & * \\ & * 
    \end{pmatrix} = \begin{pmatrix}
        * &  \\ & * 
    \end{pmatrix} \times \begin{pmatrix}
        * & \\ & * 
    \end{pmatrix}  \ltimes \begin{pmatrix}
        1 & * \\ & 1 
    \end{pmatrix} \times \begin{pmatrix}
        1 & * \\ & 1 
    \end{pmatrix}
$$
\begin{remark}
    In the notation of the main body of the paper we are taking 
    $$
    Q_{G} = B_{G}, Q_{G}^{0} = N_{G}, L_{G} = T_{G}, L_{G}^{0} = \{1\},
    $$
    and 
    $$
        Q_{H} = B_{H}, L_{H} = T_{H}. 
    $$
\end{remark}
Let $\chi_{1}, \chi_{2} \in X^{\bullet}(T_{G})$ be the projections to the first and second entries of $T_{G}$ respectively. Then $X^{\bullet}(T_{G})$ is generated by $\chi_{1}, \chi_{2}, \mu$ and we identify
$$
    X^{\bullet}(T_{G}) \cong \{(r_{1}, r_{2}; c): r_{1}, r_{2} ,c \in \Z\}
$$
by sending $\chi_{1} \mapsto (1, -1; 0), \chi_{2} \mapsto (0, 1; 0), \mu \mapsto (0, 0; 1)$. Under this isomorphism we have 
$$
    X^{\bullet}_{+}(T_{G}) = \{(r_{1}, r_{2};c): r_{1} \geq r_{2}\}
$$
For $(r_{1}, r_{2}; c) \in X^{\bullet}_{+}(T_{G})$ let $V(r_{1}, r_{2}; c)$ denote the associated irreducible algebraic representation. For $t_{1}, t_{2} \in \Z_{\geq 0}, s \in \Z$ let $W(t_{1}, t_{2}; e) := V_{t_{1}, 0}\boxtimes V_{t_{2}, 0} \otimes \det^{s}$ (recalling the notation of Section \ref{sec:eis}) be the irreducible representation of $H$ of weight $(t_{1}, t_{2}; s)$.
We have the explicit branching law
$$
    V(r_{1}, r_{2}; c) \vert_{H} = \bigoplus_{0 \leq q \leq r_{2}}\bigoplus_{0 \leq r \leq r_{1} - r_{2}}W(r_{1} - q - r, r_{2} - q + r; q + c) 
$$
 depending on the four parameters $r_{1}, r_{2}, q ,r \geq 0$ (c.f. \cite[Proposition 4.3.1]{LZGsp}). Let $V_{\Zp}(r_{1}, r_{2}; c)$ be the \textit{maximal} admissible lattice in $V(r_{1}, r_{2}; c)$ and let $W_{\Zp}(t_{1}, t_{2}; s)$ be the \textit{minimal} admissible lattice in $W(t_{1}, t_{2}; s)$, given explicitly by $\mathrm{TSym}^{t_{1}} \Zp^{2} \boxtimes \mathrm{TSym}^{t_{2}}\Zp^{2} \otimes \det^{s}$, where for $k \geq 0, \mathrm{TSym}^{k}$ of a representation is the space of $k$-fold symmetric tensors.

 The explicit branching law above has $4$ parameters, however, the machine we have constructed for interpolating algebraic branching laws gives classes varying over a weight space for the torus $Q_{G}/Q_{G}^{0} = B_{G}/N_{G} = T_{G}$, which has rank $3$. In order to vary all $4$ parameters of the branching law we need to modify the groups $G$ and $H$ slightly. 
Set 
$$
    \tilde{H} = H \times \mathrm{GL}_{1}, \ \tilde{G} = G \times C_{\tilde{H}}.
$$
and define $B_{\tilde{G}} := B_{G} \times C_{\tilde{H}}, T_{\tilde{G}} = T_{G} \times C_{\tilde{H}}$ and 
$$
    Q_{\tilde{H}}^{0} = \{\begin{psmallmatrix}
        x & * \\ & 1
    \end{psmallmatrix} \times \begin{psmallmatrix}
        xy & * \\ & y^{-1}
    \end{psmallmatrix} \times (y) : x, y \in \mathbb{G}_{m}\}.
$$
The embedding $\iota$ extends to an embedding
\begin{align*}
    \tilde{\iota}: \tilde{H} &\to \tilde{G} = G \times C_{H} \times \mathrm{GL}_{1} \\
    (h, z) &\mapsto (\iota(h), \det(h), z).
\end{align*}
 We let $\sigma: \tilde{G} \to \mathrm{GL}_{1}$ denote the projection to the auxiliary $\mathrm{GL}_{1}$ factor. Define irreducible $\tilde{G}, \tilde{H}$ representations
 $$
 \tilde{V}(r_{1}, r_{2}) = V(r_{1}, r_{2}; -(r_{1} + r_{2}))
 , \tilde{W}(t_{1}, t_{2}) = W(t_{1}, t_{2}; -(t_{1} + t_{2})) \otimes \sigma^{t_{2}}.
 $$
 and let $\tilde{V}_{\Zp}(r_{1}, r_{2};c)$ ( resp. $\tilde{W}_{\Zp}(t_{1}, t_{2})\otimes \sigma^{t_{2}}$) be the maximal (resp. minimal) lattices. A quick computation shows that the highest weight vector of $\tilde{W}(t_{1}, t_{2})$ is invariant under $Q_{\tilde{H}}^{0}$ and thus it follows from the branching law that $\tilde{V}(r_{1}, r_{2}) \otimes \mu^{q} \otimes \sigma^{r - q}$ contains a $Q_{\tilde{H}}^{0}$-invariant vector $f^{[r_{1},r_{2},q,r]}$ for each $0 \leq q \leq r_{2}, 0 \leq r \leq r_{1} - r_{2}$. There is thus a unique normalised $\tilde{H}$-equivariant branching map 
 $$
    \mathrm{br}^{[r_{1},r_{2},q,r]}: \tilde{W}(r_{1} - q - r, r_{2} -q + r) \to \tilde{V}(r_{1}, r_{2}) \otimes \mu^{q} \otimes \sigma^{r - q}
 $$
 which restricts to a map on lattices
 $$
    \mathrm{br}^{[r_{1},r_{2},q,r]}: \tilde{W}_{\Zp}(r_{1} - q - r, r_{2} -q + r)\to \tilde{V}_{\Zp}(r_{1}, r_{2}) \otimes \mu^{q} \otimes \sigma^{r - q}.
 $$
\begin{remark}
The extra $\mathrm{GL}_{1}$-factor is an instance of the construction of Section \ref{sec:superf} where we have appended to our groups the smaller torus $S_{H}^{0} \cap H^{\mathrm{der}} \subset S_{H}^{0}$ since $C_{H}$ already accounts for the variation in the parameter $q$.  
\end{remark}
The flag variety $\mathcal{F}_{\tilde{G}} = \tilde{G}/B_{\tilde{G}}$ is naturally isomorphic to $\mathcal{F}_{G} = G/B_{G}$ and $\tilde{H}$ has an open orbit on this space i.e. the pair $(\tilde{G}, \tilde{H})$ is spherical. The group $Q_{\tilde{H}}^{0}$ has an open orbit on $\mathcal{F}_{\tilde{G}} = \mathcal{F}_{G}$ with representative $u = \begin{psmallmatrix}
    1 & 1 & 1 & \\ & 1 & & 1 \\ & & 1 & -1 \\ & & & 1
\end{psmallmatrix}$. For $n \geq 1$ we set $U_{n}, V_{n} \subset G(\Zp)$ to be the level groups at $p$ corresponding to $B_{G}$ and we set $\tilde{U}_{n}, \tilde{V}_{n} \subset \tilde{G}(\Zp)$ to be the level groups at $p$ corresponding to $B_{\tilde{G}}$.

We take compatible Shimura data for the pair $(\tilde{G}, \tilde{H})$  to be that defined by the map $\mathrm{Res}_{\C/\R}\mathbb{G}_{m} \to \tilde{H}$ given by
$$
    a + bi \mapsto \frac{1}{a^{2} + b^{2}}\begin{pmatrix}
        a & b \\ -b & a
    \end{pmatrix} \times \frac{1}{a^{2} + b^{2}}\begin{pmatrix}
        a & b \\ -b & a
    \end{pmatrix} \times \{1\}. 
$$
We fix a set of primes $S$ not containing $p$ and auxiliary data $K_{S}, \underline{\phi}_{S}, W_{S}, c_{1}, c_{2}$ as in \cite[\S 8.4.2]{LZGsp} and take the test data away from $S$ to be the spherical test data (for simplicity we're taking the tame level $M = 1$). Let $\Sigma = pS$. The test data determines a prime-to-$p$ level subgroup and the Shimura varieties of this level and Siegel parahoric level at $p$ admit models over $\Z[\Sigma^{-1}]$. The test data along with the branching map $\mathrm{br}^{[r_{1},r_{2},q,r]}$
allows us to define classes 
$$
    {_{c_{1}, c_{2}}}z^{[r_{1},r_{2},q,r]}_{0} =\iota_{*} \circ \mathrm{br}^{[r_{1}, r_{2}, q, r]}\left( {_{c_{1}, c_{2}}}\mathrm{Eis}(\phi) \right) \in H^{4}(Y_{G}(G(\Zp))_{\Sigma}, V_{\Zp}(r_{1}, r_{2}, -(r_{1} + r_{2}))(-q))
$$
for $0 \leq q \leq r_{2}, 0 \leq r \leq r_{1} - r_{2}$.

The monoid $T^{-} := A^{-}  = \{t \in T(\Qp): t^{-1}N_{G}(\Zp)t \subset N_{G}(\Zp)\}$ is generated by the elements 
$$
 \tau_{S}^{-1} = \begin{psmallmatrix}
                    p^{-1} & & & \\ & p^{-1} & & \\
                    & & 1 & \\ & & & 1
                \end{psmallmatrix}, \tau_{\mathcal{K}}^{-1} = \begin{psmallmatrix}
                    p^{-2} & & & \\ & p^{-1} & & \\
                    & & p^{-1} & \\ & & & 1
                \end{psmallmatrix}
$$
associated to the standard Siegel and Klingen parabolic subgroups (the maximal parabolic subgroups of $G$ containing $B_{G}$) respectively. These define double coset operators $U'_{S}$ and $U'_{\mathcal{K}}$ respectively. We write $U'_{B} = U'_{S}U'_{\mathcal{K}}$, the double coset operator associated to the Borel subroup $B_{G}$. 

Let $\Pi$ be a non-endoscopic cohomological cuspidal automorphic representation $\Pi$ of $G(\A)$ with coefficients in $L$ whose infinitesimal character agrees with that of an algebraic representation of weight $(r_{1}, r_{2})$ with $r_{1} \geq r_{2} \geq 0$, and suppose $\Pi$ is unramified at $p$ with Satake parameters $\alpha, \beta, \gamma, \delta$ which (after choosing an isomorphism $\C \cong \overline{\Q}_{p}$) are ordered by increasing valuation. Let $\mathfrak{m}_{\Pi}$ denote the maximal ideal of $\mathbb{T}_{S, p}^{-} \otimes L$ associated to $\Pi$ and assume that $\mathfrak{m}_{\Pi}$ is non-critical in the sense of Definition \ref{def:noncrit}. Our assumption that $\Pi$ is non-endoscopic ensures that $ H^{i}(Y_{G}(J_{G}), V^{G}_{\lambda})_{\mathfrak{m}_{\Pi}} = 0$ for $i \neq 3$.

For $\Pi$ as above work of Taylor \cite{taylor1989galois} and Weissauer \cite{weissauer} associates a four-dimensional Galois representation $W_{\Pi}$ to $\Pi$ by taking the $\Pi$ eigenspace in parabolic \'etale cohomology at infinite level c.f. \cite[Theorem 10.2.2]{LZGsp}. Under our assumptions on $\Pi$ this eigenspace is identified with the $\Pi$ eigenspace in $\varprojlim_{K}H^{3}(Y_{G}(K), V^{G}_{\lambda}) \otimes_{\Qp}\overline{\Q}_{p}$, where the limit is taken over open-compact subgroups of $G(\A_{f})$. We can then define classes 
$$
    {_{c_{1}, c_{2}}}z_{\Pi, n}^{[q, r]} := \mathrm{AJ}^{\mathrm{cl}}_{\mathfrak{m}_{\Pi}}\left({_{c_{1}, c_{2}}}z^{[r_{1},r_{2},q,r]}_{0}\right) \in H^{1}(\Q_{n}, W_{\Pi}(-q))
$$
for each $0 \leq q \leq r_{2}, 0 \leq r \leq r_{1} - r_{2}$, where $\mathrm{AJ}^{\mathrm{cl}}_{\mathfrak{m}_{\Pi}}$ is the Abel--Jacobi map defined at the beginning of Section \ref{sec:abeljac}. We note that (c.f. Remarks \ref{rem:rn}) our assumption that $\Pi$ is non-endoscopic is sufficient for the construction of $\mathrm{AJ}^{\mathrm{cl}}_{\mathfrak{m}_{\Pi}}$.

\begin{remark}
    This construction technically depends on a `modular parameterisation': a choice of functional $\Pi_{f}^{*} \to \overline{\Q}_{p}$ c.f. \cite[\S 10.4]{LZGsp}, which we have suppressed from the construction. This choice can be avoided by using the paramodular newform theory of Brooks--Roberts \cite{localnewf}. 
\end{remark}
\subsection{Lemma--Flach classes in families} \label{sec:lemfam}
In order to construct classes interpolating the Lemma--Flach classes we need an input 
$$
    z_{\tilde{H}} \in H^{2}_{\mathrm{Iw}}(Y_{\tilde{H}}(Q_{\tilde{H}}^{0})_{\Sigma}, \Zp(2)).
$$
We construct this using the Eisenstein--Iwasawa classes of Definition \ref{def:eisiw}. Let $c_{1}, c_{2}$ be as in the previous section, then by taking the cup product of two Eisenstein--Iwasawa classes we obtain a class 
$$
    {_{c_{1}}}\mathcal{EI} \cup {_{c_{2}}}\mathcal{EI} \in H^{2}_{\mathrm{Iw}}(Y_{\mathrm{GL}_{2}}(K_{\infty})_{\Sigma}^{2}, \Zp(2)). 
$$
There is a natural map 
\begin{align*}
    Q_{\tilde{H}}^{0} &\to K_{\infty} \times K_{\infty} \\
    \begin{psmallmatrix}
        x & * \\ & 1
    \end{psmallmatrix} \times \begin{psmallmatrix}
        xy & *  \\ & y^{-1}
    \end{psmallmatrix}  \times (y^{-1}) &\mapsto \begin{psmallmatrix}
        x & * \\ & 1
    \end{psmallmatrix} \times \begin{psmallmatrix}
        x & * \\ & 1
    \end{psmallmatrix} 
\end{align*}
which induces a pullback map 
$$
    H^{2}_{\mathrm{Iw}}(Y_{\mathrm{GL}_{2}}(K_{\infty})^{2}, \Zp(2)) \to H^{2}(Y_{\tilde{H}}(Q_{\tilde{H}}^{0}), \Zp(2))
$$
and we define 
$$
    {_{c_{1}, c_{2}}}\mathcal{EI}_{\tilde{H}} \in H^{2}_{\mathrm{Iw}}(Y_{\tilde{H}}(Q_{\tilde{H}}^{0}), \Zp(2)) 
$$
to be the pullback of the class ${_{c_{1}}}\mathcal{EI} \cup {_{c_{2}}}\mathcal{EI}$.
\begin{definition}
For $n \geq 1$ we write ${_{c_{1}, c_{2}}}z^{\tilde{H}}_{n} \in H^{2}_{\mathrm{Iw}}(Y_{\tilde{H}}(Q_{\tilde{H}}^{0} \cap u\tilde{U}_{n}u^{-1})_{\Sigma}, \Zp(2))$ for the pullback of ${_{c_{1}, c_{2}}}\mathcal{EI}^{H}_{n}$to level $Q_{\tilde{H}}^{0} \cap u\tilde{U}_{n}u^{-1}$. 
\end{definition}
Taking ${_{c_{1}, c_{2}}}z^{\tilde{H}}_{n}$ as the input of Theorem \ref{thm:normrelation} we obtain classes
$$
    {_{c_{1}, c_{2}}}\xi_{\mathcal{U}, n}^{[q, r]} \in H^{4}(Y_{\tilde{G}}(\tilde{V}_{n})_{\Sigma}, \mathscr{A}^{\mathrm{Iw}}_{\mathcal{U}, m} \otimes (\mu/\sigma)^{q} \otimes \sigma^{r})
$$
satisfying the norm relation 
$$
    (\mathrm{pr}^{n + 1}_{n})_{*}\left( {_{c_{1}, c_{2}}}\xi^{[q,r]}_{\mathcal{U}, n + 1}\right) = U'_{B}{_{c_{1}, c_{2}}}\xi^{[q,r]}_{\mathcal{U}, n},
$$
where $(\mathrm{pr}^{n + 1}_{n})_{*}: H^{4}(Y_{\tilde{G}}(\tilde{V}_{n + 1}), \mathscr{A}^{\mathrm{Iw}}_{\mathcal{U}, m} \otimes (\mu/\sigma)^{q} \otimes \sigma^{r}) \to H^{4}(Y_{\tilde{G}}(\tilde{V}_{n}), \mathscr{A}^{\mathrm{Iw}}_{\mathcal{U}, m} \otimes (\mu/\sigma)^{q} \otimes \sigma^{r})$ is the natural pushforward map induced by the (finite index) inclusion $\tilde{V}_{n + 1} \hookrightarrow \tilde{V}_{n}$. 
\begin{lemma} \label{lem:analytic}
For any $q, r, h \geq 0$ and $h_{1}, h_{2}$ satisfying $h_{1} + h_{2} = h$ the classes $ {_{c_{1}, c_{2}}}\xi_{\mathcal{U},n}^{[q, r]}$ satisfy 
$$
    \sum_{j_{1} \leq h_{1}, j_{2} \leq h_{2}}(-1)^{j_{1} + j_{2}}\binom{h_{1}}{j_{1}}\binom{h_{2}}{j_{2}}(\mathrm{pr}_{n}^{hn})^{*}( {_{c_{1}, c_{2}}}\xi_{\mathcal{U},n}^{[j_{1}, j_{2}]}) \cup e_{(\mu/\sigma)^{-j_{1}}\sigma^{-j_{2}}, hn} \equiv 0 \ \mathrm{mod} \ p^{hn},
$$
where $e_{(\mu/\sigma)^{-j_{1}}\sigma^{-j_{2}}, hn}$ is a basis element for $H^{0}(Y_{\tilde{G}}(\tilde{V}_{hn}), \Z/p^{hn}\Z\otimes (\mu/\sigma)^{-j_{1}}\sigma^{-j_{2}})$.  
\end{lemma}
\begin{proof}
The characters $\mu/\sigma$ and $\sigma$ are trivial upon restriction to $Q_{G}^{0} = \{1\}$. The result then follows from Corollary \ref{lem:cong}. 

\end{proof}
As in \cite[Section 4.6]{loefflerspherical} we push forward $  {_{c_{1}, c_{2}}}\xi_{\mathcal{U}, n}^{[q, r]}$ along the inclusion
$$
    \tilde{V}_{n} \subset J_{G} \times C_{\tilde{H}, n}
$$
to obtain classes 
$$
    {_{c_{1}, c_{2}}}\tilde{z}_{G, n}^{[q,r]} \in H^{4}(Y_{G}(J_{G} \times C_{\tilde{H},n})_{\Sigma}, \mathscr{A}^{\mathrm{Iw}}_{\mathcal{U}, m} \otimes (\mu/\sigma)^{q} \otimes \sigma^{r})
$$
satisfying the same norm compatibility as $_{c_{1}, c_{2}}\xi_{\mathcal{U}, n}^{[q, r]}$ as $n$ varies. Let 
$$
    \Delta_{n} = \mathcal{C}_{\tilde{H}}(\Q)\backslash \mathcal{C}_{\tilde{H}}(\A)/C_{\tilde{H},n} \cdot \mathcal{C}^{p} \cdot \mathcal{C}_{H}(\R)^{\dagger}.
$$
The pair of characters $(\mu/\sigma, \sigma)$ induce an isomorphism $\Delta_{n} \cong \left((\Z/p^{n}\Z)^{\times}\right)^{2}$. Taking the limit over $n$ gives an isomorphism 
$$
    (\mu/\sigma, \sigma): \Delta_{\infty} \cong \left(\Zp^{\times}\right)^{2}.
$$
Write $\Delta_{\infty}^{0}$ for the preimage of $(1 + p\Zp^{2})^{2}$ under the above isomorphism, so that $(\mu/\sigma, \sigma): \Delta^{0}_{\infty} \cong (1 + p\Zp^{2})^{2} \cong \Zp^{2}$ determines a chart around the identity of the abelian $p$-adic Lie group $\Delta_{\infty}$. Deligne's reciprocity law gives a Galois character 
$$
    \kappa_{n}: \mathrm{Gal}(\overline{\Q}/\Q)^{\mathrm{ab}} \to \Delta_{n}
$$
such that $\mu/\sigma \circ \kappa_{n}$ is the inverse of the mod $p^{n}$ cyclotomic character. This allows us to identify
$$
    \Delta_{n} = \Gamma/\Gamma_{n} \times \Delta^{(\sigma)}_{n},
$$
where $\sigma$ gives an isomorphism $\Delta^{(\sigma)}_{n} \cong (\Z/p^{n}\Z)^{\times}$. In the limit this becomes
$$
    \Delta_{\infty} = \Gamma \times \Delta_{\infty}^{(\sigma)}
$$
with $\Delta^{(\sigma)}_{\infty}$ an abelian $p$-adic Lie group of dimension $1$ with local coordinate $\sigma$.

Suppose now that $\Pi$ admits a $p$-stabilisation (a choice of $U'_{B}$ eigenspace on $\Pi_{p}^{J_{G}}$) defining a really nice point $x_{\Pi} \in \mathcal{E}_{G}$ as in Definition \ref{def:reallynice}. There is thus a wide-open disc $\mathcal{U} \subset \mathcal{W}_{G}$ defined over $L$ and a Coleman family $\underline{\Pi} \subset \mathcal{E}_{G}$ fibred over $\mathcal{U}$ and passing through $x_{\Pi}$.  In order to construct the overconvergent Abel--Jacobi map in this context we need the following result which follows from work of of Yang--Zhu \cite{yang2025generic}:
\begin{theorem}
Suppose $p > 3$ and the representation $\Pi$ has generic $L$-parameter in the sense of \cite[Definition 1.1]{yang2025generic}. Let $\mathfrak{m}_{\underline{\Pi}} = \mathfrak{m}_{S}\cdot\mathfrak{m}_{p}$ be the kernel of the map 
$$
\mathbb{T}_{S, p}^{-}\hat{\otimes}\Lambda_{\mathcal{U}} \to \mathcal{O}(\underline{\Pi})^{\circ}/\mathfrak{m}_{\mathcal{U}}
$$
with $\mathfrak{m}_{S} \in \mathbb{T}_{S} \hat{\otimes} \Lambda_{\mathcal{U}}$ and $\mathfrak{m}_{p} \subset \mathfrak{U}_{p}^{-} \hat{\otimes} \Lambda_{\mathcal{U}}$ and suppose that $\mathfrak{m}_{S}$ does not occur in the mod $p$ cohomology of the boundary of the Borel--Serre compactification of $Y_{G}(J_{G})(\C)$. 
Then the cohomology group $\overline{H}^{4}(\overline{Y}_{G}(J_{G}), \mathscr{A}^{\mathrm{Iw}}_{\mathcal{U}, m})_{\mathfrak{m}_{S}} = 0$. In particular, it satisfies \eqref{ass:bex}.
\end{theorem}
\begin{proof}
    This follows (via Lemma \ref{lem:torvan}) from \cite[Theorem]{yang2025generic} where it's proved that $H^{i}(Y_{G}(K), \overline{\mathbb{F}}_{p})_{\mathfrak{m}_{\ell}} = 0$ (resp. $H^{i}_{c}(Y_{G}(K), \overline{\mathbb{F}}_{p})_{\mathfrak{m}_{\ell}} = 0$) for $i > q$ (resp. $i < q$), a level $K$ unramified at $\ell$, where $\mathfrak{m}_{\ell}$ is the maximal ideal in the spherical Hecke algebra at $\ell$ corresponding to a representation $\Pi$ with generic $L$-parameter, and $p$ larger than the Coxeter number of $\mathrm{Sp}_{4}$, which is $4$ (the function of the primes denoted $p$ and $\ell$ in \textit{op. cit.} is the opposite of the usage in this paper). The assumption that $\mathfrak{m}_{S}$ does not occur in the mod $p$ cohomology of the Borel-Serre boundary implies that 
    $$
        H^{i}(Y_{G}(K), \overline{\mathbb{F}}_{p})_{\mathfrak{m}_{S}} = H^{i}_{c}(Y_{G}(K), \overline{\mathbb{F}}_{p})_{\mathfrak{m}_{S}}
    $$
    so $ H^{i}(Y_{G}(K), \overline{\mathbb{F}}_{p})_{\mathfrak{m}_{S}} = 0$ for $i \neq q$ which is the condition needed to apply Lemma \ref{lem:torvan}.
\end{proof}

We define $W_{\underline{\Pi}}$ to be the $\underline{\Pi}$ eigenspace in $H^{3}(\overline{Y}_{G}(J_{G}), \mathscr{A}^{\mathrm{Iw}}_{\mathcal{U}, m})$ and let $\hat{W}_{\underline{\Pi}}$ denote its reflexive hull. We write  $\alpha_{B}(\underline{\Pi})$ for $U'_{B}$-eigenvalue of $\underline{\Pi}$ and $\lambda_{\underline{\Pi}}$ for its valuation. Shrinking $\mathcal{U}$ if necessary we can apply the Abel-Jacobi map defined in Definition \ref{def:abeljacobi} associated to $\underline{\Pi}$ to obtain classes 
$$
    {_{c_{1}, c_{2}}}z_{n}^{[\underline{\Pi}, q, r]} \in H^{1}(\Q_{n}, \hat{W}_{\underline{\Pi}}(-q)) \otimes_{L} \mathrm{CoInd}_{\{1\}}^{\Delta_{n}^{(\sigma)}}(L) \otimes \sigma^{r}, 
$$
where the twist by $\sigma^{r}$ is for the natural action of $\mathrm{Gal}(\Q_{n}/\Q) \times \Delta^{(\sigma)}_{n}$, and
the specialisation of ${_{c_{1}, c_{2}}}z_{0}^{[\underline{\Pi}, q, r]}$ at $(r_{1}, r_{2}; c) \in \mathcal{U}$ is $\mathcal{E} \cdot  {_{c_{1}, c_{2}}}z_{\Pi}^{[q, r]}$ where $\mathcal{E}$ an explicit Euler factor which arises when comparing classes at Iwahori and spherical level \textit{c.f.} \cite[Theorem 7.1.1]{loeffler2021spherical}. Let $\mathcal{V} \subset \mathcal{U}$ be an affinoid subset. By restriction we obtain classes 
$$
     {_{c_{1}, c_{2}}}z_{n}^{[\Pi\vert_{\mathcal{V}}, q, r]} \in H^{1}(\Q_{n}, \hat{W}_{\Pi\vert_{\mathcal{V}}}(-q)) \otimes_{L} \mathrm{CoInd}_{\{1\}}^{\Delta_{n}^{(\sigma)}}(L)\otimes \sigma^{r}.
$$
 The Galois representation $\hat{W}_{\Pi\vert_{\mathcal{V}}}$  has a natural $\mathcal{O}(\mathcal{V})^{\circ}$-lattice $\hat{T}_{\Pi\vert_{\mathcal{V}}}$ given by the (reflexive hull of the) image of $H^{3}(\overline{Y}_{G}(J_{G})_{\Sigma}, \mathscr{A}^{\mathrm{Iw}}_{\mathcal{V}, m})$ which defines a norm on $\hat{W}_{\Pi\vert_{\mathcal{V}}}$. We deduce from Lemma \ref{lem:analytic} that for any $0 \leq h_{1} + h_{2} \leq  h$ these classes satisfy 
\begin{equation} \label{eq: norm1}
    \norm{\sum_{j_{1} \leq h_{1}, j_{2} \leq h_{2}}^{h}(-1)^{j_{1} + j_{2}}\binom{h_{1}}{j_{1}}\binom{h_{2}}{j_{2}}{_{c_{1}, c_{2}}}z_{n}^{[\Pi\vert_{\mathcal{V}}, j_{1}, j_{2}]}} \leq Cp^{-nh}
\end{equation}
for a constant $C$ independent of $n$. Let $x_{n}(q, r)$ be the restriction of $\alpha_{B}(\underline{\Pi})^{-n}{_{c_{1}, c_{2}}}z_{n}^{[\Pi\vert_{\mathcal{V}}, q, r]}$ to $H^{1}(\Q_{\infty}, \hat{W}_{\Pi\vert_{\mathcal{V}}})^{\Gamma_{n} = \chi_{\mathrm{cyc}}^{-q}} \otimes_{L} \mathrm{CoInd}_{\{1\}}^{\Delta_{\infty}^{(\sigma)}}(L)^{\Delta^{(\sigma)}_{\infty} = \sigma^{r}}$. We can now apply Theorem \ref{thm:interpo} to interpolate in $q$ and $r$:

\begin{theorem}
There is a unique element 
$$
    {_{c_{1}, c_{2}}}\tilde{z}_{\infty}^{[\Pi_{\mathcal{V}}]} \in H^{1}(\Q_{\infty}, \mathscr{D}_{\lambda}(\Gamma \times \Delta^{(\sigma)}_{\infty}, \hat{W}_{\Pi\vert_{\mathcal{V}}}))^{\Gamma}
$$
satisfying 
$$
    \int_{(\gamma, a) \in \Gamma_{n} \times p^{n}\Delta^{(\sigma)}_{n}}\chi_{\mathrm{cyc}}^{q}(\gamma)\sigma^{-r}(a){_{c_{1}, c_{2}}}\tilde{z}_{\infty}^{[\Pi\vert_{\mathcal{V}}]}  = x_{n}(q, r)(1) \in H^{1}(\Q_{\infty}, \hat{W}_{\Pi\vert_{\mathcal{V}}})^{\Gamma_{n} = \chi_{\mathrm{cyc}}^{-q}}.
$$
\end{theorem}
\begin{lemma} \label{lem:inv0}
We have
$\hat{W}_{\Pi\vert_{\mathcal{V}}}^{G_{\Q_{\infty}}} = \{0\}$.
\end{lemma}
\begin{proof}
We know from \cite[Theorem II]{weissauer} that $H^{0}(\Q^{\mathrm{ab}}, W_{\Pi_{x}}) = 0$ for classical specialisations $x \in \mathcal{V}$ which correspond to non-CAP cuspidal automorphic $G(\A)$-representations of cohomological weight. It suffices to show that such specialisations are dense in $\mathcal{V}$; since $\hat{W}_{\Pi\vert_{\mathcal{V}}}$ is locally free, the intersection of the kernels of specialisation at a dense subset of $\mathcal{V}$ is zero. Our assumptions on $\mathcal{V}$ imply that $\mathcal{V}$ is isomorphic to its image under the weight map, which implies that classical points are dense by the classicality theorem since its easy to show that it contains arbitrarily large weights. Since classical weights $(r_{1}, r_{2}; c)$ are cohomological for sufficiently large values of $r_{1}, r_{2}$ we deduce that cohomological classical weights are dense. Finally, our assumptions imply that $H^{i}(\overline{Y}_{G}(J_{G}), V_{w(x)}^{G})_{\mathfrak{m}_{x}} = 0$ for $i \neq 3$ for all classical $x \in \mathcal{V}$ and thus since only non-CAP representations contribute solely to the middle degree then non-CAP cohomological classical points are dense in $\mathcal{V}$.
\end{proof} 
\begin{corollary}
There is a class 
$$
    {_{c_{1}, c_{2}}}z_{\infty}^{[\Pi\vert_{\mathcal{V}}]} \in H^{1}(\Q, \mathscr{D}_{\lambda}(\Gamma \times \Delta^{(\sigma)}_{\infty}, \hat{W}_{\Pi\vert_{\mathcal{V}}}))
$$
satisfying 
$$
     \int_{(\gamma, a) \in \Gamma_{n} \times p^{n}\Delta^{(\sigma)}_{n}}\chi_{\mathrm{cyc}}^{q}(\gamma)\sigma^{-r}(a){_{c_{1}, c_{2}}}z_{\infty}^{[\Pi_{\mathcal{V}}]}  = x_{n}(q, r)(1) \in H^{1}(\Q_{n}, \hat{W}_{\Pi\vert_{\mathcal{V}}}(-q)).
$$
\end{corollary}
\begin{proof}
Let $M = \mathscr{D}_{\lambda}(\Gamma \times \Delta^{(\sigma)}_{\infty}, \hat{W}_{\Pi\vert_{\mathcal{V}}})$. By \cite[Proposition 2.1.6]{recoleman} for each $n \geq 0$ there are exact sequences in continuous Galois cohomology
$$
    0 \to H^{1}(\Gamma_{n}, M^{G_{\Q_{n}}}) \to H^{1}(\Q, M) \to H^{1}(\Q_{n}, M)^{\Gamma_{n}} \to H^{2}(\Gamma_{n}, M^{G_{\Q_{n}}}).
$$
$$
    0 \to H^{1}(\Gamma, M^{G_{\Q_{\infty}}}) \to H^{1}(\Q, M) \to H^{1}(\Q_{\infty}, M)^{\Gamma} \to H^{2}(\Gamma, M^{G_{\Q_{\infty}}}).
$$
Then Lemma \ref{lem:inv0} gives isomorphisms
\begin{align*}
    H^{1}(\Q, M(-q)) \cong H^{1}(\Q_{n}, M)^{\Gamma_{n} = \chi_{\mathrm{cyc}}^{-q}} \\
    H^{1}(\Q, M) \cong H^{1}(\Q_{\infty}, M)^{\Gamma}
\end{align*}
allowing one to uniquely lift $ {_{c_{1}, c_{2}}} \tilde{z}^{[\Pi\vert_{\mathcal{V}}]}_{\infty}$ to an element of $H^{1}(\Q, M)$ satisfying the desired interpolation property. 
\end{proof}
\begin{remark}
In \cite{loeffler2020bloch} Loeffler--Zerbes prove an `explicit reciprocity law' relating the Lemma--Flach classes constructed in \cite{LZGsp} to the value of a two variable $p$-adic $L$-function interpolating a product of $L$-values for $\Pi$ as above. The above result is consonant with this, showing that the class ${_{c_{1}, c_{2}}}z_{\infty}^{[\Pi\vert_{\mathcal{V}}]}$ is in some sense a `two variable cohomology class'. We expect that the image of this class under a suitable regulator map will recover a three variable $p$-adic $L$-function living in $\mathcal{O}(\mathcal{V}) \hat{\otimes} \mathscr{D}_{\lambda}(\Gamma \times \Delta_{\infty}^{(\sigma)}, L)$ interpolating the values of the two variable $p$-adic $L$-function occurring in the work of Loeffler--Zerbes as $\Pi$ varies in a Coleman family. This will be explored in future work. 
\end{remark}
\begin{appendices}

\section{Multivariable $p$-adic Fourier theory and distribution valued Galois cohomology} \label{app:A}

We record some results concerning the theory of tempered distributions on abelian $p$-adic Lie groups. We follow quite closely the case of $\Zp$ as described in \cite{colmez}. 

Let $L/\Qp$ be a complete extension and let $\Delta$ be an abelian $p$-adic Lie group of dimension $d$. Then $\Delta$ has an open-compact subgroup $\Delta^{0}$ which is isomorphic to $\Zp^{d}$. We fix a coordinate $\underline{x} = (x_{1}, \ldots, x_{d}): \Delta^{0} \cong \Zp^{d}$ and for a function $f: \Delta \to L$ we write $f_{\underline{x}} = f\vert_{\Delta^{0}} \circ \underline{x}^{-1}: \Zp^{d} \to L$. 

\begin{definition}
Let $r \in \R_{\geq 0}$. We say that a function $f: \Delta \to L$ is of \textit{class $r$} if there are functions $f^{\underline{j}}: \Delta \to L$ for each $\underline{j} = (j_{1}, \ldots, j_{d})$ with $\sum_{i = 0}^{d} j_{i} = r$ such that if we define $\mathcal{E}_{f}: \Delta \times \Delta^{0} \to L$ by
$$
    \mathcal{E}_{f, r}(\delta_{1}, \delta_{2}) = f(\delta_{1} + \delta_{2}) - \sum_{0 \leq j_{1} + \ldots + j_{d} \leq \lfloor r \rfloor}f^{(\underline{j})}(\delta_{1})\frac{x_{1}(\delta_{2})^{j_{1}}\cdots x_{d}(\delta_{2})^{j_{d}}}{j_{1}!\cdots j_{d}!}
$$
then the quantity 
$$
\mathrm{inf}_{\delta_{1} \in \Delta, \delta_{2} \in p^{n}\Delta^{0}}\left( v_{p}\left(\mathcal{E}_{f, r}(\delta_{1}, \delta_{2})\right)\right) - rn
$$ tends to $\infty$ as $n \to \infty$. We write $\mathscr{C}^{r}(\Delta, L)$ for the space of class $r$ functions. 
\end{definition}
While the functions $f^{(\underline{j})}$ will depend on the choice of coordinate, the property of being class $r$ does not. The space $\mathscr{C}^{r}(\Delta, L)$ has a natural valuation 
$$
    v_{\mathscr{C}_{r}}\left(f\right) = \mathrm{inf}\{\mathrm{inf}_{\delta \in \Delta}(v_{p}(f(\delta)), \mathrm{inf}_{\delta_{1} \in \Delta, \delta_{2} \in p^{n}\Delta^{0}}\left( v_{p}\left(\mathcal{E}_{f, r}(\delta_{1}, \delta_{2})\right)\right) - rn\}.
$$
If for $f \in \mathscr{C}^{r}(\Delta, L)$ we define a sequence of functions 
$$
    f_{m}(\delta) =  \sum_{(j_{1}, \ldots, j_{d}) = 0}^{p^{m} - 1}f(j_{1}, \ldots, j_{d})\mathbbm{1}_{(j_{1}, \ldots, j_{d}) + p^{m}\Delta^{0}}\left(\sum_{k = 0}^{\lfloor r \rfloor }f_{\underline{x}}^{(\underline{j})}(\underline{j})\frac{(x_{1}(\delta) - j_{1})^{k_{1}}\ldots(x_{d}(\delta) - j_{d})^{k_{2}}}{j_{1}!\ldots j_{d}!}\right)
$$
then $f_{m} \to f$ in the topology given by the valuation $v_{\mathscr{C}^{r}}$ and we conclude that the space $LP^{[0, r]}(\Delta, L)$ of locally polynomial functions on $\Delta$ of degree $\leq r$ is dense in $\mathscr{C}^{r}(\Delta, L)$. 

\begin{definition}
We define the space of \textit{$r$-tempered distributions} $\mathscr{D}_{r}(\Delta, L)$ to be the continuous dual of $\mathscr{C}^{r}(\Delta, L)$. 
\end{definition}

For any $h \geq 0$ we write $\mathscr{D}^{[0, h]}(\Delta, L)$ for the dual of the space $\mathrm{LP}^{[0, h]}(\Delta, L)$ of locally polynomial functions on $\Delta$ of degree $\leq h$.
\begin{proposition} \label{prop:riemannsum}
For $h \geq \lfloor r \rfloor$ the natural restriction map 
$$
    \iota: \mathscr{D}_{r}(\Delta, L) \to \mathscr{D}^{[0, h]}(\Delta, L)
$$
is an isomorphism onto its image. A locally polynomial distribution $\mu \in \mathscr{D}^{[0,h]}(\Delta, L)$ is in the image of $\iota$ if and only if for all $n \in \Z_{\geq 0}, \delta \in \Delta$ and  $\underline{k} = (k_{1}, \ldots, k_{d}) \in \Z_{\geq 0}^{d}$  satisfying $k_{1} + \ldots +k_{d} \leq h$ there is a constant $C$ such that 
$$
    v_{p}\left(\int_{\delta + p^{n}\Delta^{0}}\left(\frac{x_{1} - x_{1}(\delta)}{p^{n}}\right)^{k_{1}}\ldots \left(\frac{x_{d} - x_{d}(\delta)}{p^{n}}\right)^{k_{d}}\mu\right) \geq C - rn.
$$
\end{proposition}
\begin{proof}
The proof is very similar to the proof of \cite[Théorème]{colmez} in the $\Delta = \Zp$ case so we just give a sketch and we assume $\Delta = \Zp^{d}$ to simplify notation. 

Consider the space $\mathscr{D}_{r}^{[0, h]}(\Delta, L)$ of $\mu \in \mathscr{D}^{[0, h]}(\Delta, L)$ for whom the quantity 
$$
    v_{\mathscr{D}_{r}, h}(\mu) := \mathrm{inf}_{\substack{\delta \in \Delta \\ 0 \leq k_{1} + \ldots + k_{d} \leq h \\ n \geq 0}}\left(v_{p}\left(\int_{\delta + p^{n}\Delta^{0}}\left(\frac{x_{1} - \delta}{p^{n}}\right)^{k_{1}}\ldots \left(\frac{x_{d} - \delta}{p^{n}}\right)^{k_{d}}\mu\right) + rn\right)
$$
is finite. We claim that $\iota$ gives an isomorphism $\mathscr{D}_{r}(\Delta, L)\cong \mathscr{D}^{[0, h]}_{r}(\Delta, L)$. We skip the proof that $\iota$ sends $\mathscr{D}_{r}(\Delta, L)$ into $\mathscr{D}_{r}^{[0, h]}(\Delta, L)$. Injectivity of $\iota$ follows from the fact that $\mathrm{LP}^{[0, h]}(\Delta, L)$ is dense in $\mathscr{C}^{r}(\Delta, L)$. It is left to prove surjectivity. We have for $\underline{i} \in \Z^{d}_{\geq 0}$ and $\underline{k} \in \Z^{d}_{\geq 0}$ satisfying $k_{1} + \ldots + k_{d} \leq \lfloor r \rfloor$ a Banach basis
$$
    e_{\underline{i}, \underline{k}, r} = p^{\lfloor(\ell(i_{1}) + \ldots + \ell(i_{d})) r\rfloor} \mathbbm{1}_{i_{1} + p^{\ell(i_{1})}\Delta \times \cdots \times i_{d} + p^{\ell(i_{d})}\Delta}\left(\frac{x_{1} - i_{1}}{p^{\ell(i_{1})}}\right)^{k_{1}}\cdots \left(\frac{x_{d} - i_{d}}{p^{\ell(i_{1})}}\right)^{k_{d}}
$$
of $\mathscr{C}^{r}(\Delta, L)$, where $\ell(i) = \mathrm{inf}\{n: i < p^{n}\}$. Let $\mu \in \mathscr{D}^{[0, h]}_{r}(\Delta, L)$ and define $\tilde{\mu} \in \mathscr{D}_{r}(\Delta, L)$ by 
$$
    \int_{\Delta}e_{\underline{i}, \underline{k}, r}\tilde{\mu} = \int_{\Delta}e_{\underline{i}, \underline{k}, r}\mu. 
$$
Writing $\lambda = \iota(\tilde{\mu}) - \mu$ we can write $\int_{\delta + p^{m}\Delta}x_{1}^{k_{1}}\cdots x_{d}^{k_{d}}\lambda$ as 
$$
    \sum_{(i_{1}, \ldots, i_{d}) = 0}^{p^{m} - 1}\sum_{(j_{1}, \ldots, j_{d}) = 0}^{(k_{1}, \ldots, k_{d})}\binom{k_{1}}{j_{1}}\cdots \binom{k_{d}}{j_{d}}(\delta_{1} + i_{1}p^{n})^{k_{1} - j_{1}}\ldots(\delta_{d} + i_{d}p^{n})^{k_{d} - j_{d}}\int_{\delta + \underline{i}p^{n} + p^{n + m}\Delta}(x_{1} - \delta_{1})^{j_{1}}\cdots(x_{d} - \delta_{d})^{j_{d}}\lambda
$$
for $m \geq 0$, where $\delta = (\delta_{1}, \ldots, \delta_{2}) \in \Delta = \Zp^{d}$. The integrals with $j_{1} + \ldots + j_{d} \leq r$ vanish and for $j_{1} + \cdots + j_{d} > r$ the valuation of the expression is greater than or equal to $\mathrm{inf}_{r < j_{1} + \cdots + j_{d} \leq h}\left((j_{1} + \ldots + j_{d} - r)(n + m)\right) + v_{\mathscr{D}_{r}, h}(\lambda)$ which tends to $\infty$ as $m \to \infty$ so $\lambda = 0$ and we conclude.
\end{proof}

Now let $A$ be a Noetherian $\Qp$-Banach algebra and let $M$ be a finite free $A$-module. Let $\Delta$ be a $d$-dimensional $p$-adic Lie group. 
\begin{proposition}
Let $\lambda \in \R_{\geq 0}$, then $H^{1}(K_{\infty}, \mathscr{D}_{\lambda}(\Delta, M))$ injects into $H^{1}(K, \mathscr{D}^{[0, h]}(\Delta, M))$ for any $h \geq \lfloor \lambda \rfloor$. An element $\mu \in H^{1}(K_{\infty}, \mathscr{D}^{[0, h]}(\Delta, M))$ is in this image if and only if the sequence 
$$
    p^{-\lfloor \lambda n \rfloor}\mathrm{sup}_{\delta \in \Delta}\norm{\int_{p^{n}\Delta^{0}}\left( \frac{(x_{1}(x) - x_{1}(\delta))^{k_{1}} \cdots (x_{1}(x) - x_{1}(\delta))^{k_{d}}}{p^{n}}\right)\mu(x)} 
$$
is bounded as $n \to \infty$ for any $0 \leq k_{1} + \cdots + k_{d} \leq h$.
\end{proposition}
\begin{proof}
The proof is essentially identical to \cite[Proposition 2.3.2]{recoleman} which is itself adapted from various results in \cite{colmez1998theorie} in the case $\Delta = \Zp^{\times}$. The more general choice of  $\Delta$ means we must use Proposition \ref{prop:riemannsum} in place of \cite[Proposition II.3.2]{colmez} but aside from that the proofs are the same. 
\end{proof}
\end{appendices}
\bibliography{SVIII_revision}{}

\providecommand{\bysame}{\leavevmode\hbox to3em{\hrulefill}\thinspace}
\providecommand{\MR}{\relax\ifhmode\unskip\space\fi MR }
\providecommand{\MRhref}[2]{%
  \href{http://www.ams.org/mathscinet-getitem?mr=#1}{#2}
}
\providecommand{\href}[2]{#2}
\begin{thebibliography}{BLV18}

\bibitem[AS08]{Ashsteve}
Avner Ash and Glenn Stevens, \emph{p-adic deformations of arithmetic cohomology}, preprint (2008).

\bibitem[BDJ21]{barrera2021p}
Daniel Barrera, Mladen Dimitrov, and Andrei Jorza, \emph{$ p $-adic {$ L $}-functions of {H}ilbert cusp forms and the trivial zero conjecture}, Journal of the European Mathematical Society (2021).

\bibitem[Bei86]{beilinson1986higher}
AA~Beilinson, \emph{Higher regulators of modular curves}, Contemporary Mathematics (1986), 1--34.

\bibitem[Bel]{bellaicheeigenbook}
Jo{\"e}l Bella{\"\i}che, \emph{The eigenbook}, Pathways in Mathematics. Birkhauser-Springer. To appear.

\bibitem[BH97]{balister1997note}
Paul~N Balister and Susan Howson, \emph{Note on {N}akayama’s lemma for compact $\lambda$-modules}, Asian J. Math \textbf{1} (1997), no.~2, 224--229.

\bibitem[BLV18]{buyukboduk2018iwasawa}
K{\^a}z{\i}m B{\"u}y{\"u}kboduk, Antonio Lei, and Guhan Venkat, \emph{Iwasawa theory for symmetric square of non-$ p $-ordinary eigenforms}, arXiv preprint arXiv:1807.11517 (2018).

\bibitem[BP21]{boxer2021higher}
George Boxer and Vincent Pilloni, \emph{Higher {C}oleman theory}, arXiv preprint arXiv:2110.10251 (2021).

\bibitem[Col98]{colmez1998theorie}
Pierre Colmez, \emph{Th{\'e}orie d'iwasawa des repr{\'e}sentations de de rham d'un corps local}, Annals of Mathematics (1998), 485--571.

\bibitem[Col10]{colmez}
\bysame, \emph{Fonctions d’une variable p-adique}, Ast{\'e}risque \textbf{330} (2010), 13--59.

\bibitem[CS19]{carianischolze}
Ana Caraiani and Peter Scholze, \emph{On the generic part of the cohomology of non-compact unitary shimura varieties}, arXiv preprint arXiv:1909.01898 (2019).

\bibitem[CT23]{caraiani2023etale}
Ana Caraiani and Matteo Tamiozzo, \emph{On the {\'e}tale cohomology of hilbert modular varieties with torsion coefficients}, Compositio Mathematica \textbf{159} (2023), no.~11, 2279--2325.

\bibitem[GS20]{greenberg2020triple}
Matthew Greenberg and Marco~Adamo Seveso, \emph{Triple product p-adic {L}-functions for balanced weights}, Mathematische Annalen \textbf{376} (2020), no.~1, 103--176.

\bibitem[Han17]{hansen2017universal}
David Hansen, \emph{Universal eigenvarieties, trianguline {G}alois representations, and p-adic {L}anglands functoriality (with an appendix by {J}ames {N}ewton)}, Journal f{\"u}r die reine und angewandte Mathematik \textbf{2017} (2017), no.~730, 1--64.

\bibitem[HL23]{hamann2023torsion}
Linus Hamann and Si~Ying Lee, \emph{Torsion vanishing for some shimura varieties}, arXiv preprint arXiv:2309.08705 (2023).

\bibitem[Kat04]{kato}
Kazuya Kato, \emph{p-adic hodge theory and values of zeta functions of modular forms}, Ast{\'e}risque \textbf{295} (2004), 117--290.

\bibitem[Kin15]{kings2015eisenstein}
Guido Kings, \emph{Eisenstein classes, elliptic soul{\'e} elements and the l-adic elliptic polylogarithm, the bloch kato conjecture for the riemann zeta function (john coates, anantharam raghuram, anupam saikia, and ramdorai sujatha, eds.)}, London Math. Soc. Lecture Note Ser \textbf{418} (2015).

\bibitem[KLZ17]{KLZ}
Guido Kings, David Loeffler, and Sarah~Livia Zerbes, \emph{Rankin--eisenstein classes and explicit reciprocity laws}, Cambridge Journal of Mathematics \textbf{5} (2017), no.~1, 1--122.

\bibitem[LLZ14]{lei2014euler}
Antonio Lei, David Loeffler, and Sarah~Livia Zerbes, \emph{Euler systems for rankin--selberg convolutions of modular forms}, Annals of mathematics (2014), 653--771.

\bibitem[Loe21]{loefflerspherical}
David Loeffler, \emph{Spherical varieties and norm relations in {I}wasawa theory}, Journal de Th{\'e}orie des Nombres de Bordeaux \textbf{33} (2021), no.~3, 1021--1043.

\bibitem[Lov17]{lovering2017integral}
Tom Lovering, \emph{Integral canonical models for automorphic vector bundles of abelian type}, Algebra \& Number Theory \textbf{11} (2017), no.~8, 1837--1890.

\bibitem[LRZ21]{loeffler2021spherical}
David Loeffler, Rob Rockwood, and Sarah~Livia Zerbes, \emph{Spherical varieties and p-adic families of cohomology classes}, arXiv preprint arXiv:2106.16082 (2021).

\bibitem[LSZ21]{LZGsp}
David Loeffler, Christopher Skinner, and Sarah~Livia Zerbes, \emph{Euler systems for $\mathrm{GSp (4)}$}, Journal of the European Mathematical Society (2021).

\bibitem[LZ16]{recoleman}
David Loeffler and Sarah~Livia Zerbes, \emph{Rankin-{E}isenstein classes in {C}oleman families}, Research in the Mathematical Sciences \textbf{3} (2016), no.~1, 1--53.

\bibitem[LZ20]{loeffler2020bloch}
\bysame, \emph{On the {B}loch-{K}ato conjecture for $\mathrm{GSp}(4)$}, arXiv preprint arXiv:2003.05960 (2020).

\bibitem[RS07]{localnewf}
Brooks Roberts and Ralf Schmidt, \emph{Local newforms for \textit{GSp}(4)}, vol. 1918, Springer Science \& Business Media, 2007.

\bibitem[SW21]{salazar2021parabolic}
Daniel~Barrera Salazar and Chris Williams, \emph{Parabolic eigenvarieties via overconvergent cohomology}, Mathematische Zeitschrift (2021), 1--35.

\bibitem[Tay89]{taylor1989galois}
Richard Taylor, \emph{On galois representations associated to hilbert modular forms}, Inventiones mathematicae \textbf{98} (1989), no.~2, 265--280.

\bibitem[Urb11]{urbaneigen}
Eric Urban, \emph{Eigenvarieties for reductive groups}, Annals of mathematics (2011), 1685--1784.

\bibitem[Wei05]{weissauer}
Rainer Weissauer, \emph{Four dimensional {G}alois representations}, Ast{\'e}risque \textbf{302} (2005), 67--150.

\bibitem[YZ25]{yang2025generic}
Xiangqian Yang and Xinwen Zhu, \emph{On the generic part of the cohomology of shimura varieties of abelian type}, arXiv preprint arXiv:2505.04329 (2025).

\end{thebibliography}
\bibliographystyle{amsalpha}
\end{document}